\documentclass[10pt,a4paper,openany,oneside]{article}

\usepackage{amsfonts,ulem}
\usepackage{fancyhdr}
\usepackage{verbatim}
\usepackage{graphicx,color}
\usepackage{authblk}
\usepackage{mathscinet}

\usepackage{times}
\usepackage[latin1]{inputenc}
\usepackage[english]{babel}
\usepackage{mathrsfs}
\usepackage{stmaryrd}
\usepackage{amssymb,amsmath,amsthm}
\usepackage{hyperref}
\usepackage{graphicx}
\usepackage{subcaption}
\usepackage{eurosym}

\usepackage{amssymb,amsfonts}
\usepackage{xcolor}

\numberwithin{equation}{section}

\newtheorem{theorem}{Theorem}[section]
\newtheorem*{theorem*}{Theorem}
\newtheorem{thm}[theorem]{Theorem}
\newtheorem{mainthm}{Theorem}
\newtheorem{lemma}[theorem]{Lemma}
\newtheorem{lem}[theorem]{Lemma}
\newtheorem{proposition}[theorem]{Proposition}

\newtheorem{corollary}[theorem]{Corollary}

\newtheorem{definition}[theorem]{Definition}

\newtheorem{example}[theorem]{Example}
\newtheorem{remark}[theorem]{Remark}

\newtheorem{defn}{Definition}[section]
\newcommand{\be}{\begin{equation}}
	\newcommand{\ee}{\end{equation}}
\newcommand{\beq}{\begin{equation*}}
	\newcommand{\eeq}{\end{equation*}}
	\newcommand{\enq}{\end{equation}}
\newcommand{\ben}{\begin{eqnarray}}
	\newcommand{\een}{\end{eqnarray}}
\newcommand{\bea}{\begin{eqnarray*}}
	\newcommand{\eea}{\end{eqnarray*}}

\newcommand{\Fc}{ {\mathcal{F}}}

\def\cH{{\mathcal H}}
\def\cE{{\mathcal E}}

\newcommand{\cD}{ {\mathcal{D}}}
\newcommand{\cL}{ {\mathcal{L}}}
\newcommand{\cP}{ {\mathcal{P}}}

\newcommand{\cR}{ {\mathcal{R}}}
\newcommand{\cM}{ {\mathcal{M}}}

\def\ker{{\mathrm{ker\,}}}
\def\Ran{{\mathrm{Ran\,}}}

\newcommand{\R}{{\mathbb{R}}}
\newcommand{\N}{{\mathbb{N}}}
\newcommand{\C}{{\mathbb{C}}}

\newcommand{\supp}{\mbox{\rm supp}}

\def\ri{{\rm i}}

 \def\dd{\, {\rm d}}
\newcommand{\pa}{\partial}
\DeclareMathOperator*{\stwo}{\sigma_{e,2}}
\DeclareMathOperator*{\sone}{\sigma_{e,1}}
\DeclareMathOperator*{\sthree}{\sigma_{e,3}}

\DeclareMathOperator*{\sfour}{\sigma_{e,4}}
\DeclareMathOperator*{\sfive}{\sigma_{e,5}}
\DeclareMathOperator*{\stwor}{\sigma^{{\rm red}}_{e,2}}
\DeclareMathOperator*{\soner}{\sigma^{{\rm red}}_{e,1}}
\DeclareMathOperator*{\sthreer}{\sigma^{{\rm red}}_{e,3}}
\DeclareMathOperator*{\sfourr}{\sigma^{{\rm red}}_{e,4}}
\DeclareMathOperator*{\sfiver}{\sigma^{{\rm red}}_{e,5}}

\DeclareMathOperator*{\coker}{coker}
\DeclareMathOperator\Hdiv{H_{div}}
\DeclareMathOperator\Hcurl{H_{curl}}
\DeclareMathOperator\sweyl{\sigma_{Weyl}}
\DeclareMathOperator\sweylr{\sigma^{{\rm red}}_{Weyl}}
\DeclareMathOperator\perm{\boldsymbol{\epsilon}}

\newcommand{\bspm}{\left(\begin{smallmatrix}}\newcommand{\espm}{\end{smallmatrix}\right)}
\newcommand{\bpm}{\begin{pmatrix}}\newcommand{\epm}{\end{pmatrix}}
\def\blem{\begin{lem}}\def\elem{\end{lem}}
\def\bthm{\begin{thm}}\def\ethm{\end{thm}}
\def\bprop{\begin{proposition}}\def\eprop{\end{proposition}}
\def\bcor{\begin{corollary}}\def\ecor{\end{corollary}}
\def\beq{\begin{equation}}\def\eeq{\end{equation}}
\def\beqq{\begin{equation*}}\def\eeqq{\end{equation*}}
\def\bal{\begin{align}}\def\eal{\end{align}}
\def\bpf{\begin{proof}}\def\epf{\end{proof}}
\def\bex{\begin{example}}\def\eex{\end{example}}
\def\brem{\begin{remark}}\def\erem{\end{remark}}

\newcommand{\norm}[1]{\left\Vert#1\right\Vert}

\newcommand{\curl}[1]{\nabla \times #1 }
\renewcommand{\div}[1]{\nabla \cdot #1 }

\newcommand{\pd}[1]{ {\partial_{x_{#1}}}}
\newcommand{\nn}{\nonumber}
\newcommand{\ba}{\begin{aligned}}
	\newcommand{\ea}{\end{aligned}}

\newcommand\bluesout{\bgroup\markoverwith{\textcolor{blue}{\rule[0.5ex]{2pt}{0.4pt}}}\ULon}

\def\C{\mathbb C}
\def\R{\mathbb R}
\def\N{\mathbb N}

\makeatletter
\def\@fnsymbol#1{\ensuremath{\ifcase#1\or \dagger\or \ddagger\or
		\mathsection\or \mathparagraph\or \|\or **\or \dagger\dagger
		\or \ddagger\ddagger \else\@ctrerr\fi}}
\makeatother

\newcommand{\Addresses}{{% additional braces for segregating \footnotesize
		\bigskip
		\footnotesize

		B.M.~Brown, \textsc{School of Computer Science, Abacws, Cardiff University, Senghennydd Road, Cardiff CF24 4AG, UK}%\par\nopagebreak
		\medskip

		T.~Dohnal, \textsc{Martin Luther University Halle-Wittenberg, Institute of Mathematics, Theodor Lieser Str.5, 06120 Halle, Germany}\par\nopagebreak
		\textit{E-mail address}, T.~Dohnal: \texttt{tomas.dohnal@mathematik.uni-halle.de}, ORCID: 0000-0002-0379-5967

		\medskip

		M.~Plum, \textsc{KIT Karlsruhe, Institute for Analysis, Englerstrasse 2, 76131 Karlsruhe, Germany}\par\nopagebreak
		\textit{E-mail address}, M.~Plum: \texttt{michael.plum@kit.edu}, ORCID: 0000-0001-6998-128X

		\medskip

		I.~Wood, \textsc{School of Mathematics, Statistics and Actuarial Sciences, Sibson Building, University of Kent, Canterbury, CT2 7FS, UK}\par\nopagebreak
		\textit{E-mail address}, I.~Wood: \texttt{i.wood@kent.ac.uk}, ORCID: 0000-0001-7181-7075
}}

\begin{document}
	\title {Spectrum of the Maxwell Equations for a Flat Interface between Homogeneous Dispersive Media}

	\author{Malcolm Brown\thanks{Malcolm Brown died on 14th January 2022.}, Tom\'a\v{s} Dohnal, Michael Plum, and Ian Wood}

	\maketitle
	\Addresses

	\bigskip

	\textbf{Statements and Declarations:}
	Tomas Dohnal acknowledges funding by the Deutsche Forschungsgemeinschaft (DFG, German Research Foundation), Project-ID DO1467/4-1. Michael Plum acknowledges funding by the DFG, Project-ID 258734477 - SFB 1173. For the purpose of open access, the authors have applied a CC BY public copyright licence (where permitted by UKRI, an Open Government Licence or CC BY ND public copyright licence may be used instead) to any Author Accepted Manuscript version arising.

	\medskip

	\abstract{The paper determines and classifies the spectrum of a non-self-adjoint operator pencil generated by the time-harmonic Maxwell problem with a nonlinear dependence on the frequency for the case of two homogeneous materials joined at a planar interface. We study spatially one-dimensional and two-dimensional reductions in the whole space $\R$ and $\R^2$. The dependence on the spectral parameter, i.e. the frequency, is in the dielectric function and we make no assumptions on its form. These function values determine the spectral sets. In order to allow also for non-conservative media, the dielectric function is allowed to be complex, yielding a non-self-adjoint problem. The whole spectrum consists of eigenvalues and the essential spectrum, but the various standard types of essential spectra do not coincide in all cases. The main tool for determining the essential spectra are Weyl sequences.}


	\section{Introduction to the problem}

	In this work  we  analyze  the spectral  properties of a time-harmonic  Maxwell pencil  with the frequency $\omega$ being the spectral parameter. We consider a composite material which  incorporates   an interface dividing  the space into two half-spaces with different material properties.
	\newline\indent
	When one of the materials is dispersive, i.e.~the dielectric constant depends on $\omega$, then the spectral problem is nonlinear  in $\omega$.  Also, when  one of the materials is non-conservative, i.e.~with a complex valued dielectric constant, for example, a metal, then the problem is non-self-adjoint.
	An interesting phenomenon in such structures is the existence of surface plasmons. These are  states localized at an interface of a metal and a dielectric and therefore, in the one-dimensional case, correspond to the existence of eigenvalues.  The simplest example is a planar interface (which without loss of generality we take to be at $x_1=0$). Here plasmons are transverse magnetic solutions which are localized in $x_1$ and  have a plane wave form in the tangential direction. In this paper  we consider both a two-dimensional problem, where the fields depend only on $(x_1,x_2)$, and a one-dimensional problem, where the fields are independent of $x_3$ and the dependence on $x_2$ is prescribed to produce a plane wave. We formulate the Maxwell spectral problem (in the second order formulation for the $E$ field) via the operator pencil $P(\omega;1)$, where
	$$P(\omega;\lambda)=A(\omega)-\lambda B(\omega),$$
	and where $A$ and $B$ are operators between Hilbert spaces. Informally, $A$ is the curl-curl operator and $B$ is the operator of multiplication by the dielectric function. As we explain below, in order to be able to define the concept of isolated eigenvalues, we introduce the extra parameter $\lambda$ in $P(\omega;\lambda)$.
	\newline\indent
	The topic of surface plasmons is well studied in the physics literature, see e.g. \cite{Zayats05,Pitarke_2007}. There are also rigorous mathematical results on eigenvalues and surface plasmons as corresponding eigenfunctions \cite{Grieser14,ARYZ16}  - with the former reference studying the quasistatic approximation and both considering a problem linear in the spectral parameter. Our aim is a description of the entire spectrum in the dispersive case (i.e.~nonlinear dependence on the spectral parameter $\omega$). Unlike the above papers we restrict ourselves here to the simplest, i.e.~planar, interface. The main contribution of the paper is the characterization and classification of the spectrum, which includes the non-trivial task of the definition of isolated eigenvalues in the nonlinear setting.
	\newline\indent
	In the existing literature there are several results on the essential spectrum of the Maxwell problem with or without interfaces. In \cite{Lassas98} the author studies  the case of a bounded domain in the presence of conductivity, which makes the problem non-self-adjoint. The resulting operator pencil is  however linear. Also \cite{ABMW19} studies the non-self-adjoint Maxwell case with a linear dependence on the spectral parameter but in the presence of interfaces. The case of a nonlinear pencil in the Maxwell problem for a cavity is studied in \cite{HP20}, where the spectral problem is a self-adjoint one. The Maxwell problem with a rational pencil (as given by the Drude model) is studied in \cite{CHJ2017,CHJ2021}, where a self-adjoint case with electric and magnetic currents is considered, the spectrum is determined and the behaviour of the resolvent and resonances studied.
	\newline\indent
	Pencil problems arise naturally in many other application areas, e.g., hydrodynamics \cite{Shkaliko96}. In addition,  \cite{ST96} studies a linear pencil problem on a finite interval (a Kamke problem) and its application to the Orr-Sommerfeld problem.  In \cite{MT13} a linear non-self-adjoint pencil with an application to the Hagen-Poiseuille flow is analyzed. Most mathematical literature on pencils is restricted to a linear or polynomial dependence on the spectral parameter, see e.g.~\cite{Markus1988}. However, also a rational dependence has been analyzed \cite{ET18,Engstr21}. We do not make any assumptions on the form of the dependence on the spectral parameter.

	As mentioned above, the material in our problem depends only on the spatial variable $x_1$ and the frequency $\omega$. Hence, the (linear) electric susceptibility in frequency space is
	$$\hat\chi:\R\times (\C\setminus S)\to\C, \hat\chi: (x_1,\omega)\mapsto \hat\chi(x_1,\omega),$$
	where $S$ is a set of (typically finitely many) singularity points of $\hat{\chi}(x_1,\cdot)$.
	In the case of a dielectric in $x_1>0$ and a metal in $x_1<0$ the following is a simple example
	\begin{equation}\label{E:chi}
		\hat{\chi}(x_1,\omega):=
		\begin{cases}
			\hat{\chi}_m(\omega) & \text{for } x_1<0, \\
			\eta & \text{for } x_1>0,
		\end{cases}
	\end{equation}
	where $\eta>0$. Two classical examples of $\hat{\chi}_m$ are the Drude model
		\beq\label{E:Drude}
		\hat{\chi}_m(\omega)=-\frac{c_D}{\omega^2+\ri\gamma\omega}
		\eeq
		with $ c_D,\gamma>0$ and the Lorentz model
		\beq\label{E:Lorentz}
		\hat{\chi}_m(\omega)=-\frac{c_L}{\omega^2+\ri\gamma\omega-\omega_*^2}
		\eeq
		with $c_L,\gamma,\omega_*>0$, see e.g. \cite{Pitarke_2007,ACL2018,CJM2023}.

	The time-harmonic ansatz $(\mathcal E, \mathcal H)(x,t)=(E,H)(x)e^{-\ri\omega t}+(\overline{E},\overline{H})(x)e^{\ri\overline{\omega} t}, \omega \in\C,$ for the electric field $\mathcal E$ and the magnetic field $\mathcal H$ reduces time dependent Maxwell's equations in $\R^3$ in the absence of free charges and free currents and with a constant permeability to
	\begin{equation}\label{E:Maxw-1stord}
		\begin{aligned}
			\nabla\times E & = \ri\omega \boldsymbol{\mu}_0 H, \\
			\nabla\times H & = -\ri\omega D, \quad D =\perm_0 (1+\hat\chi(x_1,\omega))E,\\
			\nabla \cdot D & =0, \quad \nabla \cdot H =0,
		\end{aligned}
	\end{equation}
	where $\perm_0>0$ and $\boldsymbol{\mu}_0>0$ are the electric permittivity and magnetic permeability of vacuum, respectively. Here we choose the permeability of the material equal to that of vacuum. Note that the divergence condition has an impact on the functional analytic setting and hence also on the spectral properties. On the other hand, for eigenfunctions with $\omega\neq 0$ the $H$- and $D$-components automatically satisfy $\nabla \cdot H=0=\nabla \cdot D$. In order to allow also $\omega=0$, the divergence conditions need to be imposed. While from a physical point of view the case $\omega=0$ may be irrelevant, we include it for mathematical completeness. In order to reflect the assumptions in the model, our choice of the operator domain includes the condition $\nabla \cdot D=0$. The condition  $\nabla \cdot H=0$ plays no role in our analysis as we reduce the system to a second order problem for the $E$ field.

		Also note that allowing the material permeability $\boldsymbol{\mu}$ (replacing $\boldsymbol{\mu}_0$) to depend on $\omega$ would have no effect on the analysis of the problem. However, allowing  $\boldsymbol{\mu}$ to depend on $x_1$ would change the analysis and the functional analytic setting considerably, see Remarks \ref{R:setting-mu-1d}, \ref{R:setting-mu-2d}, and \ref{R:evals-mu-2d}. In the rest of the papers we set, for the sake of brevity,
		$$\boldsymbol{\mu}_0=1.$$
		This can be achieved by rescaling $H$ and the variable $x$. This simplification has no effect on the results.

	In the second order formulation we have
	\begin{equation}\label{E:Maxw-2ndord}
		\nabla\times\nabla\times E- \omega^2\perm(x_1,\omega)E=0, \quad \nabla \cdot (\perm E)=0,
	\end{equation}
	where
	$$\perm(x_1,\omega):=\perm_0(1+\hat\chi(x_1,\omega)).$$

	We denote the $\omega$-domain of $\perm(x_1,\cdot)$ by $D(\perm)$, i.e.
	$$D(\perm):=\C\setminus S.$$
	Note that although the domain of $\omega \mapsto \omega^2 \perm(x_1,\omega^2)$ can be larger than $D(\perm)$ (since it is possible that $\omega^2 \perm(x_1,\omega^2)$ is defined at $\omega=0$ while $0\notin D(\perm)$), we have to work in the potentially smaller domain of $\perm(x_1,\cdot)$ as $\perm$ occurs in the condition $\nabla \cdot D =0$.

	We make no assumptions on the properties of $\omega \mapsto \perm(x_1,\omega)$. Of course, from a physical point of view, such a degree of generality may not be necessary but no assumptions on the form of the $\omega-$dependence are needed for the mathematical spectral analysis. However, for physically relevant materials, $\perm$ should satisfy a set of properties determined by realness, causality, and passivity. Physically, these are defined in the time domain. In the frequency domain they translate to the conditions
		$$
		\begin{aligned}
			&\perm(\cdot,-\overline{\omega}) = \overline{\perm(\cdot,\omega)} \quad \forall \omega \in \C^+:=\{\omega\in \C: \Im(\omega)>0\},\\
			&\omega\mapsto \perm(x_1,\omega) \ \text{is holomorphic in } \C_a^+:=\{\omega\in \C: \Im(\omega)>a\} \ \text{for some } a>0,\\
			& \omega\mapsto \perm(x_1,\omega) \ \text{is a Herglotz function,}
		\end{aligned}
		$$
		respectively. Herglotz functions are complex valued functions, holomorphic in $\C^+$ (implying causality already) and with values in $\{\omega\in \C: \Im(\omega)\geq 0\}$. In the time domain realness means that $\chi_1(x_1,\omega)$ is real, causality is the condition that the displacement field $\cD(t)$ depends only on the past, i.e. on $\mathcal E(s)$ with $s\leq t$, and passivity means that the $L^2$-energy of the field $(\mathcal E,\cH)$ stays below its initial value. In addition, one expects that the dispersiveness of the medium decays for ``large" frequencies, i.e. $\perm(\cdot,\omega)\to\perm_0$  for $|\omega|\to \infty$. For a discussion of the connections between these mathematical and physical conditions see \cite{CJK2017,CJM2023}.

	In the one-dimensional reduction of \eqref{E:Maxw-2ndord} we set $E(x)=e^{\ri kx_2}u(x_1)$ with $k\in \R$ fixed and $u(\cdot)$ being a suitable function of one  real variable, see Sec.~\ref{S:1D}. In the two-dimensional reduction we set $E(x)=E(x_1,x_2)$, see Sec.~\ref{S:2D}.

	We assume that the material is homogeneous in each half-space $\R^n_\pm:=\{x\in \R^n: \pm x_1>0\}$. Hence, the function $\perm (\cdot,\omega)$ is piecewise constant with a possible jump only at $x_1=0$.

	We define also the $x_1$-independent functions
	$$\perm_\pm:D(\perm)\to\C, \perm_\pm(\omega):=\perm(\pm x_1>0,\omega).$$

	Due to the discontinuity of $\perm$, solutions of system \eqref{E:Maxw-1stord} are not smooth at $x_1=0$. Nevertheless, formally, they satisfy the condition that $E_2, E_3, D_1,$ and $H$ be continuous across the interface at $x_1=0$, see e.g.~Sec.~33-3 in \cite{feynman2}. In one dimension (where the components of the functions in the operator domain lie in $H^1(\R_+)$ as well as $H^1(\R_-)$) this continuity  is in the classical sense, while in two dimensions it has to be interpreted in the trace sense, see \eqref{E:IFC} and Appendix \ref{App:traces2D}.
	Thus, in one dimension we have
	\beq\label{E:jump}
	\llbracket \psi\rrbracket :=\psi(x_1\to 0+)-\psi(x_1\to 0-)=0
	\eeq
	for $\psi=E_2, E_3, D_1, H_1, H_2,$ and $H_3$. We deduce these jump conditions from our functional analytic setting for both reductions to  spatial dimensions $n=1$ and $n=2$.

	The aim of this work is to describe the spectrum of \eqref{E:Maxw-2ndord} for the interface problem, where $\omega$ plays the role of a spectral parameter. Due to the generally nonlinear dependence of $\perm$ on $\omega$, we have to    model the problem  in terms of  an operator pencil.

	We define first the exceptional set
	$$\Omega_0:=\{\omega \in D(\perm): \omega^2\perm_+(\omega)=0 \ \text{or} \ \omega^2\perm_-(\omega)=0\}.$$
	In most physical applications $\Omega_0=\{0\}$ because typically $\perm(\omega)_\pm\neq 0$ for all $\omega \in\C$ and $\lim_{\omega\to 0}\omega^2\perm(\omega)=0$. However, from the mathematical point of view there is no need to make such an assumption, and we shall avoid it.

	In the \textbf{one-dimensional reduction} with
	\beq\label{E:ansatz1D}
	E(x)=e^{\ri kx_2} u(x_1),
	\eeq
	where $k\in \R$ is fixed, and $u:\R\to\C^3$, the equation for the profile $u$ becomes
	\beq\label{E:NL-eval}
	L_k(\omega)u:=T_k(\pa_{x_1})u-\omega^2 \perm(x_1,\omega)u=0, \ x_1 \in \R\setminus \{0\},
	\eeq
	where
	\beq\label{E:Tk}
	T_k(\pa_{x_1}):=\bspm k^2 & \ri k\pa_{x_1}& 0 \\ \ri k\pa_{x_1} & -\pa_{x_1}^2 & 0 \\ 0 & 0 & k^2-\pa_{x_1}^2\espm=\nabla_k\times\nabla_k\times, \quad \nabla_k:=(\pa_{x_1},\ri k,0)^T.
	\eeq

	In the \textbf{two-dimensional reduction} with
	\beq\label{E:ansatz2D}
	E(x)=E(x_1,x_2),
	\eeq
	where $E:\R^2\to\C^3$, the equation for $E$ is now
	\beq\label{E:NL-eval-2D}
	L(\omega)E:=T(\pa_{x_1},\pa_{x_2})E-\omega^2 \perm(x_1,\omega)E=0, \ (x_1,x_2) \in \R^2,
	\eeq
	where
	\beq\label{E:T}
	T(\pa_{x_1},\pa_{x_2}):=\bspm -\pa_{x_2}^2  & \pa_{x_1}\pa_{x_2}& 0\\ \pa_{x_1}\pa_{x_2} & -\pa_{x_1}^2 & 0 \\ 0 & 0 & -\pa_{x_1}^2-\pa_{x_2}^2\espm= \nabla \times\nabla \times, \quad \nabla :=(\pa_{x_1},\pa_{x_2},0)^T.
	\eeq

	Due to the block structure in \eqref{E:Tk} and  and \eqref{E:T} all fields $E$ that are relevant for our spectral considerations are of a polarized form, i.e. $E=(E_1,E_2,0)$ or $E=(0,0,E_3)$. The corresponding $H$ field can be constructed using $\curl E = \ri \omega H$. The former case then corresponds to the TM polarization with $H=(0,0,H_3)$ and in the latter case one has $H=(H_1,H_2,0)$.

	The rest of the paper is organized as follows. In Section \ref{S:gen-pencil} we define the various types of spectrum for operator pencils $P(\omega;\lambda)=A(\omega)-\lambda B(\omega)$. The main results of the paper are summarized in Section \ref{S:main-res}. The results in the one-dimensional setting are proved in Section \ref{S:1D} while the two-dimensional case is analyzed in Section \ref{S:2D}. In Section \ref{S:t-dep}  we discuss implications for the time dependent problem and in Section \ref{S:concl} we provide a summary and discuss open problems. Finally, the appendices explain our choice of interface conditions in the functional analytic $L^2$-based setting.

	Note that throughout the paper we use the definition of the complex square root $\sqrt{z}$, $z\in \C$  as the unique solution $a$ of $a^2=z$ with $\text{arg}(a)\in (-\pi/2,\pi/2]$.

	%---------------------------------------------------------------------------
	\section{Spectrum of Operator Pencils}\label{S:gen-pencil}

	Much of this section is a straightforward generalisation of well-known concepts for the spectrum of operator pencils, see for example \cite{MoellerPivo}, and Banach space operators, see \cite{EE18}. However, most operator pencils found in the literature depend polynomially on the spectral parameter $\omega$. We have already seen in \eqref{E:Drude}, \eqref{E:Lorentz} that this is not the case in physically relevant models here, so we  do not wish to make this restriction. Therefore,
	we start by providing a general framework for the spectral analysis of operator pencils relevant for our application.
	The aim is to study the spectrum of an $\omega$-dependent operator pencil $A(\omega)- B(\omega)$, where
		$$
		\begin{aligned}
			&A(\omega): H \supset \cD_\omega\to \cR \ \text{is closed},\\
			&B(\omega): H \supset \cD_\omega \to \cR \ \text{is bounded}
		\end{aligned}
		$$
		with some Hilbert spaces $H$ and $\cR$ and where the spectral parameter $\omega$ lies in a set $\Omega \subset \C$. The index $\omega$ in $\cD_\omega$ means that the domain of $A(\omega)$ and $B(\omega)$ can depend on $\omega$.

	However, as we explain below, to study  we need to consider also the 2-parameter pencil
	$$\cP :=(P(\omega;\lambda))_{\substack{\omega\in \Omega\\ \lambda\in \C}},\; \ P(\omega;\lambda)=A(\omega)-\lambda B(\omega).$$

	The spectral parameter in our analysis is $\omega$ with $\lambda$ fixed at $\lambda=1$ for most of our discussions. However, when we define the essential spectrum $\sfive$ (see Definition \ref{def1}) and the concept of an isolated eigenvalue, we need to consider $\lambda$ as an auxiliary spectral parameter, see \eqref{E:sp-def-1D}. To the best of our knowledge, the introduction of the second auxiliary spectral parameter is new.

	The \textbf{resolvent set}  of  the pencil $\cP$ is defined as
	$$\rho(\cP):=\{\omega\in \Omega: P(\omega;1):\cD_\omega\to \cR \text{ is bijective with a bounded inverse}\}.$$
	\brem\label{R:equiv-est}
	The boundedness in the definition of $\rho(\cP)$ can be equivalently formulated as
	\beq\label{E:bded-1}
	\exists c_1>0: \ \|u\|_H \leq c_1\|P(\omega;1)u\|_\cR \ \forall u \in \cD_\omega
	\eeq
	and
	\beq\label{E:bded-2}
	\exists c_2>0: \ \|u\|_{\cD_\omega} \leq c_2\|P(\omega;1)u\|_\cR \ \forall u \in \cD_\omega,
	\eeq
	where
	\beq\label{E:graph-norm}
	\|u\|_{\cD_\omega}:= \sqrt{\langle u,u\rangle_{\cD_\omega}}, \ \langle u,v\rangle_{\cD_\omega}:=\langle u,v\rangle_H + \langle A(\omega)u,A(\omega)v\rangle_{\cR}
	\eeq
	denotes the graph norm corresponding to $A(\omega)$. Note that due to the closedness of $A(\omega)$ one easily obtains that $(\cD_\omega,\langle\cdot,\cdot\rangle_{\cD_\omega})$ is a Hilbert space. Another trivial observation is that $A:\cD_\omega\to\cR$ is bounded due to the definition of the norm in $\cD_\omega$.

	The implication \eqref{E:bded-2} $\Rightarrow$ \eqref{E:bded-1} is trivial. For the reverse direction note that
	$$
	\begin{aligned}
		\|A(\omega)u\|_{\cR}&=\|P(\omega;1)u+B(\omega)u\|_{\cR}\leq \|P(\omega;1)u\|_{\cR}+\|B(\omega)u\|_{\cR}\\
		& \leq \|P(\omega;1)u\|_{\cR}+c\|u\|_{H}\leq c\|P(\omega;1)u\|_{\cR},
	\end{aligned}
	$$
	where the generic constant $c$ changes in each step and where \eqref{E:bded-1} was used in the last step.
	\erem

	Following \cite{MoellerPivo}, we now introduce and discuss the \textbf{spectrum} of $\cP$ defined by
	\beq\label{E:Pspec}
	\sigma(\cP):=\Omega\setminus \rho(\cP)=\{\omega\in \Omega: 0\in \sigma(P(\omega;1))\},
	\eeq
	where $\sigma(P(\omega;1))$ is the spectrum (defined in the standard sense) of the operator $P(\omega;1)$ at a
	fixed   $\omega$.

	Next we introduce the concept of the \textbf{point spectrum} defined by
	$$\sigma_p(\cP):=\{\omega\in \Omega: \exists u \in \cD_\omega\setminus \{0\}: P(\omega;1)u=0 \}.$$
	Elements of $\sigma_p(\cP)$ are called eigenvalues of $\cP$.

	As  preparation for the definition of eigenvalues of \textbf{finite and infinite algebraic multiplicity} of the pencil $\cP$, we define these properties for the second eigenvalue parameter $\lambda$. The algebraic multiplicity of $\lambda$ as an eigenvalue of $P(\omega;\cdot)$ is called infinite if its geometric multiplicity~$\dim\ker(P(\omega;\lambda))$ is infinite or there exists a sequence $(u_k)_{k\in\N_0}$ of linearly independent elements $u_k\in \cD_{\omega}$ such that $(A(\omega)-\lambda B(\omega))u_{k+1}= B(\omega)u_k$ for all $k\in \N_0$ with $u_0\in \ker(A(\omega)-\lambda B(\omega))\setminus \{0\}$. Otherwise the algebraic multiplicity is called finite.

	The eigenvalue $\lambda$ of $P(\omega;\cdot)$ is called algebraically simple if it is geometrically simple and there is no solution $u\in \cD_{\omega}$ of
	\beq\label{E:gen-evec-eq}
	(A(\omega)-\lambda B(\omega))u= B(\omega)v,
	\eeq
	where $v\in \ker(A(\omega)-\lambda B(\omega))\setminus \{0\}$ and such that $u$ and $v$ are linearly independent.

	Note that we do not define the algebraic multiplicity in general as a number since we do not use it. For such a definition (compatible with the above) see Section 1.1 in \cite{MoellerPivo}.

	For the pencil $\cP$ we subdivide the point spectrum $\sigma_p(\cP)$ as
	\begin{align*}
		&\sigma_p^{(<\infty)}(\cP):=\{\omega\in \sigma_p(\cP): \text{the algebraic multiplicity of } 1 \text{ as an eigenvalue of } P(\omega;\cdot) \text{ is finite}\},\\
		&\sigma_p^{(\infty)}(\cP):=\{\omega\in \sigma_p(\cP): \text{the algebraic multiplicity of } 1 \text{ as an eigenvalue of } P(\omega;\cdot) \text{ is infinite}\},
	\end{align*}
	and we call an eigenvalue $\omega \in \sigma_p(\cP)$ \textbf{algebraically simple} if $\lambda=1$ is an algebraically simple eigenvalue of $P(\omega;\cdot)$.

	A sequence $(u^{(n)})\subset \cD_\omega$ is called a   \textbf{Weyl sequence at $\omega$} if
	$$\|u^{(n)}\|_{H}=1 \ \forall n\in\N, u^{(n)} \rightharpoonup 0 \text{ in } H, \text{ and } \|P(\omega;1)u^{(n)}\|_{\cR}\to 0 \ (n\to\infty).$$

	The \textbf{Weyl spectrum} is
	$$\sweyl(\cP):=\{\omega \in \Omega: \text{ a Weyl sequence at $\omega$ exists}\}.$$

	%We now define various notions of the
	There are several differing notions  of  \textbf{essential spectrum}   and we now introduce them by adapting the corresponding definitions %of the quantities $\sone, \stwo,\sthree,\sfour,\sfive $
	from \cite[Ch.\ I, \S 4]{EE18}  to our present needs.
	\begin{defn}\label{def1}
		The essential spectra $\sone(\cP), \stwo(\cP),\sthree(\cP),\sfour(\cP)$, and $\sfive(\cP)$ are defined as follows.
		\begin{enumerate}
			\item $\omega\in \sone(\cP)$ if $P(\omega;1)$ is not  semi-Fredholm (an operator is semi-Fredholm if its range is closed and its kernel or its cokernel is finite-dimensional);
			\item $\omega\in\stwo(\cP)$ if $P(\omega;1)$ is not in the class of semi-Fredholm operators with finite-dimensional kernel;
			\item $\omega\in \sthree(\cP)$ if $P(\omega;1)$ is not in the class of Fredholm operators with finite-dimensional kernel and cokernel;
			\item $\omega\in \sfour(\cP)$ if $P(\omega;1)$ is not  Fredholm with index zero, where $\operatorname{ind}P(\omega;1)=\dim\ker P(\omega;1)-\dim\coker P(\omega;1)$;
			\item $\omega\in \sfive(\cP)$ if $1\notin\Delta(P(\omega;\cdot))$. Here, $\Delta(P(\omega;\cdot))\subseteq\C$ consists of all connected components of $\{\lambda \in \C: P(\omega,\lambda) \text{ is semi-Fredholm}\}$ that contain a point in the resolvent set of $P(\omega;\cdot)$.
		\end{enumerate}
	\end{defn}

	\brem\label{rem:essinc}
	Note that $\sone\subset \stwo\subset \sthree\subset \sfour\subset \sfive$ holds in general, see \cite[Ch.\ IX]{EE18}. For the case of a self-adjoint operator all the above definitions of the essential spectrum coincide, see \cite[Ch.\ IX, Thm.\ 1.6]{EE18}.
	\erem

	\begin{remark}\label{rem:essential}
		When $H$ and $\cR$ are separable, infinite-dimensional Hilbert spaces,
		the statement $\omega\in \stwo(\cP)$ is equivalent to $\omega\in \sweyl(\cP)$.
		This follows from \cite[Ch.\ IX, Thm.\ 1.3]{EE18}, which covers the case $H_1=H_2$. However, by using a straightforward argument involving the isomorphism between $H_1$ and $H_2$ the result can be correspondingly extended.
	\end{remark}
	The closedness of $\Ran(P(\omega;1))$ plays an important role in describing the essential spectrum. We will make use of the following lemma.
	\blem\label{L:ran_not-closed}
	Assume there exists a Weyl sequence $(u_n)\subset V(\omega):=\cD_\omega\cap \ker(P(\omega;1))^\perp$ with the orthogonal complement understood in $H$. Then $\Ran(P(\omega;1))$ is not closed in $\mathcal{R}$.
	\elem
	\bpf
	To simplify the notation, let us set $\tilde{P}:=P(\omega;1)$.
	$V(\omega)$ is a closed subspace of $\cD_\omega$ because for every sequence $(v_n)\subset V(\omega)$ with $v_n\to v$ in $\cD_\omega$ we also have $v_n\to v$ in $H$ by \eqref{E:graph-norm}. Hence, for each $w\in \ker(\tilde{P})$ we have
	$$0=\langle v_n, w\rangle_H\to \langle v, w\rangle_H,$$
	implying $v\in \ker(\tilde{P})^\perp$ and thus $v\in V(\omega)$. We also conclude that $V(\omega)$ is a Hilbert space (with the inner product of $\cD_\omega$).
	\break
	\indent Next we show that $\tilde{P}|_{V(\omega)}:V(\omega)\to\Ran(\tilde{P}) $ is bijective and bounded. The injectivity follows because $\tilde{P}v=0$ implies that $v\in\ker(\tilde{P}) \cap \ker(\tilde{P})^\perp$ and hence $v=0.$ For the surjectivity pick $r\in \Ran(\tilde{P})$ and $w\in \cD_\omega$ such that $\tilde{P}w=r$. Next, we define $\pi:H\to\ker(\tilde{P})$ as the orthogonal projection. Indeed, $\ker(\tilde{P})$ is closed with respect to the $H$-norm as we show next. Let $(\varphi_n)$ denote a sequence in the kernel which converges to some $\varphi \in
	H$ w.r.t.~the $H$-norm. In particular, $(\varphi_n)$ is a sequence in the domain
	$\cD_\omega$ of $\tilde{P}$ and $\tilde{P}(\varphi_n) = 0$ for all $n$,
	whence the sequence $(\tilde{P}(\varphi_n))$ converges to $0$. Since $\tilde{P}$ is closed (with
	domain $\cD_\omega$ in the surrounding Hilbert space $H$), we get $\varphi \in \cD_\omega$
	and $\tilde{P}(\varphi) = 0$, i.e. $\varphi$ is in the kernel. See also Problem 5.9 in Chapter 3, \S 5 Sec 2 of \cite{Kato}.\\
	\indent
	Let us now set $u:= \pi w$. As $u \in \ker(\tilde{P})\subset \cD_\omega$, we get
	$$v:=w-u\in \cD_\omega \cap \ker(\tilde{P})^\perp =V(\omega)$$
	and $\tilde{P}v = \tilde{P}w-\tilde{P}u=\tilde{P}w=r.$\\
	\indent
	Finally, the boundedness of  $\tilde{P}|_{V(\omega)}:V(\omega)\to\Ran(\tilde{P}) $  is trivial since  $\tilde{P}:\cD_\omega\to\cR$ is bounded.\\
	\indent If $\Ran(\tilde{P})$ was closed, then the inverse  $\left(\tilde{P}|_{V(\omega)}\right)^{-1}:\Ran(\tilde{P})\to V(\omega)$ would be a bounded bijective operator between two Hilbert spaces by the open mapping theorem (more precisely its formulation as the bounded inverse theorem). This contradicts the assumption of existence of a Weyl sequence in $V(\omega)$ because $1=\|u_n\|_H\leq \|u_n\|_{\cD_\omega}\leq c\|\tilde{P}u_n\|_\cR\to 0$.
	\epf

	Finally, we define the \textbf{discrete spectrum}. This turns out to be slightly more complicated  than the procedure we have used in the above spectral definitions. If $\cD_\omega\not\subset \cR$ (as is the case in our applications), then it is not possible
	to define the notion of an isolated  eigenvalue $\omega$ of $\cP$ via the notion of $0$  being isolated as an eigenvalue of $P(\omega;1)$. This is because for $\mu \neq 0$ the operator $P(\omega;1)-\mu I$ does not map $\cD_{\omega}$ into $\mathcal{R}$ as $Iu=u\not\in \mathcal{R}$ in general. Hence, for the discrete spectrum $\sigma_d$ we use the following  definition
	\beq\label{E:sp-def-1D}
	\omega \in \sigma_d(\cP)  \ :\Leftrightarrow \ 1\in \sigma_d(P(\omega;\cdot)),
	\eeq
	i.e. $\lambda=1$ is an isolated eigenvalue of finite algebraic multiplicity of the standard generalized eigenvalue problem $A(\omega)u=\lambda B(\omega)u$ (with $\omega\in \Omega$ fixed).
	This formulation is suitable for defining isolated eigenvalues  because in this case we have $P(\omega;\lambda): \cD_{\omega}\to \mathcal{R}$ for any $\lambda\in \C$.

	\brem\ \label{R:disc-sfive} An important motivation for introducing $\sfive(\cP)$ in this way is that
		it allows the spectrum $\sigma(\cP)$ to be written as a disjoint union $$\sigma(\cP)=\sigma_d(\cP)~ \dot\cup~\sfive(\cP).$$
	This result follows from results by Hundertmark and Lee in \cite{HL07}, as discussed in \cite[Page 460]{EE18}. Note that this useful relation would not be available if we defined the isolatedness of eigenvalues (and hence $\sigma_d$) based on properties of the spectrum in the $\omega$-plane rather than in the $\lambda$-plane.
	\erem

	\brem
	The definition of isolatedness of $\omega_0\in \sigma_d(\cP)$ via the isolatedness of $\lambda=1$ in the spectrum of the linear operator pencil $P(\omega;\cdot)$ is also beneficial for applications in functional analytic arguments. For example, in bifurcation problems for nonlinear equations $(A(\omega) -B(\omega))u=f(u)$ (see e.g. \cite{DR2021,DR2022,DH2024}) the closedness of the range of $L_0:=A(\omega_0)-B(\omega_0)$ is crucial when looking for a solution of $(A(\omega_0)-B(\omega_0))u=b\in \ker(L_0^*)^\perp$. This closedness is guaranteed if $\lambda=1$ is a simple isolated eigenvalue \cite[Theorem IV.5.28]{Kato} and not if $\omega_0$ is isolated from other points $\omega$ in $\sigma(\cP$).

	Note that, with our definition of the discrete spectrum, $\omega_0\in \sigma_d(\cP)$ does not imply  that $\omega_0$ is isolated from other points in $\sigma(\cP)$.
	\erem

	%----------------------------------------------------------------------
	\section{Main Results}\label{S:main-res}
	Before formulating the results rigorously, we give a brief overview and provide some intuition as to why they hold. In the one-dimensional case the interface is expected to cause localization, i.e.~the existence of eigenfunctions. With the material being homogeneous in each half space the fundamental system of the ODE problem in $\R_+$ and $\R_-$ can be found explicitly and the eigenfunctions are determined by forcing the respective decaying solutions to match via the interface conditions. These turn out to reduce to a single equation relating $\perm_+(\omega)$ and $\perm_-(\omega)$, see equation \eqref{E:ev.cond-1D}, which is known in the physics literature and plays a central role in determining eigenvalues. In addition, the problem supports radiation modes (often called non-normalized eigenfunctions in the physics literature), i.e. non-localized solutions which resemble plane-waves far away from the interface and travel to $\pm \infty$. In the spectral analysis these are the building blocks of our Weyl sequences. The corresponding frequencies $\omega$ then constitute the Weyl spectrum.
	\newline\indent
	Because in the two-dimensional case the material is homogeneous in the $x_2$-direction, full localization is not expected. Hence, we should get empty discrete spectrum. Radiation modes come in two forms in 2D, namely plane-wave-like solutions traveling to $x_1\to \infty$ or $x_1\to -\infty$ and solutions which are localized in the $x_1$-direction at the interface and have a plane-wave dependence on $x_2$. The latter ones then propagate along the interface to $x_2\to \pm \infty$ and their existence is again determined by the same relation between $\perm_+(\omega)$ and $\perm_-(\omega)$ as in 1D.
	\newline\indent
	In both 1D and 2D the infinite dimensional kernel of the curl operator could generate eigenvalues of infinite multiplicity if $\omega^2\perm_+(\omega)=0$ or $\omega^2\perm_-(\omega)=0$ with the corresponding eigenfunctions being gradients of smooth functions supported in the respective half-space. However, as we incorporate the divergence condition in the domain of $\cP(\omega)$, only zeros of $\perm_+$ and $\perm_-$ (at which the divergence condition is inactive) can generate such eigenvalues.

	\subsection{Main Results in One Dimension}\label{S:main-res-1D}
	For the problem in one dimension, %(with $E(x)$ defined in   (\ref{E:ansatz1D}))
	we consider for a fixed wavenumber
	%\textcolor{red}{ have we said what this is ?}
	$k\in \R$ the $\omega$- and $\lambda$-dependent operator pencil
	$$\cP=(P(\omega;\lambda))_{\substack{\omega\in \Omega\\ \lambda\in \C}}:=\cL_k:=(L_k(\omega;\lambda))_{\substack{\omega\in D(\perm)\\ \lambda\in \C}},$$
	where
	$$L_k(\omega;\lambda):= A_k(\omega)-\lambda B(\omega), \ A_k(\omega)u:=\nabla_k \times \nabla_k \times u,\  B(\omega)u:=\omega^2\perm(x_1,\omega)u, \ \nabla_k:=(\pa_{x_1},\ri k, 0)^T.$$

	In order to simplify the notation, we also define
	$$L_k(\omega):= L_k(\omega;1).$$

	The choice of the spaces is detailed below and motivated in Sec.~\ref{S:1D}. We choose
	\beq\label{E:Dkomega}
	\begin{aligned}
		H:=&L^2(\R,\C^3),\\
		\cD_{\omega}:=\cD_{k,\omega}:=&\{u\in L^2(\R,\C^3): \ \nabla_k\times u, \nabla_k\times \nabla_k\times u \in L^2(\R,\C^3),  \nabla_k\cdot (\perm(\omega) u) =0\ \text{distributionally}\}
	\end{aligned}
	\eeq
	%\begin{align}
	%\perm_\pm(\omega) \nabla_k\cdot u_\pm &=0 \text{ on } \R_\pm,\label{E:div-cond-1d}\\
	%\llbracket \perm(\omega) u_1\rrbracket = \llbracket u_2\rrbracket=\llbracket u_3\rrbracket =  \llbracket u_2'-\ri k u_1 \rrbracket = \llbracket u_3'\rrbracket &=0,\label{E:IFC-1D}
	%\end{align}
	and
	\beq\label{E:Rk}
	\cR:=\cR_k:= \{f\in L^2(\R,\C^3): \nabla_k\cdot f=f_1'+\ri k f_2 = 0 \ \text{distributionally}\}
	\eeq
	equipped with the $L^2$-inner product.
	The distributional divergence $\nabla_k\cdot f$ is defined in Appendix \ref{App:interf-1D}.

	Note that $B(\omega):H\supset \cD_{k,\omega} \to \cR_k$ is bounded and  $A_k(\omega):H\supset \cD_{k,\omega} \to \cR_k$ is closed since it is bounded as an operator from $\cD_{k,\omega}$ to $\cR_k$ (with the convergence in $\cD_{k,\omega}$ with respect to the graph norm). This argument uses the fact that $\cD_{k,\omega}$ is a Hilbert space, see Lemma \ref{L:HS-1D}. Hence, unlike in the general setting in Sec.~\ref{S:gen-pencil}, we show first the completeness of $\cD_{k,\omega}$ and this then implies the closedness of $A_k(\omega)$.

	%internal remark (proof): Let (u_n)_n \subset \cD_{\omega,k}, u_n\to u in H, and  Au_n\to v in H. To show: u\in \cD_{\omega,k}, Au=v.
	%Due to the definition of the inner product in  \cD_{\omega,k} and due to u_n\to u in H and  Au_n\to v in H, we get that (u_n) is Cauchy in the HS \cD_{\omega,k}. Hence u_n\to u in \cD_{\omega,k}. As A is bded, we get Au_n \to Au, hence v=Au.

	The following equation plays a central role in our analysis
	\beq\label{E:ev.cond-1D}
	k^2(\perm_+(\omega)+\perm_-(\omega))=\omega^2\perm_+(\omega)\perm_-(\omega).
	\eeq
	Before summarising our main results, we introduce some notation. Recall that
	$$\Omega_0=\{\omega \in D(\perm): \omega^2\perm_+(\omega)=0 \ \text{or} \ \omega^2\perm_-(\omega)=0\}.$$
	We next introduce the three sets which are sufficient to describe the spectrum outside $\Omega_0$:
	\beq\label{E:Mkpm}
	\begin{split}
		M^{(k)}_\pm:= &~\{ \omega \in  D(\perm)\setminus\Omega_0: \omega^2\perm_\pm (\omega)\in [k^2,\infty)\}, \ k\neq 0,\\
		M^{(0)}_\pm:= &~\{ \omega \in  D(\perm)\setminus\Omega_0: \omega^2\perm_\pm (\omega)\in (0,\infty)\},
	\end{split}
	\eeq
	\beq
	N^{(k)}:= ~\{ \omega \in  D(\perm)\setminus \Omega_0: \omega^2\perm_+(\omega),\omega^2\perm_-(\omega)\notin [k^2,\infty) \text{ and } \eqref{E:ev.cond-1D} \text{ holds}  \}, k\in \R. \label{E:Nk}
	\eeq
	Note that typically the equations $\omega^2\perm_\pm (\omega)=a$ with $a\geq k^2$ have a discrete set of solutions which depend continuously on $a$. Then the sets $M_\pm^{(k)}$ are continuous curves in $\C$. Clearly, $M_\pm^{(k_1)}\subset M_\pm^{(k_2)}$ for $k_1>k_2$. On the other hand, the set $N^{(k)}$ is typically a discrete set. For the physically relevant case of rational functions $\perm_\pm(\omega)$, see e.g. \eqref{E:Drude}, \eqref{E:Lorentz}, these discrete sets are finite.

	For our first result, we  restrict our attention to the set $D(\perm)\setminus \Omega_0$. We define the reduced sets
	\beq\label{E:red-sets}
	\sigma^{{\rm red}}(\cL_k):=\sigma(\cL_k) \setminus \Omega_0, \quad \rho^{{\rm red}}(\cL_k):=\rho(\cL_k) \setminus \Omega_0,
	\eeq
	and similarly, we define the reduced version of all parts of the spectrum as their intersection with the complement of $\Omega_0$.
	We can now formulate our main result for the reduced spectrum.
	\begin{mainthm}[One dimension; reduced spectrum]\label{T:main1D}
		Let $k\in\R$.\\
		\noindent\textbf{Spectrum}:
		\beq\label{E:sp1D}
		\sigma^{{\rm red}}(\cL_k)= (M^{(k)}_+\cup M^{(k)}_-)\ \dot{\cup}\ N^{(k)},
		\eeq
		where $\dot{\cup}$ denotes the disjoint union.\\
		\medskip
		\noindent\textbf{Point and discrete spectrum}:
		\beq\label{E:pt-sp1D}
		\sigma^{{\rm red}}_p(\cL_k)=\sigma^{(<\infty)\ {\rm red}}_p(\cL_k)= N^{(k)}, \qquad N^{(0)}=\emptyset.
		\eeq
		All eigenvalues in $\sigma^{{\rm red}}_p(\cL_k)$ are algebraically and geometrically simple. Moreover, also
		\beq\label{E:disc-sp1D}
		\sigma^{{\rm red}}_d(\cL_k) = N^{(k)}.
		\eeq
		\noindent\textbf{Essential and Weyl spectrum}:
		\beq\label{E:weyl-sp1D}
		\sigma^{{\rm red}}_{e,j}(\cL_k)=\sweylr(\cL_k)= M^{(k)}_+\cup M^{(k)}_-, \ j=1,\dots,5.
		\eeq
	\end{mainthm}

	\brem
	\begin{enumerate}
		\item[1.] The disjointness in \eqref{E:sp1D} follows directly from the definition of the sets $M_\pm^{(k)}$ and $N^{(k)}$.

		\item[2.] Note that for $\omega \notin \Omega_0$ satisfying \eqref{E:ev.cond-1D}, we have $\perm_+(\omega)+\perm_-(\omega)\neq 0$ and the condition \eqref{E:ev.cond-1D} for $\omega$ to be an eigenvalue is equivalent to
		\beq
		\label{E:om-ev-cond-neces}
		k^2=\omega^2\frac{\perm_+(\omega)\perm_-(\omega)}{\perm_+(\omega)+\perm_-(\omega)},
		\eeq
		which is a well known condition for the existence of a plasmon, see e.g. equation (2.13) in \cite{Pitarke_2007}.
		\item[3.]  As we explain below in Remark \ref{R:Wp-neq-pWm}, there are no finite multiplicity eigenvalues if $\perm_+(\omega)= \perm_-(\omega)$ (i.e.~no plasmons without an interface), because $N^{(k)}$ is empty in this case.
	\end{enumerate}
	\erem

	We can now state the main result for $\omega$ in the exceptional set $\Omega_0$.

	\begin{mainthm}[One dimension; exceptional set] \label{TheoPlum2}
		\ \\
		{\bf a)} Let $k \in \R \setminus \{ 0 \}$.\newline
		\noindent\textbf{Spectrum}:
		\beq\label{E:1Dspec-Om0}
		\sigma  (\cL_k)  \cap \Omega_0 = \left\{ \begin{array}{cl} \Omega_0 \setminus \{ 0 \} & \text{if }~\perm_+ (0) \neq 0, ~  \perm_- (0) \neq 0 , ~  \text{and }~\perm_+ (0) + \perm_- (0) \neq 0, \\
			\Omega_0 & \text{otherwise. } \end{array} \right.
		\eeq
		\noindent\textbf{Point and discrete spectrum}:
		\begin{align}
			\sigma_p^{(< \infty)}  (\cL_k) \cap \Omega_0  &=  \{ \omega \in \Omega_0 : \omega^2\perm_+ (\omega) = \omega^2\perm_- (\omega) =  0,
			\perm_+ (\omega) \neq 0, ~  \perm_- (\omega) \neq 0 , ~ \text{and} ~  \perm_+ (\omega) + \perm_- (\omega) = 0 \} \label{Plum2}\\
			&=  \left\{ \begin{array}{cl} \{ 0 \} & \text{if }  ~ 0\in D(\perm), ~ \perm_+ (0) \neq 0, ~  \perm_- (0) \neq 0 , ~  \text{and } \perm_+ (0) + \perm_- (0) = 0, \\
				\emptyset & \text{otherwise, } \end{array} \right. \notag\\
			\sigma_p^{(\infty)}  (\cL_k) \cap \Omega_0  &=  \{ \omega \in \Omega_0 : \perm_+ (\omega) = 0 \text{ or } \perm_- (\omega) = 0 \}. \label{Plum3}
		\end{align}
		In particular, $ \sigma_p  (\cL_k) \cap \Omega_0=\sigma  (\cL_k)  \cap \Omega_0$.\\
		If $\sigma_p^{(<\infty)}  (\cL_k) \cap \Omega_0$ is non-empty, and thus equals $\{ 0 \}$, then $\omega = 0$ is an algebraically simple eigenvalue, and $\Ran(L_k (0))$ is closed with co-dimension $1$. In particular, $L_k(0)$ is Fredholm with index $0$.

		Moreover,
		\beq\label{E:1Ddiscsp-Om0}
		\sigma_d (\cL_k) \cap \Omega_0 = \emptyset.
		\eeq
		\noindent\textbf{Essential and Weyl spectrum}:
		\begin{align}
			\sigma_{e,2}  (\cL_k) &\cap \Omega_0 = \sigma_{e,3}  (\cL_k) \cap \Omega_0 =
			\sigma_{e,4}  (\cL_k) \cap \Omega_0 =
			\sigma_{\text{Weyl }}  (\cL_k) \cap \Omega_0 = \sigma_p^{(\infty)}  (\cL_k) \cap \Omega_0\  \text{ (see \eqref{Plum3}), } \label{E:1DWeyl-Om0}\\
			\sigma_{e,1}  (\cL_k) &\cap \Omega_0 =  \{ \omega \in \Omega_0 : (\perm_+ (\omega) = 0,~ \omega^2\perm_- (\omega) \in [k^2, \infty ))
			\text{ or }  (\perm_- (\omega) = 0,~ \omega^2\perm_+ (\omega) \in [k^2, \infty )) \}, \label{E:1Des1-Om0} \\
			\sigma_{e,5}  (\cL_k) &\cap \Omega_0 = \sigma  (\cL_k)  \cap \Omega_0 \ (\text{see } \eqref{E:1Dspec-Om0})\label{E:1Des5-Om0}.
		\end{align}
		\\
		\noindent{\bf b)} Let $k = 0$.\\
		\noindent\textbf{Spectrum}:
		\beq\label{E:1Dsp-Om0k0}
		\sigma (\cL_k) \cap \Omega_0 = \Omega_0.
		\eeq
		\noindent\textbf{Point and discrete spectrum}:
		\begin{align}
			&\sigma_p^{(<\infty)} (\cL_k) \cap \Omega_0 = \sigma_d (\cL_k) \cap \Omega_0 = \emptyset,\label{E:1Dptsp-Om0k0}\\
			&\sigma_p^{(\infty)} (\cL_k) \cap \Omega_0 = \{ \omega \in \Omega_0 : \perm_+ (\omega) = 0  \text{ or } \perm_- (\omega) =0 \}.\label{E:1Dptspinf-Om0k0}
		\end{align}
		\noindent\textbf{Essential and Weyl spectrum}:
		\beq\label{E:1Dspec-ess-Om0k0}
		\sigma_{e,j} (\cL_k) \cap \Omega_0 = \sigma_{\text{Weyl }} (\cL_k) \cap \Omega_0 = \Omega_0, \quad j = 1, \dots, 5.
		\eeq
	\end{mainthm}

	\begin{remark} \label{RemPlum1}
		\begin{enumerate}
			\item[a)] Theorem \ref{TheoPlum2} (part a)) shows that, within the $k$-range $\R \setminus \{ 0 \}$, only $\sigma_{e,1} (\cL_k) \cap \Omega_0$ depends (possibly) on $k$, while for all other considered parts of the spectrum, and the resolvent set, their intersection with $\Omega_0$ is independent of $k \in \R \setminus \{ 0 \}$.
			\item[b)] Equations \eqref{E:1DWeyl-Om0} and \eqref{E:1Des1-Om0} of Theorem \ref{TheoPlum2} show that, for $k \in \R \setminus \{ 0 \}$,
			\be\label{Plum6}
			\ba
			(\sigma_{e,2} (\cL_k) \cap \Omega_0) &\setminus (\sigma_{e,1} (\cL_k) \cap \Omega_0) = \\
			& \{ \omega \in \Omega_0 : (\perm_+ (\omega) = 0, ~ \omega^2\perm_- (\omega) \notin [k^2, \infty)) \text{ or } (\perm_- (\omega) = 0, ~ \omega^2\perm_+ (\omega) \notin [k^2, \infty)) \}
			\ea
			\ee
			and
			$$ (\sigma_{e,5} (\cL_k) \cap \Omega_0) \setminus (\sigma_{e,4} (\cL_k) \cap \Omega_0) = \sigma_p^{(<\infty)} (\cL_k) \cap \Omega_0 $$
			can both be non-empty, in which case we have three different essential spectra.
		\end{enumerate}
	\end{remark}
	\brem[Isolatedness]
	All spectral properties of the pencil are determined by the function values $\perm_\pm(\omega)$ and not by the position of $\omega$ in $\C$. Hence, we cannot retrieve any geometric properties of the spectrum viewed as a set in the $\omega$-plane unless specific information on $\omega\mapsto \perm_\pm (\omega)$ is provided.\\
	An extreme case of such geometric properties is when $\omega^2\perm_\pm\equiv c_\pm \in \C\setminus \{0\}$ with $c_\pm$ such that \eqref{E:ev.cond-1D} holds. Then the whole complex $\omega$-plane is the set of isolated eigenvalues (isolated according to our definition, i.e. with the isolatedness in the $\lambda$-plane). If $c_+=0$ or $c_-=0$, then the whole complex $\omega$-plane consists of eigenvalues but these are not isolated,  see \eqref{Plum3} and \eqref{E:NL-eval}.\\
	Another special case is that of non-constant rational functions $\perm_\pm$. Here there is a finite number of solutions of \eqref{E:ev.cond-1D} and hence a finite number of eigenvalues $\omega$. As zeroes of a rational function, they are all isolated in the $\omega$-plane. Note also that, more generally, for non-constant meromorphic functions $\perm_\pm$ the set $N^{(k)}$ is a subset of the set of zeros of the meromorphic function $f(\omega):=k^2(\perm_+(\omega)+\perm_-(\omega))-\omega^2\perm_+(\omega)\perm_-(\omega)$  and hence consists of isolated points. If, in addition, the set $N^{(k)}$ is bounded, then it must be isolated also from $M^{(k)}_+\cup M^{(k)}_- \cup \Omega_0$ since each point in $N^{(k)}$ lies outside $M^{(k)}_+\cup M^{(k)}_- \cup \Omega_0$. But if, e.g., $\omega^2\perm_+(\omega)=\omega^2\perm_-(\omega)=0$ for one such eigenvalue $\omega$, then $\omega$ is not isolated according to our definition.
	\erem
	\brem[Block diagonal structure]\label{R:block1D}
	It is most easily visible in \eqref{E:Tk} that the operator has a block diagonal structure. Also the domain $\cD_{k,\omega}$ is such that the conditions on the first two components are decoupled from those on the third component, which is most apparent in \eqref{E:Dkom}. As a result, the spectrum of $\cL_k$ is the union of the spectra of $\cL_k^{(1,2)}$ and $\cL_k^{(3)}$ in the sense of $\eqref{E:Pspec}$, where
	$$\cL_k^{(1,2)}:=\bspm k^2 & \ri k\pa_{x_1} \\ \ri k\pa_{x_1} & -\pa_{x_1}^2 \espm-\lambda \omega^2\perm(x_1,\omega)$$
	and $\cL_k^{(3)}:=k^2-\pa_{x_1}^2-\lambda \omega^2\perm(x_1,\omega)$ with the domains
	$$
	\begin{aligned}
		D(\cL_k^{(1,2)}) &:= \{ (u_1,u_2)^T:  u \in \cD_{k,\omega}\},\\
		D(\cL_k^{(3)}) &:= \{ u_3:  u \in \cD_{k,\omega}\}.
	\end{aligned}
	$$
	In our analysis we indirectly take advantage of this property. Namely when constructing Weyl sequences, it is often easier to use sequences $(u^{(n)})\subset \cD_{k,\omega}$ with $u^{(n)}_1=u^{(n)}_2=0$ such that only the scalar operator $k^2-\pa_{x_1}^2-\lambda \omega^2\perm(x_1,\omega)$ acting on $u^{(n)}_3$ is activated.
	\erem

	\begin{example}\label{Ex:1D-Drude}
		For illustration of the spectrum we plot in Fig. \ref{F:spec1D} the sets $M^{(k)}_+, M^{(k)}_-, N^{(k)}$, $\Omega_0$, and $S$ for the case of the interface of a Drude material and a dielectric, i.e. (normalizing $\perm_0=1$)
		\beq\label{E:Wt-ex}
		\perm(x_1,\omega)=1 +\hat{\chi}(x_1,\omega)
		\eeq
		with $\hat{\chi}$ given in \eqref{E:chi}, \eqref{E:Drude} and with the parameters
		\beq \label{E:param-ex}
		\eta=1, c_D=2\pi 0.8^2, \gamma=1, k\in \{3,0.7\}.
		\eeq
		Clearly, $S=\{0,-\ri \gamma\}$ and
		$$\Omega_0=\{\omega\in D(\perm): \perm_-(\omega)=0\}=\left\{\tfrac{1}{2}\left(-\ri\gamma \pm \sqrt{4c_D-\gamma^2}\right)\right\}.$$
		Note that the sets $M^{(k)}_+, M^{(k)}_-$,  and $N^{(k)}$ are symmetric about the imaginary axis since $\perm_\pm(-\overline{\omega})=\overline{\perm_\pm(\omega)}$ for all $\omega\in D(\perm)$. Hence, if $\omega^2\perm_\pm(\omega)\in[k^2,\infty)$, then $\omega^2\perm_\pm(-\overline{\omega})\in[k^2,\infty)$ and if $k^2=\frac{\omega^2\perm_+(\omega)\perm_-(\omega)}{\perm_+(\omega)+\perm_-(\omega)}$, then
		$$k^2= \frac{\overline{\omega^2\perm_+(\omega)\perm_-(\omega)}}{\overline{\perm_+(\omega)+\perm_-(\omega)}} = \frac{\omega^2\perm_+(-\overline{\omega})\perm_-(-\overline{\omega})}{\perm_+(-\overline{\omega})+\perm_-(-\overline{\omega})}.$$

		A more detailed calculation shows that
		$$M^{(k)}_-= (P\cap D(\perm)) \setminus \Omega_0,$$
		where
		$$
		\begin{aligned}
			&P=\{\ri s:s\in(-\gamma,0) \text{ and } (s^2+k^2)(s+\gamma)+ c_D s \leq 0\} \\
			&\cup\left\{\ri s\pm\sqrt{-\frac{c_D\gamma}{2s}-(s+\gamma)^2}:s\in \left(-\frac{\gamma}{2},0\right), \frac{c_D\gamma}{2s}+(s+\gamma)^2\leq 0 \text{ and } \frac{c_D}{2s}(2s+\gamma)+(2s+\gamma)^2\leq -k^2 \right\}.
		\end{aligned}
		$$
		The set $M^{(k)}_-$ is unbounded in both horizontal directions and approaches the real axis at infinity. Along the imaginary axis the set $M^{(k)}_-$ equals the set $\{\ri s: s\in(-\gamma,0), (s^2+k^2)(s+\gamma)+c_D s\leq 0\}$. Because $-\ri \gamma \notin M^{(k)}_-$, the set $M^{(k)}_-$ is not closed. As a result, the spectrum of $\cL_k$ is not closed in the $\omega$-plane! This is no contradiction due to our definition of the spectrum in \eqref{E:Pspec}.

		The set $M^{(k)}_+$ is given by
		$$M^{(k)}_+= (-\infty,-\sqrt{k^2/(1+\eta)}] \cup [\sqrt{k^2/(1+\eta)},\infty).$$
		\begin{figure}[h!]
			\centering
			\includegraphics[width=0.45\linewidth]{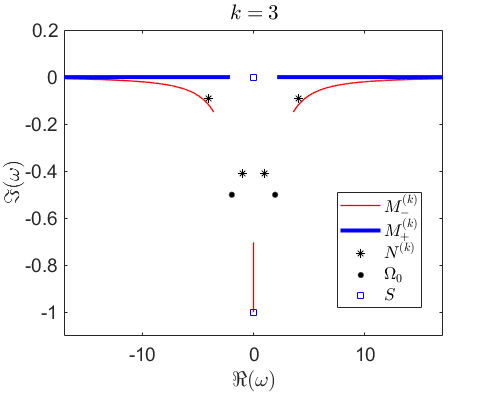}
			\includegraphics[width=0.45\linewidth]{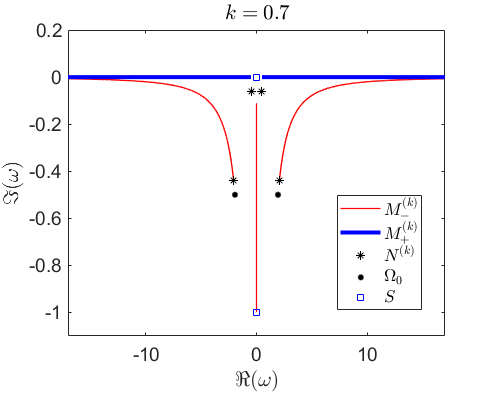}
			\caption{The sets  $M^{(k)}_+, M^{(k)}_-, N^{(k)}$, $\Omega_0$, and $S$ for the 1D Drude setting \eqref{E:chi}, \eqref{E:Drude} with  $\perm_0=1$ and the parameter values in \eqref{E:param-ex}.}
			\label{F:spec1D}
		\end{figure}
		For the set $N^{(k)}$ we get by a simple calculation
		$$N^{(k)}=\left\{\omega \in D(\perm): \omega^4+\ri\gamma\omega^3-\left(c_D+k^2\frac{2+\eta}{1+\eta}\right)\omega^2-\ri\gamma k^2\frac{2+\eta}{1+\eta}\omega +\frac{c_D k^2}{1+\eta}=0\right\}\setminus (\Omega_0\cup M^{(k)}_+\cup M^{(k)}_-),$$
		which shows that $N^{(k)}$ consists of at most four points in $\C$.

		It is also simple to show that $\Omega_0 \subset \stwo\setminus\sone,$ whence, in particular, $\stwo \neq \sone.$
	\end{example}

	\begin{example}\label{Ex:1D-Lorentz}
		For comparison, in Fig. \ref{F:spec1D-Lorentz} we plot also the spectrum for the interface of a Lorentz metal and a dielectric, i.e. \eqref{E:chi} and \eqref{E:Lorentz}. Again, we normalize $\perm_0=1$ and choose the parameters
			\beq\label{E:param-ex-Lorentz}
			\eta =1, \omega_*=1, \gamma = 1, c_L=2, k\in\{3,0.7\}.
			\eeq
			Here we obtain $S=\{-\ri\tfrac{\gamma}{2}\pm(\omega_*^2-\tfrac{\gamma^2}{4})^{1/2}\}$ and
			$\Omega_0=\{0,-\ri\tfrac{\gamma}{2}\pm (c_L+\omega_*^2-\tfrac{\gamma^2}{4})^{1/2}\}$. The set $M_+^{(k)}$ is the same as in Example \eqref{Ex:1D-Drude} and the set $M_-^{(k)}$ is given as
			$$M_-^{(k)} = \left\{\omega\in D(\perm): \omega^4 +\ri\gamma \omega^3 -(k^2+t+c_L+\omega_*^2)\omega^2 -\ri\gamma (k^2+t)\omega+\omega_*^2(k^2+t)=0 \ \text{for some} \ t\geq 0\right\},$$
			i.e. it consists of four curves.

		The set $N^{(k)}$ has again at most four points because
			$$
			\begin{aligned}
				N^{(k)}=& \left\{\omega\in D(\perm): \omega^4+\ri\gamma\omega^3-\left(\frac{2+\eta}{1+\eta}k^2+c_L+\omega_*^2\right)\omega^2 -\ri\gamma k^2 \frac{2+\eta}{1+\eta}\omega + \frac{k^2}{1+ \eta}(c_L+(2+\eta)\omega_*^2)=0\right\}\\
				& \setminus (\Omega_0\cup M_+^{(k)}\cup M_-^{(k)}).
			\end{aligned}
			$$
			In our numerical examples one can observe from the figures that, unlike in the Drude case, the spectrum does not intersect the imaginary axis (except for $\omega=0$). This can be proven for $\omega_*>\gamma/2$ by a short calculation. In Section \ref{S:t-dep} we discuss the significance of this threshold for the time dependent Maxwell equations.

		Similarly to Example \ref{Ex:1D-Drude} we have $\stwo \neq \sone$ because  $\Omega_0\setminus \{0\} \subset \stwo\subset \sone.$

		\begin{figure}[h!]
			\centering
			\includegraphics[width=0.45\linewidth]{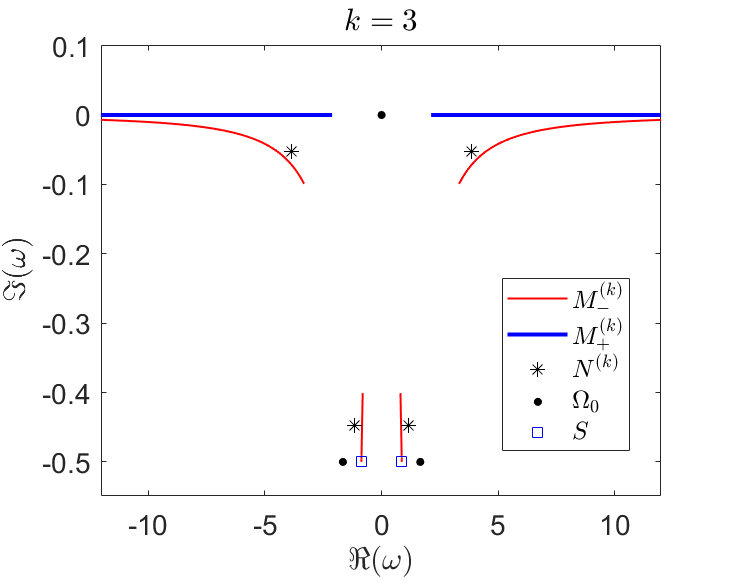}
			\includegraphics[width=0.45\linewidth]{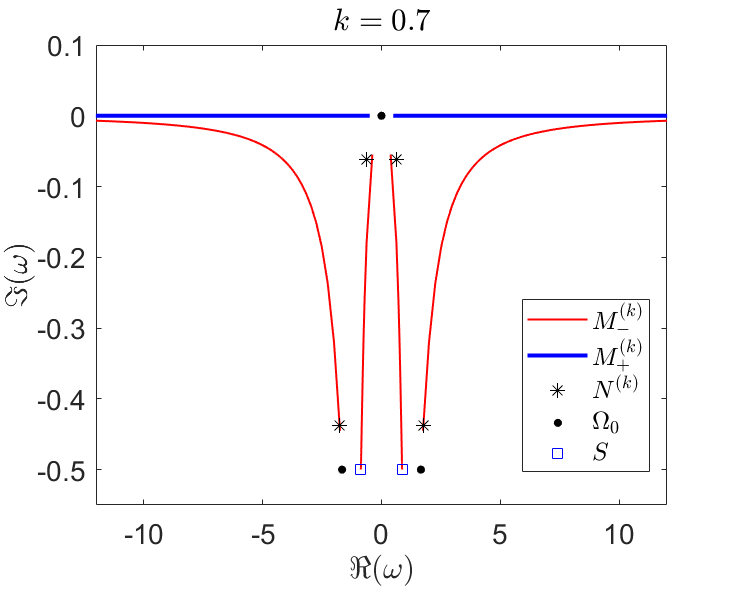}
			\caption{The sets  $M^{(k)}_+, M^{(k)}_-, N^{(k)}$, $\Omega_0$, and $S$ for the 1D Lorentz setting \eqref{E:chi}, \eqref{E:Lorentz} with  $\perm_0=1$ and the parameter values in \eqref{E:param-ex-Lorentz}.}
			\label{F:spec1D-Lorentz}
		\end{figure}
	\end{example}
	%--------------------------------------------------
	\subsection{Main Results in Two Dimensions}\label{S:main-res-2D}
	For the two-dimensional setting we define the operator pencil
	$$\cP:=\cL:=(L(\omega;\lambda))_{\substack{\omega\in D(\perm)\\ \lambda\in \C}},$$
	where
	$$L(\omega;\lambda):=A(\omega)-\lambda B(\omega), \ A(\omega)u:=\nabla\times \nabla\times u, \ B(\omega)u:= \omega^2\perm(x_1,\omega)u, \ \nabla:=(\pa_{x_1}, \pa_{x_2}, 0)^T.$$

	Just like in the one-dimensional case, we also define
	$$L(\omega):= L(\omega;1).$$

	Next, we describe our choice of function spaces. This choice is motivated in Sec.~\ref{S:2D}. We set
	$$H:= L^2(\R^2,\C^3)$$
	and choose the domain
	\begin{equation}
		\cD_\omega:=\{E \in L^2(\R^2,\C^3):  \nabla \times E, \nabla\times \nabla\times E\in L^2(\R^2,\C^3),
		\nabla\cdot(\perm E)=0 \ (\text{distributionally})\}.
	\end{equation}
	The range space becomes
	\beq\label{E:R}
	\cR:= \{r\in L^2(\R^2,\C^3): \nabla \cdot r=0 \ \text{distributionally}\}
	\eeq
	equipped with the $L^2$-inner product.

	Analogously to the 1D case $B(\omega):H\supset \cD_{\omega} \to \cR$ is bounded and  $A(\omega):H\supset \cD_{\omega} \to \cR$ is closed since it is bounded as an operator from $\cD_{\omega}$ to $\cR$. Also here the Hilbert space property of $\cD_{\omega}$ is used, see Lemma \ref{L:HS-2D}.

	We use the notation
	\begin{align}
		M_\pm:= &\{ \omega \in D(\perm)\setminus\Omega_0: \omega^2\perm_\pm (\omega)\in (0,\infty)\},\label{E:Mpm}\\
		N:= &\{ \omega \in D(\perm)\setminus \Omega_0: \text{ there exists } a\geq 0 \text{ such that } \omega^2\perm_\pm (\omega)\notin [a,\infty) \text{ and } \eqref{E:ev.cond-2D} \text{ holds}\},\label{E:N}\\
		& a(\perm_+ (\omega) + \perm_-(\omega))= \omega^2\perm_+(\omega)\perm_-(\omega),\label{E:ev.cond-2D}
	\end{align}
	and an analogous notation to \eqref{E:red-sets} for the "reduced" spectrum outside $\Omega_0$. Note that the sets $M_\pm$ coincide with $M^{(0)}_\pm$ from the one dimensional setting. For the set $N$ we have $N=\cup_{k\geq 0} N^{(k)}$. Hence, $N$ is typically a set of curves in $\C$.

	Our main results are
	\begin{mainthm}[Two dimensions; reduced spectrum]\label{T:main2D}
		~\newline\textbf{Spectrum}:
		\beq\label{E:2Dsp-red}
		\sigma^{{\rm red}}(\cL) = M_+\cup M_-\cup N.
		\eeq
		\textbf{Point spectrum}:
		\beq\label{E:2Dptsp-red}
		\sigma_p^{{\rm red}}(\cL) = \emptyset.
		\eeq
		\textbf{Essential and Weyl spectrum}:
		\beq\label{E:2Dess-red}
		\sigma^{{\rm red}}_{e,j}(\cL)=\sigma_\text{Weyl}^{{\rm red}}(\cL) =  M_+\cup M_-\cup N, \ j=1,\dots, 5.
		\eeq
	\end{mainthm}
	\begin{mainthm}[Two dimensions; exceptional set]\label{T:main2D_Om0}
		~\newline\textbf{Spectrum}:
		\beq\label{E:2Dsp-Om0}
		\sigma(\cL)\cap \Omega_0 = \Omega_0.
		\eeq
		\textbf{Point spectrum}:
		\beq\label{E:2Dptsp-Om0}
		\begin{aligned}
			\sigma_p^{(<\infty)}(\cL)\cap \Omega_0  &=  ~\emptyset,\\
			\sigma_p^{(\infty)}(\cL)\cap \Omega_0  &=  ~\{\omega \in \Omega_0: \perm_+(\omega)=0 \text{ or } \perm_-(\omega)=0 \text{ or } \perm_+(\omega)+\perm_-(\omega)=0 \}.\\
		\end{aligned}
		\eeq
		\textbf{Essential and Weyl spectrum}:
		\beq\label{E:2Dess-Om0}
		\sigma_{e,j}(\cL) \cap \Omega_0 = \sigma_\text{Weyl}(\cL)\cap \Omega_0 =  \Omega_0, \ j=1,\dots, 5.
		\eeq
	\end{mainthm}
	\brem[Block diagonal structure]\label{R:block2D}
	Just like in the 1D case the operator at hand has a block diagonal structure as visible in \eqref{E:T} and \eqref{E:domega}. As a result, the spectrum of $\cL$ is the union of the spectra of $\cL^{(1,2)}$ and $\cL^{(3)}$ in the sense of $\eqref{E:Pspec}$, where
	$$\cL^{(1,2)}:=\bspm -\pa_{x_2}^2 & \pa_{x_1}\pa_{x_2} \\ \pa_{x_1}\pa_{x_2} & -\pa_{x_1}^2 \espm-\lambda \omega^2\perm(x_1,\omega)$$
	and $\cL^{(3)}:=-\pa_{x_1}^2-\pa_{x_2}^2-\lambda \omega^2\perm(x_1,\omega)$ with the domains
	$$
	\begin{aligned}
		D(\cL^{(1,2)}) &:= \{ (u_1, u_2)^T:  u \in \cD_{\omega}\},\\
		D(\cL^{(3)}) &:= \{ u_3:  u \in \cD_{\omega}\}.
	\end{aligned}
	$$
	\erem

	\begin{example}\label{Ex:2D-Drude}
		In Fig. \ref{F:spec2D} we plot the sets $M_+, M_-, N$, $\Omega_0$, and $S$ for the same interface studied in Example \ref{Ex:1D-Drude}, i.e. with \eqref{E:chi}, \eqref{E:Drude} and with the parameters $\perm_0=1, \eta=1, \gamma=2$,  and $c_D$ specified below.
		The sets $\Omega_0$ and $S$ are the same as in Example \ref{Ex:1D-Drude} and $M_+=\R\setminus\{0\}$. For $M_-$ we get
		$$M_-=M^{(0)}=\cup_{k>0}M_-^{(k)}$$
		with $M_-^{(k)}$ from  Example \ref{Ex:1D-Drude}. Hence, for $\omega \in M^{(k)}_-  \cap  \ri\R$,  $\Im(\omega)$ is still bounded below by $-\gamma$.

		In Fig. \ref{F:spec2D} we study three values of the parameter $c_D$: $2\pi , 3.09,$ and $2.26$. The set $N$ consists of three components. The union of the two unbounded components united with $\Omega_0$ is closed in $\C$: for $a\to \infty $ the corresponding solution of \eqref{E:ev.cond-2D} satisfies $\omega \to \infty$ or $\omega\to -\infty$ and for $a\to 0+$ we have $\omega\to \omega_0\in \Omega_0$. On the other hand, the bounded component of $N$ is not closed. For $a\to \infty$  the corresponding solution $\omega(a)$ of \eqref{E:ev.cond-2D} converges to a zero of $\perm_++\perm_-$,
		i.e. to one of the points $\omega_0^\pm:=\frac{1}{2}\left(-\ri\gamma \pm \sqrt{4c_D/(\eta+2)-\gamma^2}\right)$, in such a way that $a(\perm_+(\omega(a))+\perm_-(\omega(a)))\to \omega_0^2\perm_+(\omega_0^\pm)\perm_-(\omega_0^\pm)$ as $a\to \infty$. The points $\omega_0^\pm$ are marked by green circles in  Fig. \ref{F:spec2D}. As a result, the spectrum of $\cL$ is not closed in the $\omega$-plane, similarly to the 1D case in  Example \ref{Ex:1D-Drude}. For $c_D \leq \gamma^2(\eta+2)/4$ the points $\omega_0^\pm$ lie on the imaginary axis and they coincide when
		$c_D =\gamma^2(\eta+2)/4=3$. The case $(c_D,\gamma) = (3.09,2)$ corresponds to $\gamma$ slightly above this threshold.
		Below the threshold the bounded part of $N$ includes a closed curve and a segment on the imaginary axis, which connects $\omega_0^-$ and $\omega_0^+$. The last parameter value $c_D=2.26$ depicted in Fig. \ref{F:spec2D} is such a case.
		\begin{figure}[h!]
			\centering
			\includegraphics[width=0.45\linewidth]{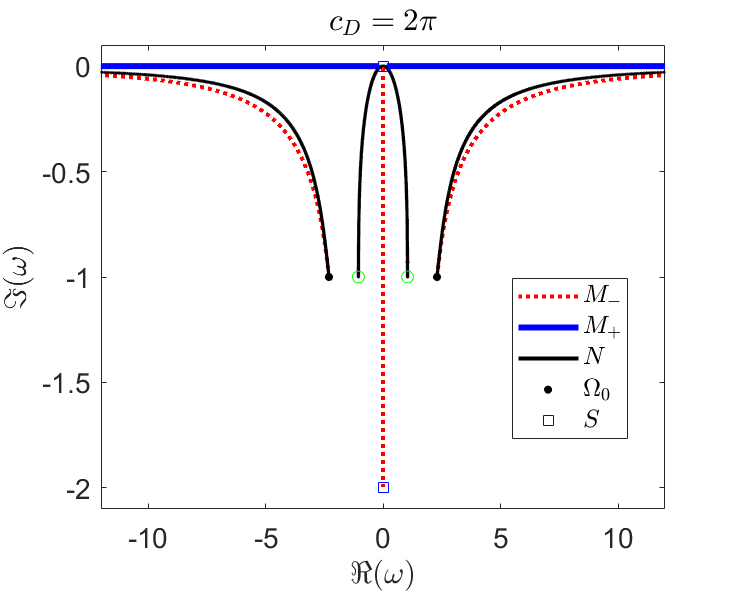}
			\includegraphics[width=0.45\linewidth]{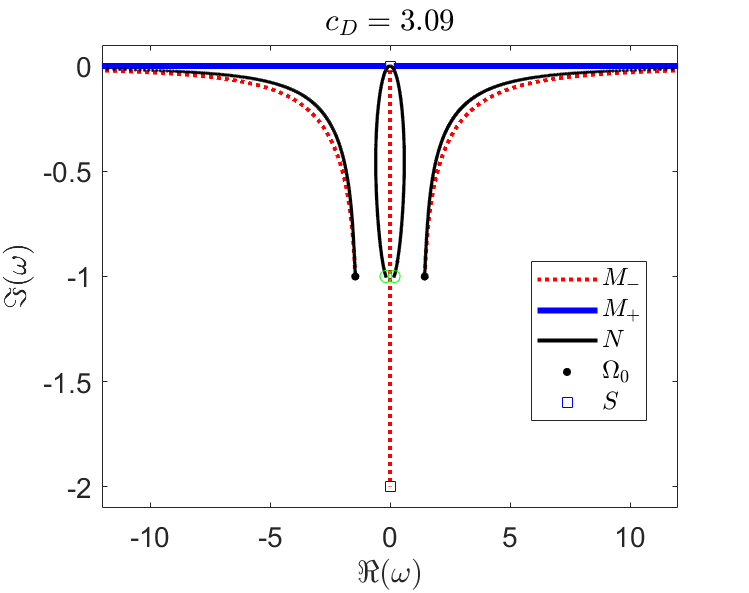}
			\includegraphics[width=0.45\linewidth]{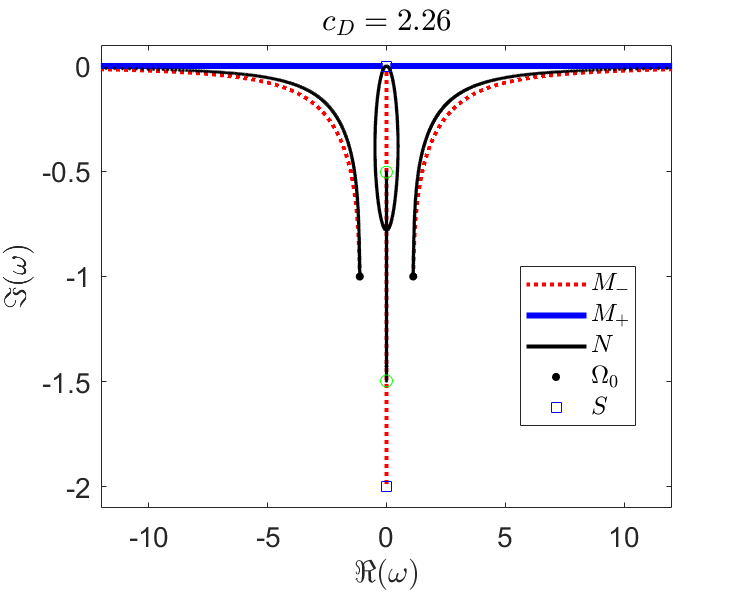}
			\caption{The sets  $M_+, M_-, N$, $\Omega_0$, and $S$ for the 2D Drude setting with \eqref{E:chi} and \eqref{E:Drude}, $\perm_0=1,$ $(\eta, \gamma)=(1,2)$, and $\omega_* \in \{2\pi,3.09,2.26\}$. The green circles mark zeros of  $\perm_++\perm_-$.}
			\label{F:spec2D}
		\end{figure}
	\end{example}
	\begin{example}\label{Ex:2D-Lorentz}
		The case of the Lorentz model \eqref{E:chi}, \eqref{E:Lorentz} in two dimensions is plotted Figure \ref{F:spec2D-Lorentz}. The chosen parameters are  $\perm_0=1, \eta=1, \gamma=1$ and we study different values of $c_L$ and $\omega_*$. Again, the bounded curves in $N$ terminate at zeroes of $\perm_+ + \perm_-$, i.e. at the points $\omega_0^\pm :=-\ri\tfrac{\gamma}{2}\pm (-\tfrac{\gamma^2}{4}+\omega_*^2+\tfrac{c_L}{2+\eta})^{1/2}$, marked by the green circles. The bounded curve in $M_-$ terminates at the singularities in $S$, i.e. $-\ri\tfrac{\gamma}{2}\pm(\omega_*^2-\tfrac{\gamma^2}{4})^{1/2}$.
			For $\omega_*\to\tfrac{\gamma}{2}+$ the two end points meet on the imaginary axis, see the last plot in Fig. \ref{F:spec2D-Lorentz}.

		        In the first three plots we visualize the dependence on $c_L$. There is a critical value of $c_L$ at which the three curves in $N$ connect. As $c_L$ passes through this value the topological structure of $N$ changes.

		Unlike in the Drude example (Example \ref{Ex:2D-Drude}) there is no spectrum on the imaginary axis if $\omega_*>\gamma/2$ (besides the point $\omega=0$), see also Example \ref{Ex:1D-Lorentz}.
		\begin{figure}[h!]
			\centering
			\includegraphics[width=0.45\linewidth]{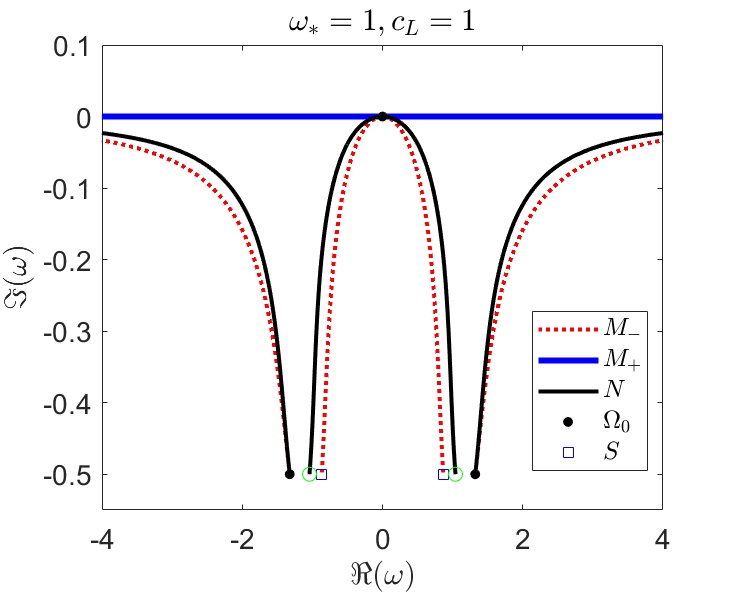}
			\includegraphics[width=0.45\linewidth]{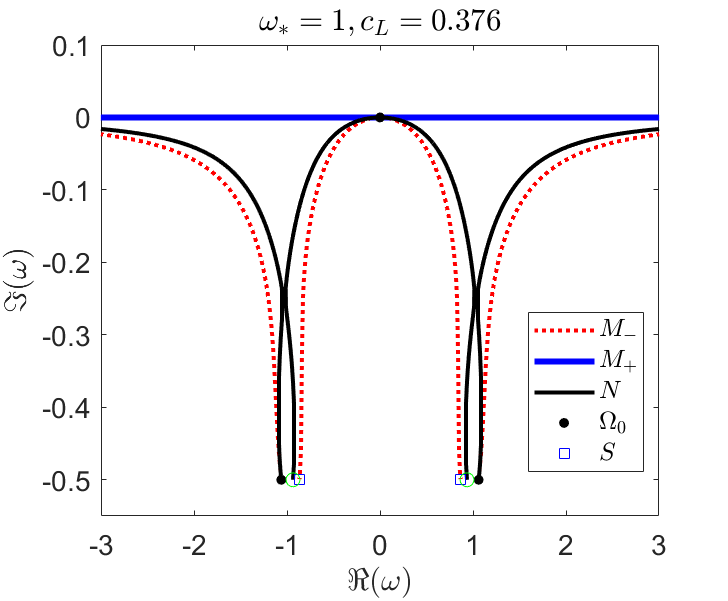}
			\includegraphics[width=0.45\linewidth]{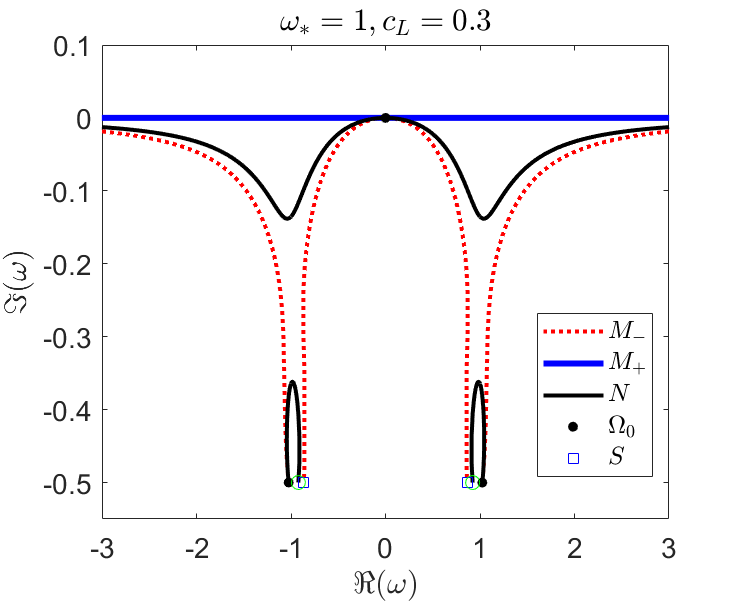}
			\includegraphics[width=0.45\linewidth]{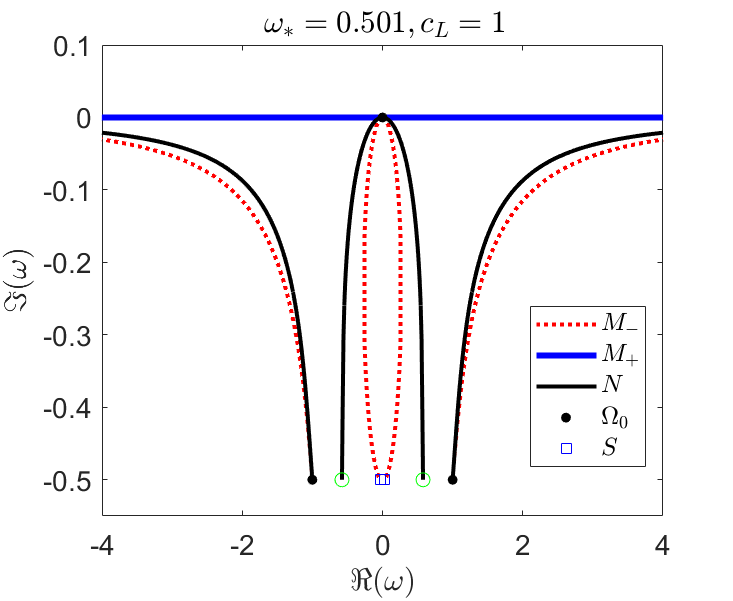}
			\caption{The sets  $M_+, M_-, N$, $\Omega_0$, and $S$ for the 2D Lorentz setting with \eqref{E:chi}, \eqref{E:Lorentz}, $\perm_0=1,$ and $(\eta, \gamma)=(1,1)$. In the first three plots we have $\omega_*=1$ and vary $c_L$ within the set $\{1, 0.376, 0.3\}$. In the last plot we choose $c_L=1$ and $\omega_*=0.501$, i.e. near the threshold $\gamma/2=0.5$. The green circles mark zeros of  $\perm_++\perm_-$.}
			\label{F:spec2D-Lorentz}
		\end{figure}
	\end{example}
	\break
	%-------------------------------------------------------------------------------------------
	%-------------------------------------------------------------------------------------------

	\section{One-Dimensional Reduction}\label{S:1D}

	In this section we study the problem defined in      \eqref{E:ansatz1D} and \eqref{E:NL-eval}. The rigorous functional analytic setting is given in Sec. \ref{S:main-res-1D}.

	\subsection{Explanation of the Functional Analytic Setting}

	We first note that we have
	\begin{equation}
		\label{E:divk}\nabla_k\cdot(\perm u)=(\perm u_1)'+\ri k\perm u_2,
	\end{equation}
	\begin{equation}\label{E:curlk}
		\nabla_k\times u=\bspm \ri k u_{3} \\ -u_{3}' \\ u_{2}'-\ri k u_{1}\espm \quad \text{and}\quad
		T_k(\pa_{x_1})u=\nabla_k\times \nabla_k\times u=\bspm k^2u_{1} + \ri k u_{2}'\\ \ri k u_{1}'-u_{2}''\\ -u_{3}''+k^2u_{3}\espm,
	\end{equation}
	see \eqref{E:Tk}.
	We are now able to rewrite $\cD_{k,\omega}$ in several ways. For $u\in \cD_{k,\omega}$, we have $u, \nabla_k\times u,
	\nabla_k\times \nabla_k\times u\in L^2(\R,\C^3)$. From \eqref{E:curlk}, we see that this is equivalent to the conditions
	$u_3\in H^2(\R,\C)$, $u_2\in H^1(\R,\C)$ and $u_2'-\ri k u_{1}\in H^1(\R,\C)$. Moreover, from \eqref{E:divk} we see that the divergence condition
	is satisfied precisely when $\perm u_1\in H^1(\R,\C)$ and $(\perm u_1)'+\ri k\perm u_2=0$. Therefore, we obtain
	\beq\label{E:Dkom}
	\cD_{k,\omega} = \{u\in L^2(\R,\C^3): u_3\in H^2(\R,\C), u_2, u_2'-\ri k u_{1}, \perm u_1\in H^1(\R,\C)\text{ and } (\perm u_1)'+\ri k\perm u_2=0\}.
	\eeq

	Next, we reformulate the conditions for lying in $H^1(\R,\C)$ in terms of functions not having a jump across the interface at $x=0$ in the sense of \eqref{E:jump}. This shows that
	\beq
	\begin{aligned}
		\cD_{k,\omega}=&\{u\in L^2(\R,\C^3): \ \nabla_k\times u_\pm, \nabla_k\times \nabla_k\times u_\pm \in L^2(\R_\pm,\C^3), \\
		&\quad \text{the divergence condition} \ \eqref{E:div-cond-1d}, \ \text{and the interface conditions} \ \eqref{E:IFC-1D} \ \text{hold}\},
	\end{aligned}
	\eeq
	\begin{align}
		\perm_\pm(\omega) \nabla_k\cdot u_\pm =\perm_\pm(\omega) (u'_{\pm,1}+\ri ku_{\pm,2}) &=0 \text{ on } \R_\pm,\label{E:div-cond-1d}\\
		\llbracket \perm(\omega) u_1\rrbracket = \llbracket u_2\rrbracket=\llbracket u_3\rrbracket =  \llbracket u_2'-\ri k u_1 \rrbracket = \llbracket u_3'\rrbracket &=0,\label{E:IFC-1D}
	\end{align}
	where we define
	$$u_\pm:= u|_{\R_\pm}, \text{ with } \R_\pm :=\{x\in \R: \pm x>0\}.$$
	For a detailed proof, see Lemma \ref{lem:1Dinterface} in Appendix \ref{App:interf-1D}.

	Also,
	$$\cD_{k,\omega} = \{u \in L^2(\R,\C^3): u_{\pm,3}\in H^2(\R_\pm,\C), u_{\pm,2}, u_{\pm,2}'-\ri k u_{\pm,1}, \perm_\pm u_{\pm,1}\in H^1(\R_\pm,\C)\text{ and }
	\eqref{E:div-cond-1d}, \eqref{E:IFC-1D} \text{ hold}\},$$
	which follows just like in \eqref{E:Dkom}.

	\brem
		Note that using the range space $\cR_k$ (which is a closed subspace of $L^2(\R,\C^3)$), see \eqref{E:Rk}, instead of $L^2(\R,\C^3)$ is necessary for the spectral analysis. Namely, when investigating the resolvent set of the spectral problem, we have to consider the inhomogeneous problem
		\be
		T_k(\pa_{x_1}) u - \omega^2\perm(\omega)u =r,
		\label{eq9}
		\ee
		where $r\in L^2(\R,\C^3)$.  Since, for $u\in \cD_{k,\omega}$ the left hand side  of (\ref{eq9})  is divergence free, we have to impose  the condition $\nabla_k\cdot r=0$ (in the distributional sense)  also to the right hand side of (\ref{eq9}); otherwise, the resolvent set would be empty, implying that  the spectrum would be the whole of $D(\perm)$.

	The functional setting is non-standard as the domain  of the pencil  is not contained within its range space. Indeed, unless  $\perm_+=\perm_-$, there exist $u\in \cD_{k,\omega}$ such that $\nabla_k \cdot u \neq 0$, which can be concluded using only the interface conditions. The functions  $u\in \cD_{k,\omega}$ satisfy $\llbracket \perm u_1\rrbracket =0$ and hence $u_1$ is discontinuous unless $\perm_+=\perm_-$, while  the  functions $f\in \cR_k$ satisfy $\nabla_k\cdot f=0$. Thus $f_1'=-\ri k f_2\in L^2(\R)$, i.e. $f_1\in H^1(\R)$ and $f_1$ is continuous. In summary $\cD_{k,\omega}  \not\subset \cR_k$.
	\erem

	We   equip $\cD_{k,\omega}$ with the inner product
	\beq
	\langle u, v \rangle_{\cD_{k,\omega}}:=\langle u, v \rangle_{L^2}+  \langle \nabla_k \times \nabla_k\times u, \nabla_k \times \nabla_k\times v \rangle_{L^2}, u,v\in \cD_{k,\omega},  \label{E:in-prod-1D}
	\eeq
	which corresponds to the graph norm \eqref{E:graph-norm} in our abstract setting. As we show next, $\cD_{k,\omega}$ equipped with this inner product is a Hilbert space. This result is classical but we provide it for the reader's convenience. (Note that, equivalently, we could show that $A(\omega):\cD_{k,\omega}\to\cR$ is closed.)
	%\textcolor{red}{RATHER: As we know from Remark 1.1, $\cD_{k,\omega}$ equipped with this inner product is a Hilbert space. DROP THE FOLLOWING LEMMA. ------------ NO, we at least need to show $A(\omega)$ is closed.}
	\blem\label{L:HS-1D}
	$(\cD_{k,\omega},\langle\cdot,\cdot\rangle_{\cD_{k,\omega}})$ is a Hilbert space for any $k\in \R$ and $\omega \in D(\perm)$.
	\elem
	\bpf
	First note that the norm
	$$\|v\|:=\left(\|v\|_{L^2(\R)}^2 + \|\nabla_k \times v\|_{L^2(\R)}^2 + \|\nabla_k \times \nabla_k \times v\|_{L^2(\R)}^2 \right)^{1/2}$$
	is equivalent to the norm generated by $\langle \cdot, \cdot\rangle_{\cD_{k,\omega}}$ as (via Lemma \ref{L:PI-1D})
	$$\|\nabla_k \times v\|_{L^2(\R)}^2 = \langle v, \nabla_k \times \nabla_k \times v\rangle_{L^2} \leq \|v\|_{L^2(\R)}\|\nabla_k \times \nabla_k \times  v\|_{L^2(\R)}$$
	for each $v\in \cD_{k,\omega}$. Here, for the application of Lemma \ref{L:PI-1D} we set $u:=\nabla_k\times v$. Due to $v\in \cD_{k,\omega}$ we know that $u, \nabla_k\times u\in L^2$ and hence the assumptions of the lemma are indeed satisfied.

	Let $(u_n)\subset \cD_{k,\omega}$ be a Cauchy sequence. Due to the above norm equivalence we have that $(u_n), (\nabla_k \times u_n)$, and $ (\nabla_k \times \nabla_k \times u_n)$ are Cauchy in $L^2(\R,\C^3)$. Hence there are $u,f,g\in L^2(\R,\C^3)$ such that
	\beq\label{E:un-conv}
	u_n\to u, \nabla_k\times u_n \to f, \ \text{and} \ \nabla_k \times \nabla_k \times u_n \to g \quad \text{in}\quad L^2(\R,\C^3).
	\eeq

	Because
	$$\int_\R u\cdot (\nabla_k\times \phi)\dd x = \lim_{n\to\infty}\int_\R u_n \cdot (\nabla_k\times \phi)\dd x =\lim_{n\to\infty}\int_\R(\nabla_k\times u_n) \cdot  \phi \dd x = \int_\R f\cdot \phi\dd x  $$
	for each $\phi \in C^\infty_c(\R,\C^3)$, we get $\nabla_k\times u = f$. Analogously, we obtain $\nabla_k\times \nabla_k\times u = g$.

	Since
	$$\int_\R (\perm u)\cdot \nabla_k \varphi\dd x = \lim_{n\to\infty}\int_\R (\perm u_n) \cdot \nabla_k\varphi\dd x=0$$
	for each $\varphi \in C^\infty_c(\R,\C)$, we have $\nabla_k\cdot (\perm u)=0$.
	We conclude that $u\in \cD_{k,\omega}$ and due to \eqref{E:un-conv} we have also $u_n\to u$ in $\cD_{k,\omega}$.
	\epf

	\brem\label{R:setting-mu-1d}
	If one allows also the permeability to depend on $x_1$, i.e. replacing $\boldsymbol{\mu}_0$ by $\boldsymbol{\mu}=\boldsymbol{\mu}(x_1)$ in \eqref{E:Maxw-1stord}, then the setting changes. In this case one arrives at the second order equation $\curl(\tfrac{1}{\boldsymbol{\mu}} \curl E)=\omega^2 D$. In the case of $\boldsymbol{\mu}$ being constant on each half line $\R_\pm$, one obtains the old second order equation \eqref{E:Maxw-2ndord} (with a modified $\perm$) on $\R_+$ and $\R_-$ but not on $\R$. The last two interface conditions in \eqref{E:IFC-1D} have to be replaced by  $\llbracket \tfrac{1}{\boldsymbol{\mu}}(u_2'-\ri k u_1) \rrbracket = \llbracket\tfrac{1}{\boldsymbol{\mu}}u_3'\rrbracket =0$.
	\erem

	%------------------------------------------------------------
	\subsection{A Subset of the Resolvent Set $\rho^{{\rm red}}(\cL_k)$}

	First, we restrict our attention to the set $\omega \in D(\perm)\setminus \Omega_0$.
	Note that for $\omega \in D(\perm)\setminus \Omega_0$ the divergence condition \eqref{E:div-cond-1d} reduces to
	\beq
	\nabla_k\cdot u_\pm =0 \text{ on } \R_\pm.\label{E:div-cond-1d-red}
	\eeq
	The first step towards the proof of Theorem \ref{T:main1D}  is the following
	\bprop\label{T:resolv-1D}
	$$\sigma^{{\rm red}}(\cL_k)\subset (M^{(k)}_+\cup M^{(k)}_-) \dot{\cup} N^{(k)}=:\cM^{(k)},$$
	where $ M^{(k)}_\pm$ and $ N^{(k)}$ are defined in \eqref{E:Mkpm} and \eqref{E:Nk}.
	\eprop
	\brem
	In other words, the proposition  states that for the resolvent set we have
	$$\rho^{{\rm red}}(\cL_k)\supset D(\perm)\setminus (\Omega_0 \cup \cM^{(k)}).$$
	\erem
	\brem \label{R:Nk-equiv}
	In several of the subsequent calculations,  we will make use of the following equivalent descriptions of $N^{(k)}$.
	$$
	\begin{aligned}
		N^{(k)}=&~\{ \omega \in  D(\perm)\setminus \Omega_0: \omega^2\perm_+(\omega),\omega^2\perm_-(\omega)\notin [k^2,\infty),  \perm_+ (\omega)\sqrt{ k^2-\omega^2\perm_-(\omega)}+ \perm_- (\omega)\sqrt{ k^2-\omega^2\perm_+(\omega)}=0\}.
	\end{aligned}
	$$
	To prove this it suffices to show that, for $ \omega \in  D(\perm)\setminus \Omega_0$ such that $\omega^2\perm_+(\omega),\omega^2\perm_-(\omega)\notin [k^2,\infty)$, equation \eqref{E:ev.cond-1D} and
	\beq\label{E:Nk1}
	\perm_+ (\omega)\sqrt{ k^2-\omega^2\perm_-(\omega)}= -\perm_- (\omega)\sqrt{ k^2-\omega^2\perm_+(\omega)}
	\eeq
	are equivalent.\\
	Clearly \eqref{E:ev.cond-1D} follows from \eqref{E:Nk1} by squaring and using $\perm_+(\omega)\neq \perm_-(\omega)$ which is also a consequence of \eqref{E:Nk1}. Assume now \eqref{E:ev.cond-1D}. Then we get
	$(k^2-\omega^2\perm_+(\omega))(k^2-\omega^2\perm_-(\omega))=k^4$ and thus
	\beq\label{E:Wpm1}
	\sqrt{k^2-\omega^2\perm_-(\omega)}=\sqrt{\frac{k^4}{k^2-\omega^2\perm_+(\omega)}}=k^2\sqrt{\frac{1}{k^2-\omega^2\perm_+(\omega)}}=k^2\frac{1}{\sqrt{k^2-\omega^2\perm_+(\omega)}}
	\eeq
	as $k^2-\omega^2\perm_+(\omega)\notin (-\infty,0]$. At the same time \eqref{E:ev.cond-1D} implies $\perm_-(\omega)(k^2-\omega^2\perm_+(\omega))=-k^2\perm_+(\omega)$ and hence
	\beq\label{E:Wpm2}
	\perm_-(\omega)=-k^2\frac{\perm_+(\omega)}{k^2-\omega^2\perm_+(\omega)}.
	\eeq
	Using \eqref{E:Wpm1} and \eqref{E:Wpm2}, we   obtain
	$$\perm_+ (\omega)\sqrt{ k^2-\omega^2\perm_-(\omega)}+\perm_- (\omega)\sqrt{ k^2-\omega^2\perm_+(\omega)}=k^2\frac{\perm_+(\omega)}{\sqrt{k^2-\omega^2\perm_+(\omega)}}-k^2\frac{\perm_+(\omega)}{k^2-\omega^2\perm_+(\omega)}\sqrt{k^2-\omega^2\perm_+(\omega)}=0,$$
	showing the desired equality of the sets.
	\erem

	\bpf (of Prop. \ref{T:resolv-1D})
	Let $r \in \cR_k$ and $\omega \in  D(\perm)\setminus (\Omega_0\cup \cM^{(k)})$. We need to show the existence of a unique $u\in  \cD_{k,\omega}$ such that
	\be\label{Neq22}
	\left. \begin{array}{rcl}
		(k^2 -\omega^2\perm(\omega))u_1 + \ri k u_2'&=& r_1 \\
		\ri k u_1 ' -u_2 '' - \omega^2\perm(\omega)u_2  &=& r_2 \\
		- u_3 ''  +(k^2 -\omega^2\perm(\omega)) u_3     &=& r_3
	\end{array} \right\} {\rm ~on~} (-\infty,0) {\rm ~and~} (0,\infty).
	\ee

	We solve the first equation for $u_1$:
	\be u_1 =\dfrac{ r_1- \ri k u_2'}{k^2-\omega^2\perm(\omega)} . \label{Neq23} \ee
	By \eqref{E:div-cond-1d-red}, we have $r_1'=-\ri kr_2$  and hence, by (\ref{Neq23})
	\be
	u_1'=-\dfrac{\ri k}{k^2-\omega^2\perm(\omega)}(u''_2+r_2) \ \text{ on } (-\infty,0) \text{ and } (0,\infty).
	\label{eq24}
	\ee
	Transferring this  into the second equation in (\ref{Neq22})  gives
	$$\dfrac{k^2}{k^2-\omega^2\perm(\omega)} (u''_2+r_2)-u''_2-\omega^2\perm(\omega)u_2=r_2, $$
	and using $\omega^2\perm \neq 0$,
	\be  -u''_2+(k^2-\omega^2\perm(\omega))u_2 =r_2 \quad \mbox{ on }(-\infty,0)\mbox{  and } (0,\infty). \label{Neq25}
	\ee

	Next, we apply the variation of parameters formula to \eqref{Neq25} and search for an $L^2(\R)$-solution $u_2$ with $\llbracket u_2\rrbracket =0$.
	We need the following lemma.
	\begin{lemma}\label{Nl2lem}
		Let $\mu \in \C, \Re \mu >0, r \in L^2(\R)$. Then
		\begin{align}
			\psi_1(x)&:=e^{\mu x} \int_x^\infty e^{-\mu t } r(t) \dd t \text{ satisfies } \psi_1 \in L^2(0,\infty) \mbox{ with } \norm{\psi_1}_{L^2(0,\infty)}\leq \dfrac1{\Re \mu}\norm{r}_{L^2}, \label{E:psi1-est}\\
			\psi_2(x)&:=e^{-\mu x} \int_0^x e^{\mu t } r(t) \dd t \text{ satisfies } \psi_2\in L^2(0,\infty) \mbox{ with } \norm{\psi_2}_{L^2(0,\infty)}\leq \dfrac1{\Re \mu}\norm{r}_{L^2}.\label{E:psi2-est}
		\end{align}
	\end{lemma}

	\begin{proof}
		The proof is a simple application of Young's inequality for convolutions $\|f*g\|_{L^2(\R)}\leq \|f\|_{L^1(\R)}\|g\|_{L^2(\R)}$.  For \eqref{E:psi1-est} one chooses $f(t):=\chi_{\R_-}(t)e^{\mu t}$ and $g(t):=\chi_{\R_+}(t)r(t)$. And for \eqref{E:psi2-est} one takes $f(t):=\chi_{\R_+}(t)e^{-\mu t}$ and $g(t):=\chi_{\R_+}(t)r(t)$. In both cases one uses $\|g\|_{L^2(\R)}\leq \|r\|_{L^2(\R)}$.
	\end{proof}

	Below we use the notation
	$$\mu_\pm:=\sqrt{k^2-\omega^2\perm_\pm},$$
	where the $\omega$ dependence of $\perm_\pm$ and $\mu_\pm$ has been suppressed. Note that $\Re \mu_\pm >0$ because $\omega \notin M^{(k)}_\pm$. The variation of parameters formula applied to   \eqref{Neq25} together with the $L^2-$condition  for $u_2$ in \eqref{E:Dkomega} and the second interface condition $\llbracket u_2\rrbracket =0$ in (\ref{E:IFC-1D}) yield
	\begin{eqnarray}
		u_2(x_1) \hspace{-0.6cm}&& =\left[ C_2(k) -\dfrac1{2\mu_+} \int_0^\infty e^{-\mu_+t} r_2(t)\dd t \right] e^{-\mu_+x_1} \nonumber \\
		&&+ \dfrac1{2\mu_+} \left[ e^{\mu_+x_1} \int_{x_1}^\infty  e^{-\mu_+t}r_2(t)\dd t + e^{-\mu_+x_1} \int_0^{x_1}  e^{\mu_+t}r_2(t)\dd t \right]  \mbox{ for }x_1 \in (0,\infty), \label{E:u2plus}\\
		u_2(x_1) \hspace*{-0.6cm} &&= \left[ C_2(k) -\dfrac1{2\mu_-} \int_{-\infty}^0 e^{\mu_-t} r_2(t)\dd t \right] e^{\mu_-x_1} \nonumber \\
		&&+ \dfrac1{2\mu_-} \left[ e^{\mu_-x_1} \int_{x_1}^0e^{-\mu_-t}r_2(t)\dd t + e^{-\mu_-x_1} \int_{-\infty}^{x_1}  e^{\mu_-t}r_2(t)\dd t \right]   \mbox{ for }x_1 \in (-\infty,0) , \label{E:u2minus}
	\end{eqnarray}
	with $C_2(k)\in \C$ denoting a free constant. We track the $k-$dependence of the constants for the purposes of the two-dimensional case in Sec. \ref{S:2D}. Now (\ref{Neq23}) gives
	\begin{eqnarray}
		&&u_1(x_1) =\dfrac1{\mu_+^2}   \Big\{
		r_1 (x_1)  +\ri k\mu_+  \Big[ C_2(k) -\dfrac1{2\mu_+} \int_0^\infty e^{-\mu_+t} r_2(t)\dd t \Big] e^{-\mu_+x_1} \nonumber \\
		&&-\dfrac12 \ri k \Big[     e^{\mu_+x_1} \int_{x_1}^\infty  e^{-\mu_+t}r_2(t)\dd t -e^{-\mu_+x_1} \int_0^{x_1}  e^{\mu_+t}r_2(t)\dd t \Big] \Big\}  \quad \mbox{ for }x_1 \in (0,\infty) , \label{E:u1plus}\\
		&&u_1(x_1) =\frac{1}{\mu_-^2} \Big\{   r_1(x_1)-  \ri k \mu_- \Big[ C_2(k) -\dfrac1{2\mu_-} \int_{-\infty}^0 e^{\mu_-t} r_2(t)\dd t \Big] e^{\mu_-x_1} \nonumber \\
		&&-\dfrac12 \ri k    \Big[ e^{\mu_-x_1} \int_{x_1}^0e^{-\mu_-t}r_2(t)\dd t - e^{-\mu_-x_1} \int_{-\infty}^{x_1}  e^{\mu_-t}r_2(t)\dd t \Big]  \Big\} \quad  \mbox{ for }x_1 \in (-\infty,0). \label{E:u1minus}
	\end{eqnarray}
	Next, we ensure that the first interface condition in \eqref{E:IFC-1D} is satisfied. Note that $r\in \cR_k$ implies $r_1\in H^1(\R)$ and hence $r_1$ is continuous. Equations (\ref{E:u1plus})  and (\ref{E:u1minus})  imply that
	\beq \label{E:u1-0pm}
	\begin{aligned}
		& u_1(0+)=  \dfrac1{\mu_+} \ri k C_2(k)  + \dfrac1{\mu_+^2} \Big[  r_1(0)-\ri k \int_0^\infty e^{-\mu_+t} r_2(t)\dd t \Big],\\
		& u_1(0-)= - \dfrac1{\mu_-} \ri k C_2(k)  + \dfrac1{\mu_-^2} \Big[  r_1(0)+\ri k \int_{-\infty}^0 e^{\mu_-t} r_2(t)\dd t \Big].
	\end{aligned}
	\eeq
	Noting that $\frac{\perm_+}{\mu_+} + \frac{\perm_-}{\mu_-} = \frac{1}{\mu_+
		\mu_-} \Big[ \perm_+ \mu_- + \perm_- \mu_+ \Big] \neq 0$ by the assumption $\omega\notin N^{(k)}$ (see Remark \ref{R:Nk-equiv}),
	the first interface condition in \eqref{E:IFC-1D}  gives
	\begin{eqnarray}\label{E:C2}
		\hspace{-0.6cm}
		&C_2(k)=\dfrac1{\frac{\perm_+}{\mu_+} +\frac{\perm_-}{\mu_-}}   \Big[ \hspace{-0.1cm}  \underbrace{ \Big( \dfrac{\perm_-} {\mu_-^2} -\dfrac{\perm_+} {\mu_+^2 } \Big)
		}_{
			=\frac{1}{\mu_+^2\mu_-^2} k^2  (\perm_- - \perm_+) } \hspace{-0.1cm}
		% [ \ri k (\perm_+-\perm_-)
		\frac{r_1(0) }{\ri k}
		+\dfrac{ \perm_+}{\mu_+^2}  \int_0^{\infty} e^{-\mu_+t}r_2(t)\dd t
		+\dfrac{\perm_-}{\mu_-^2}  \int_{-\infty}^0 e^{\mu_-t}r_2(t)\dd t \Big]  \nonumber \\
		& =\dfrac1{ \mu_+ \mu_- (\perm_+\mu_-+\perm_-\mu_+)}  \Big[ \ri k (\perm_+ - \perm_-)r_1(0)+\perm_+ \mu_-^2 \int_0^\infty e^{ -\mu_+t}r_2 (t)\dd t + \perm_- \mu_+^2 \int_{-\infty}^0 e^{ \mu_- t} r_2 (t)\dd t \Big]
	\end{eqnarray}
	provided $k\neq 0$. If $k=0$, then $C_2(k)$ is arbitrary since \eqref{E:u1-0pm} gives $u_1(0+)=-r_1(0)/(\omega^2\perm_+(\omega))$ and $u_1(0-)=-r_1(0)/(\omega^2\perm_-(\omega))$, which implies the interface condition.

	Finally, the differential equation for $u_3$ in (\ref{Neq22}) has the same form as (\ref{Neq25}).  Hence, together with the $L^2-$condition for $u_3$ in (\ref{E:Dkomega}) and the third interface condition $ \llbracket u_3\rrbracket =0$ in (\ref{E:IFC-1D}) we obtain (\ref{E:u2plus}), (\ref{E:u2minus}) also for $u_3$, with a new constant $C_3(k)$  and $r_3$ instead of $r_2$ on the right-hand side. Next, the interface condition $ \llbracket u_3'\rrbracket =0$ reads
	$$
	\begin{aligned}
		&-\mu_+ \Big[ C_3(k)-\dfrac1{2\mu_+}\int_0^\infty e^{-\mu_+t}r_3(t)dt \Big]+ \dfrac12\int_0^\infty e^{-\mu_+t} r_3(t)dt\\
		&=\mu_-\Big[ C_3(k)-\dfrac1{2\mu_-}\int_{-\infty}^0 e^{\mu_-t}r_3(t)dt \Big] - \dfrac12\int_{-\infty}^0 e^{\mu_-t }r_3(t)dt,
	\end{aligned}
	$$
	which gives, noting that $\mu_+ + \mu_- \neq 0$ since $\Re \mu_{\pm} > 0$,
	\be
	C_3(k) =\dfrac1{\mu_+ + \mu_-} \Big[ \int_0^\infty e^{-\mu_+t }r_3(t)dt +  \int_{-\infty}^0 e^{\mu_-t }r_3(t)dt \Big].
	\label{Neq31}
	\ee
	The expressions (\ref{E:u2plus}), (\ref{E:u2minus}) (also for $u_3$), (\ref{E:u1plus}), (\ref{E:u1minus}), (\ref{E:C2}), and (\ref{Neq31}) uniquely determine   the function $u\in L^2(\R,\C^3)$ which satisfies the differential equations (\ref{Neq22}) and all the interface conditions (\ref{E:IFC-1D}) except for $\llbracket u_2'-\ri k u_1 \rrbracket=0$.  We are left to show that this $u$  also satisfies the divergence condition \eqref{E:div-cond-1d}, the last interface condition $\llbracket u_2'-\ri k u_1 \rrbracket=0$, the $L^2$ conditions in (\ref{E:Dkomega}), and the estimate
	\beq
	\|u\|_{L^2(\R)}\leq C\|r\|_{L^2(\R)}  \mbox{ with } C  \mbox{ independent of } r. \label{Neq32}
	\eeq
	For this purpose, we first note that due to Lemma \ref{Nl2lem}, for each $\varphi \in L^2(\R,\C)$ and each $\mu\in \C$ such that $\Re \mu>0$,
	\be
	\norm{e^{\mu \cdot} \int_\cdot^\infty e^{-\mu t} \varphi(t)\dd t}_{L^2(0,\infty)},
	\norm{e^{-\mu \cdot} \int_0^\cdot e^{\mu t}\varphi(t)\dd t}_{L^2(0,\infty)}\leq
	\dfrac1{\Re \mu}\norm{\varphi}_{L^2(0,\infty)}, \label{Neq33}
	\ee
	and also
	\be
	\norm{e^{\mu \cdot}\int_\cdot^0e^{-\mu t } \varphi(t)\dd t}_{L^2(-\infty,0)},
	\norm{e^{-\mu \cdot}\int_{-\infty} ^\cdot e^{\mu t } \varphi(t)\dd t}_{L^2(-\infty,0)}
	\leq \dfrac1{\Re \mu}\norm{\varphi}_{L^2(-\infty,0)}. \label{Neq34}
	\ee
	Consequently, (\ref{E:u2plus}), (\ref{E:u2minus}) show that
	\begin{align}
		\norm{  u_{2,3}}_{L^2(0,\infty)} &\leq \dfrac1{ \sqrt{2 \Re \mu_+}} \Big[ |C_{2,3}(k)|+\left ( \dfrac1{2 \Re \mu_+}\right)^{3/2} \norm{ r_{2,3}}_{L^2(0,\infty)} \Big] +\dfrac1{(\Re \mu_+)^2 }\norm{  r_{2,3}}_{L^2(0,\infty)} \nonumber \\
		&\leq \dfrac1{\sqrt{2\Re \mu_+}} | C_{2,3}|(k) + \dfrac5{4(\Re \mu_+)^2}\norm{r_{2,3}}_{L^2(0,\infty)},\label{E:u23pl-est} \\
		\norm{  u_{2,3}}_{L^2(-\infty,0)} &\leq \dfrac1{ \sqrt{2\Re \mu_-}} \Big[ |C_{2,3}(k)| + \left ( \dfrac1{2 \Re \mu_-}\right  )^{3/2}\norm{ r_{2,3}}_{L^2(-\infty,0)} \Big] +\dfrac1{(\Re \mu_-)^2 }\norm{  r_{2,3}}_{L^2(-\infty,0)} \nonumber \\
		&\leq \dfrac1{\sqrt{2\Re \mu_-}} | C_{2,3}(k)| + \dfrac5{4(\Re \mu_-)^2}\norm{r_{2,3}}_{L^2(-\infty,0)},\label{E:u23min-est}
	\end{align}
	Similarly, from (\ref{E:u1plus}), (\ref{E:u1minus}) we get $u_1\in L^2(\R)$ and
	\begin{align}
		\norm{u_1}_{L^2(0,\infty)}&\leq \dfrac1{(\Re \mu_+)^2}\norm{r_1}_{L^2(0,\infty)} + \dfrac{|k|}{\sqrt{2} (\Re \mu_+)^{3/2}} \Big[ |C_2(k)| + \dfrac1{(2 \Re \mu_+)^{3/2}} \norm{r_2}_{L^2(0,\infty)} \Big] +\dfrac{|k|}{(\Re \mu_+)^3} \norm{r_2}_{L^2(0,\infty)}\nonumber \\
		&\leq \frac{|k|}{\sqrt{2} (\Re \mu_+)^{3/2}}  |C_2(k)| + \dfrac1{(\Re \mu_+)^2} \norm{r_1}_{L^2(0, \infty)}  + \frac{5|k|}{4(\Re \mu_+)^3}  \norm{r_2}_{L^2(0,\infty)} \label{E:u1pl-est}
	\end{align}
	and also
	\beq
	\norm{u_1}_{L^2(-\infty,0)}\leq \frac{|k|}{\sqrt{2} (\Re \mu_-)^{3/2}}  |C_2(k)| + \dfrac1{(\Re \mu_-)^2} \norm{r_1}_{L^2(-\infty,0)}  + \frac{5|k|}{4(\Re \mu_-)^3}  \norm{r_2}_{L^2(-\infty,0)} \label{E:u1min-est}.
	\eeq
	Our next task is to bound $|C_2(k)|$  and $|C_3(k)|$. We have for any any $\alpha > 0$,
	\begin{eqnarray*}
		r_1(0) &=& - \int\limits_0^{\infty} \pd{1} \left[ e^{-\alpha x_1} r_1 (x_1) \right] \dd x_1 \quad ( {\rm note:~ } r_1 \in H^1 (\R) )\\
		&=& \alpha \int\limits_0^{\infty} e^{- \alpha x_1} r_1 (x_1) \dd x_1 - \int\limits_0^{\infty} e^{- \alpha x_1} \underbrace{\pd{1} r_1}_{= - \ri k r_2}  (x_1) \dd x_1
	\end{eqnarray*}
	and therefore
	\begin{eqnarray*}
		&&|r_1(0) | \le \frac{\alpha}{\sqrt{2 \alpha}} \| r_1 \|_{L^2 (0,\infty)} + \frac{|k|}{\sqrt{2 \alpha}} \| r_2\|_{L^2 (0,\infty)}.
	\end{eqnarray*}
	Choosing $\alpha := |k| + 1$ (a useful choice for the two-dimensional case in Sec. \ref{S:2D}), we obtain
	\beq\label{E:rho1_0_est}
	|r_1(0)| \le \sqrt{|k| + 1} \big( \| r_1\|_{L^2 (0,\infty)} + \| r_2\|_{L^2 (0,\infty)} \big).
	\eeq
	Consequently, by (\ref{E:C2}), (\ref{Neq33}), and (\ref{Neq34})
	\begin{align}
		|C_2(k)| &\le \frac{1}{| \mu_+| \; |\mu_-| \;  |\perm_+ \mu_- +\perm_- \mu_+|} \Big[ |k|  \; |\perm_+ - \perm_-| \;  |r_1(0)| \nonumber \\
		&\hspace*{1cm} + |\perm_+ | \;  | \mu_-|^2 \frac{1}{\sqrt{2 \Re \mu_+}} \| r_2 \|_{L^2(0,\infty)} + | \perm_- | \; | \mu_+ |^2 \frac{1}{\sqrt{2 \Re \mu_-}} \| r_2\|_{L^2(-\infty,0)}\Big] \label{E:C2-est1} \\
		& \le \tilde{C}_2(k) \Big(\|r_1\|_{L^2(\R)}+\|r_2\|_{L^2(\R)} \Big)  \label{E:C2-est}
	\end{align}
	with $\tilde{C}_2(k)>0$. Similarly, by (\ref{Neq31}), (\ref{Neq33}), and (\ref{Neq34})
	\begin{align}
		|C_3(k)| &\leq \dfrac1{\Re(\mu_+ +\mu_-)} \Big[  \dfrac1{\sqrt{2 \Re \mu_+}}\norm{r_3}_{L^2(0,\infty)}+ \dfrac1{\sqrt{2 \Re \mu_-}}\norm{r_3}_{L^2(-\infty,0)} \Big] \label{E:C3-est1} \\
		& \leq\tilde{C}_3(k)  \Big( \norm{ r_3}_{L^2(0,\infty)}+\norm{ r_3}_{L^2(-\infty,0)} \Big).\label{E:C3-est}
	\end{align}
	with $\tilde{C}_3(k)>0$.

	Now \eqref{E:u23pl-est}-\eqref{E:u1min-est}, \eqref{E:C3-est}, and \eqref{E:C2-est} show that \eqref{Neq32} holds true.

	Next, we prove the remaining $L^2$ conditions in (\ref{E:Dkomega}), i.e. $\nabla_k\times u_\pm, \nabla_k\times \nabla_k\times u_\pm \in L^2(\R_\pm,\C^3)$. The latter follows directly from the differential equations \eqref{Neq22}. For the former it suffices to show $u_3', u_2'\in L^2(\R_\pm)$. First, (\ref{E:u2plus}) (also for $u_3$) gives
	\beq\label{E:pa1_u23-est}
	\begin{aligned}
		|u_{2,3}'(x_1)|\leq & \left [ |\mu_+||C_{2,3} (k)| +
		\dfrac1{2\sqrt{ 2 \Re \mu_+}}
		\norm{ r_{2,3}}_{L^2(0,\infty)} \right ] e^{-(\Re \mu_+)x_1}\\
		&+\dfrac12 \left| e^{\mu_+x_1} \int_{x_1}^\infty  e^{-\mu_+t} r_{2,3}(t)\dd t
		- e^{-\mu_+x_1}\int_{0}^{x_1}  e^{\mu_+t} r_{2,3}(t)\dd t \right|
	\end{aligned}
	\eeq
	and hence by (\ref{Neq33}), (\ref{Neq34}), we get  $u_{2,3}'\in L^2(0,\infty)$. Analogously we have    $u_{2,3}'\in L^2(-\infty,0)$.

	Next  we show that the divergence condition $\nabla_k\cdot u_\pm=0$ on $\R_\pm$ holds automatically due to the differential equation.
	$$u_\pm=\frac{1}{\omega^2\perm_\pm}\left(\nabla_k\times \nabla_k\times u_\pm -r_\pm\right), \text{ where } r_\pm:=r|_{\R_\pm}.$$
	As $\nabla_k\cdot r_\pm=0$ for any $r \in \cR_k$ and as $\nabla_k\cdot(\nabla_k\times f)=0$ for any $f$, we get $\nabla_k\cdot u_\pm=0$.

	Finally, it remains to show  $\llbracket u_2'-\ri k u_1 \rrbracket =0$. For $k\neq 0$ this follows from the first equation in \eqref{Neq22} and from $r\in\cR_k$. More precisely, we have
	$$u_2'-\ri k u_1 =\frac{1}{\ri k}(\omega^2\perm u_1+r_1) \text{ if } k \neq 0 \qquad \text{on} \  \R_\pm,$$
	where $Wu_1$ is continuous at $x_1=0$ due to $u\in \cD_{k,\omega}$ and the continuity of $r_1$ follows from $r_1\in H^1(\R)$, which is guaranteed by $r \in L^2(\R,\C^3)$ and the divergence condition $r_1'+\ri k r_2=0$.

	For $k=0$ the condition $\llbracket u_2'-\ri k u_1 \rrbracket =0$ reduces to the continuity of $u_2'$. From \eqref{E:u2plus} and \eqref{E:u2minus} we receive
	$$
	u_2'(0-)=\mu_-C_2(0)-\int_{-\infty}^0 e^{\mu_-t}r_2(t)\dd t  \quad \text{and} \quad u_2'(0+)=-\mu_+C_2(0)+\int_0^{\infty} e^{-\mu_+t}r_2(t)\dd t, $$
	implying that $\llbracket u_2'\rrbracket =0$ is equivalent to
	$$C_2(0)=\frac{1}{\mu_++\mu_-}\left(\int_{-\infty}^0 e^{\mu_-t} r_2(t)\dd t+\int_0^{\infty} e^{-\mu_+t} r_2(t)\dd t\right),$$
	which determines the value of $C_2$ at $k=0$, which was left free in \eqref{E:C2}.
	\epf

	%-------------------------------------------------------------------------------------------
	\subsection{Proof of \eqref{E:pt-sp1D} in Theorem \ref{T:main1D}: Eigenvalues of $\cL_k$}\label{S:evals-1D}

	The next result determines the set of eigenvalues of $\cL_k$ for any $k\in \R$, i.e. it shows \eqref{E:pt-sp1D} and the statement on the simplicity of eigenvalues.

	First, we recall the definition of $N^{(k)}$ (see \eqref{E:Nk})
		\beq\label{E:Nk-b}
		N^{(k)}= ~\{ \omega \in  D(\perm)\setminus \Omega_0: \omega^2\perm_+(\omega),\omega^2\perm_-(\omega)\notin [k^2,\infty) \text{ and } \eqref{E:ev.cond-1D-b} \text{ holds}  \}, k\in \R,
		\eeq
		\beq\label{E:ev.cond-1D-b}
		k^2(\perm_+(\omega)+\perm_-(\omega))=\omega^2\perm_+(\omega)\perm_-(\omega).
		\eeq
		In Remark \ref{R:Nk-equiv} we showed that \eqref{E:ev.cond-1D-b}  is equivalent to
		\beq\label{E:om-ev-cond}
		\perm_+(\omega)\sqrt{k^2-\omega^2\perm_-(\omega)}=-\perm_-(\omega)\sqrt{k^2-\omega^2\perm_+(\omega)}.
		\eeq

	\bprop\label{T:pt-spec-1D}
	Let $k \in \R$.
	\beq\label{E:pt-sp1Dknz-omnz}
	\begin{aligned}
		\sigma_p^{{\rm red}}(\cL_k) = N^{(k)}.
	\end{aligned}
	\eeq
	For $k=0$ we get
	\beq\label{E:N0empty}
	\begin{aligned}
		N^{(0)}=\emptyset.
	\end{aligned}
	\eeq
	Moreover, all eigenvalues in $\sigma_p^{{\rm red}}(\cL_k)$ are geometrically and algebraically simple.
	\eprop

	\brem\label{R:Wp-neq-pWm}
	Note that for $\omega  \in \sigma_p(\cL_k) \setminus \Omega_0$, neither $\perm_+(\omega)=\perm_-(\omega)$ nor $\perm_+(\omega)=-\perm_-(\omega)$ is possible. Indeed, if $\perm_+(\omega)=\perm_-(\omega)$ and $\omega \notin \Omega_0,$ then $\perm_+(\omega)=\perm_-(\omega)\neq 0$ and equation \eqref{E:om-ev-cond} implies $\sqrt{k^2-\omega^2\perm_+(\omega)}=-\sqrt{k^2-\omega^2\perm_+(\omega)}$, i.e. $\omega^2\perm_+(\omega)=k^2$, which contradicts \eqref{E:Nk-b}. On the other hand if $\perm_+(\omega)=-\perm_-(\omega)$ and $\omega \notin \Omega_0,$ then  $\perm_+(\omega)=-\perm_-(\omega)\neq 0$ and equation \eqref{E:om-ev-cond} implies $\sqrt{k^2+\omega^2\perm_+(\omega)}=\sqrt{k^2-\omega^2\perm_+(\omega)}$, i.e. $\omega^2\perm_+(\omega)=0$, a contradiction with $\omega\notin \Omega_0$.
	\erem

	\bpf
	We need the following two lemmas. As a first step we prove that all $\cD_{k,\omega}$-solutions of the differential equation have a vanishing third component.
	\blem\label{L:psi3_zero}
	Let $k\in\R, \omega\in D(\perm)$, $\omega^2\perm_\pm(\omega)\notin[k^2,\infty)$ and let $\psi\in \cD_{k,\omega}$ be a solution of $~T_k\psi-\omega^2\perm(\omega)\psi=0$. Then $\psi_3 = 0$.
	\elem
	\bpf
	From the third equation in $T_k\psi-\omega^2\perm(\omega)\psi(\omega)=0$ we get
	$$\pa_{x_1}^2\psi_3 = (k^2-\omega^2\perm(\omega))\psi_3, \quad  x_1\in \R\setminus \{0\},$$
	which implies $\psi_3(x_1)=c_\pm e^{\mp \sqrt{k^2-\omega^2\perm_\pm(\omega)}x_1}$
	for $\pm x_1>0$, where $c_\pm \in \C$.
	Note that $\psi_3\in L^2(\R)$ because $\Re(\sqrt{k^2-\omega^2\perm_\pm(\omega)})>0$.

	The interface condition $\llbracket \psi_3\rrbracket =0$ implies $c_+=c_-$ and because $\Re(\sqrt{k^2-\omega^2\perm_\pm(\omega)})\neq 0$, the condition $\llbracket \pa_{x_1}\psi_3\rrbracket=0$ yields $c_+=c_-=0$ as $-\sqrt{k^2-\omega^2\perm_+(\omega)}=\sqrt{k^2-\omega^2\perm_-(\omega)}$ is impossible.
	\epf

	\blem\label{L:efn-1D}
	Let $k\in\R$ and $\omega\in D(\perm)\setminus \Omega_0$. A non-trivial solution $\psi\in \cD_{k,\omega}$ of $~T_k\psi-\omega^2\perm(\omega)\psi=0$ exists if and only if $k\neq 0$, $k^2-\omega^2\perm_+ \notin (-\infty,0]$, $k^2-\omega^2\perm_- \notin (-\infty,0]$ and \eqref{E:om-ev-cond} holds. Up to a normalization it has the form
	\beq\label{E:psi12_form}
	\bspm \psi_1\\ \psi_2\espm(x_1)=v_\pm e^{\mp \mu_\pm x_1}, \psi_3=0, \quad \pm x_1>0
	\eeq
	with
	\beq\label{E:lam_pm}
	\mu_\pm = \sqrt{k^2-\omega^2\perm_\pm(\omega)}
	\eeq
	and
	\beq\label{E:vp_vm}
	v_+=\frac{\mu_-}{\mu_+}\bspm\ri k\\\mu_+\espm, \quad v_-=\bspm-\ri k\\ \mu_-\espm.
	\eeq
	\elem
	\bpf
	That    $\psi_3=0$ follows from Lemma \ref{L:psi3_zero}. Let   $k\in \R\setminus \{0\}$. The first two equations in $T_k\psi-\omega^2\perm(\omega)\psi=0$ can be rewritten as
	\beq\label{E:psi12-eq}
	\pa_{x_1}\begin{pmatrix} \psi_1\\ \psi_2\end{pmatrix}  = \begin{pmatrix}  0 & -\ri k\\ \frac{\omega^2\perm_\pm-k^2}{\ri k} & 0\end{pmatrix} \begin{pmatrix} \psi_1\\ \psi_2\end{pmatrix}=:M_\pm \begin{pmatrix} \psi_1\\ \psi_2\end{pmatrix}, \quad \pm x_1>0,
	\eeq
	where the second equation follows directly from the first equation in $T_k\psi-\omega^2\perm(\omega)\psi=0$ and the first equation is obtained by combining the second and the differentiated first equation in $T_k\psi-\omega^2\perm(\omega)\psi=0$. This process is moreover reversible, hence \eqref{E:psi12-eq} is equivalent to $T_k\psi-\omega^2\perm(\omega)\psi=0$ (with $\psi_3=0$).

	The eigenpairs of the matrix $M_+$ are $(\mu_+,(-\ri k, \mu_+)^T)$ and $(-\mu_+,(\ri k, \mu_+)^T)$ and the eigenpairs of $M_-$ are $(\mu_-,(-\ri k, \mu_-)^T)$ and $(-\mu_-,(\ri k, \mu_-)^T)$. Hence, all $L^2(\R)$-solutions are given by
	$$
	\bspm \psi_1\\ \psi_2 \espm(x_1)=\begin{cases} A\bspm \ri k\\ \mu_+\espm e^{-\mu_+ x_1},  \ x_1>0,\\
		B\bspm -\ri k\\ \mu_-\espm e^{\mu_- x_1},  \ x_1<0
	\end{cases}
	$$
	with $\Re(\mu_\pm)>0$, i.e. $k^2-\omega^2\perm_\pm \notin (-\infty,0].$ The $L^2$-property of $\nabla_k\times \psi$ and $\nabla_k\times\nabla_k\times \psi$ on $\R_\pm$ is obvious from the exponential form of $\psi$. The divergence condition \eqref{E:div-cond-1d} follows directly from the first equation in \eqref{E:psi12-eq}.

	Next, we consider the interface conditions \eqref{E:IFC-1D} (for $k\neq 0$). The conditions $\llbracket \psi_2\rrbracket =0$ and $\llbracket \perm\psi_1\rrbracket =0$ yield
	\beq \label{E:ABrel}
	A\mu_+=B\mu_- \quad \text{and} \quad A \perm_+=-B\perm_-.
	\eeq
	The first condition in \eqref{E:ABrel} implies that \eqref{E:psi12_form} is the only nontrivial solution (up to normalization). Combining the equations in \eqref{E:ABrel} for a nontrivial solution, (i.e. for $A,B\neq 0$), we get
	\beq\label{E:ABrel2}
	\mu_-\perm_+=-\mu_+\perm_-,
	\eeq
	i.e. \eqref{E:om-ev-cond}.

	The remaining interface condition to be satisfied in \eqref{E:IFC-1D} is $\llbracket \pa_{x_1}\psi_2-\ri k \psi_1\rrbracket =0$.   This is  satisfied due to \eqref{E:psi12_form} and \eqref{E:om-ev-cond} as one easily checks.

	Finally, we discuss the case $k=0$. The first equation in $T_k\psi-\omega^2\perm\psi=0$ reads $\omega^2\perm\psi_1=0$, and hence $\psi_1=0$ since $\omega^2\perm_\pm \neq 0$. The remaining equation
	$$\psi_2'' + \omega^2\perm\psi_2=0, \ x_1\in \R\setminus \{0\}$$
	and the $L^2(\R)$ property imply $\psi_2(x_1)=c_\pm e^{\mp \sqrt{-\omega^2\perm_\pm}x_1}$ for $\pm x_1>0$ with $c_\pm \in \C$. The interface conditions for $k=0$ and $\psi_1=\psi_3=0$ reduce to  $\llbracket \psi_2\rrbracket=\llbracket \pa_{x_1}\psi_2\rrbracket=0$. However, this is possible only if $c_+=c_-=0$ since $\omega^2\perm_\pm \neq 0$. Hence $\psi=0$.
	\epf
	We continue with the proof of the proposition. We have shown \eqref{E:pt-sp1Dknz-omnz} and $\sigma_p^{{\rm red}}(\cL_0)=\emptyset$.

	Next, equation \eqref{E:N0empty} follows immediately from the definition of $N^{(0)}$ because $0=\omega^2\perm_+(\omega)\perm_-(\omega)$ is impossible with $\omega \notin \Omega_0$.

	Lemma \ref{L:efn-1D} implies that each eigenvalue $\omega$ in $\sigma_p^{{\rm red}}(\cL_{k})$ is geometrically simple in the sense that $\lambda=1$ is a geometrically simple eigenvalue of \eqref{E:gen-sp-prob-1D} and
	the eigenfunction $\psi$ is given by \eqref{E:psi12_form}, \eqref{E:lam_pm}, and \eqref{E:vp_vm}.

	Finally, we show that the eigenvalue  $\omega$ in $\sigma_p^{{\rm red}}(\cL_{k})$ is also algebraically simple, which by \eqref{E:gen-evec-eq} means that the problem
	$$L_k(\omega)u=\omega^2\perm\psi, \quad \psi\in\ker(L_k(\omega))\setminus\{0\},$$
	has no solution $u\in D_{k,\omega}$.

	Assuming for a contradiction that $u$ is a solution, we first follow the lines of the proof of Proposition \ref{T:resolv-1D} with $r:=\omega^2\perm\psi\in\cR_k$. Since $\omega^2\perm_\pm\notin\{0\}\cup[k^2,\infty)$, we obtain \eqref{E:u2plus}-\eqref{E:u1-0pm} as before. Similarly to the calculation following \eqref{E:u1-0pm}, we find that the interface condition $\llbracket\perm u_1\rrbracket=0$ is equivalent to
	\begin{eqnarray}\label{E:C2b}
		\hspace{-0.6cm}
		&C_2(k) \mu_+ \mu_- (\perm_+\mu_-+\perm_-\mu_+) \nonumber\\
		& = \ri k (\perm_+ - \perm_-)(\omega^2\perm\psi_1)(0)+\perm_+ \mu_-^2 \int_0^\infty e^{ -\mu_+t}\omega^2\perm_+\psi_2(t)\dd t + \perm_- \mu_+^2 \int_{-\infty}^0 e^{ \mu_- t} \omega^2\perm_-\psi_2(t)\dd t,
	\end{eqnarray}
	which in the proof of Proposition \ref{T:resolv-1D} led to the expression \eqref{E:C2} for $C_2(k)$. Here, however, the left hand side of \eqref{E:C2b} is zero by \eqref{E:om-ev-cond}.

	Next, Lemma \ref{L:efn-1D} and \eqref{E:om-ev-cond} give
	$$(\perm\psi_1)(x_1)=\begin{cases} -\ri k\perm_- e^{-\mu_+ x_1},  & x_1>0,\\
		-\ri k\perm_- e^{\mu_- x_1},  & x_1<0,
	\end{cases} \quad (\perm\psi_2)(x_1)=\begin{cases} -\mu_+\perm_- e^{-\mu_+ x_1},  & x_1>0,\\
		\mu_-\perm_- e^{\mu_- x_1},  & x_1<0,
	\end{cases}$$
	so \eqref{E:C2b} implies
	\begin{eqnarray*}
		0&=& k^2(\perm_+-\perm_-)\perm_--\perm_+\mu_-^2\frac{\perm_-}{2}+\perm_-\mu_+^2\frac{\perm_-}{2}\\
		&=& \perm_-\left[k^2(\perm_+-\perm_-)-\frac12 \perm_+(k^2-\omega^2\perm_-)+\frac12\perm_-(k^2-\omega^2\perm_+)\right]\\
		&=&\frac{k^2}{2}\perm_-(\perm_+-\perm_-),
	\end{eqnarray*}
	contradicting Remark \ref{R:Wp-neq-pWm}.

	This concludes the proof of Proposition \ref{T:pt-spec-1D}.
	\epf

	\brem
	The fact that $\psi_3=0$ in Lemma \ref{L:psi3_zero} implies (for $\psi=E$) that $H=\frac{\ri}{\omega}\nabla \times E=(0,0,h(x_1))^Te^{\ri kx_2}$, i.e.~the eigenfunctions are TM-modes.
	\erem

	%-----------------------------------------------------------------------------------------------
	\subsection{Weyl Spectrum of $\cL_k$}\label{S:ess-sp-1D}
	Recall that
	$$\sweyl(\cL_k)=\{\omega \in D(\perm): \exists (u^{(n)})\subset \cD_{k,\omega}: \|u^{(n)}\|_{L^2(\R)}=1 \ \forall n\in\N, u^{(n)} \rightharpoonup 0, \|L_k(\omega)u^{(n)}\|_{L^2(\R)}\to 0 \ (n\to\infty)\}.$$

	\bprop\label{L:Weyl_sp_Tk}
	Let $k\in \R$. Then
	$$
	\begin{aligned}
		\sweylr(\cL_k) \supset M_+^{(k)}\cup M_-^{(k)}=
		\begin{cases}
			\{\omega \in D(\perm): \omega^2\perm_+(\omega)\in [k^2,\infty) \text{  or  } \omega^2\perm_-(\omega)\in [k^2,\infty)\} \ &\text{if } \ k\neq 0,\\
			\{\omega \in D(\perm): \omega^2\perm_+(\omega)\in (0,\infty) \text{  or  } \omega^2\perm_-(\omega)\in (0,\infty)\} \ &\text{if } \ k= 0.
		\end{cases}
	\end{aligned}
	$$
	\eprop
	\bpf
	We prove $\{\omega \in D(\perm): \omega^2\perm_+(\omega)\in [k^2,\infty)\}\subset \sweyl(\cL_k) $ for any $k\in \R$ in detail using a Weyl sequence $(u^{(n)})$ with the support moving to $x_1\to+\infty$. The other part is proved analogously with a sequence with the support moving to $x_1\to -\infty$.

	The most natural choice of a Weyl sequence is one given by a plane wave (in $x_1$) smoothly cut off to have compact support and with the support moving out to $x_1\to +\infty$. Such sequences do not see the interface (for $n$ large enough), hence the interface conditions can be ignored in their construction.

	%for a Weyl sequence with nonzero first two components see v. 33
	Although Weyl sequences with non-zero first and second components exist, we choose a much simpler sequence $(u^{(n)})$ with $u^{(n)}_1=u^{(n)}_2=0$ in the spirit of Remark \ref{R:block1D}.

	If $k^2-\omega^2\perm_+(\omega)\leq 0$, simple plane-wave solutions for $x_1>0$ are
	\beq\label{E:plane-wave1D}
	e^{\ri l_j x_1} \bspm 0\\0\\1\espm  \quad \text{with}\quad l_j = (-1)^j\sqrt{\omega^2\perm_+(\omega)-k^2}\in \R, \quad j =1,2.
	\eeq
	Choosing freely $j\in \{1,2\},$ we set $l:=l_j$ and
	$$u^{(n)}(x_1):=\frac{e^{\ri l x_1}}{\sqrt{n}} \varphi\left(\frac{x_1-n^2}{n}\right)\bspm 0 \\ 0 \\1\espm,$$
	where $\varphi \in C^\infty_c(\R), \|\varphi\|_{L^2(\R)}=1$. To check that $(u^{(n)})\subset \cD_{k,\omega}$, note that the divergence condition and all regularity conditions hold trivially. The interface conditions can be ignored for $n$ large enough as explained above. Also $\|u^{(n)}\|_{L^2(\R)}=1$ is satisfied due to the normalization of $\varphi$.

	Let us now check that  $\|L_ku^{(n)}\|_{L^2(\R)}\to 0$. For $n$ large enough
	$L_ku^{(n)}=(0,0,-u^{(n)''}_3-l^2 u^{(n)}_3)^T$ because $u^{(n)}\omega^2\perm=u^{(n)}\omega^2\perm_+$ for $n$ large enough. We have
	$$-u^{(n)''}_3-l^2 u^{(n)}_3 = -2\ri l n^{-3/2}e^{\ri l x_1} \varphi'\left(\frac{x_1-n^2}{n}\right)-n^{-5/2}e^{\ri l x_1} \varphi''\left(\frac{x_1-n^2}{n}\right)$$
	and hence
	$$
	\|L_ku^{(n)}\|_{L^2(\R)}\leq c(n^{-1}\|\phi'\|_{L^2(\R)}+n^{-2}\|\phi''\|_{L^2(\R)})\to 0.
	$$
	Finally, to show $u^{(n)}\rightharpoonup 0$, let $\eta \in L^2(\R,\C^3)$ be arbitrary.
	$$
	\begin{aligned}
		\left|\int_\R \overline{\eta}^Tu^{(n)}\dd x_1\right| &=\left|n^{-1/2}\int_\R e^{\ri l x_1}\varphi\left(\frac{x_1-n^2}{n}\right)\overline{\eta}_3(x_1)\dd x_1 \right|\\
		&\leq n^{-1/2}\int_{[n,\infty)} \left|\varphi\left(\frac{x_1-n^2}{n}\right)\right| |\eta_3(x_1)|\dd x_1\\
	\end{aligned}
	$$
	for $n$ large enough. And hence
	$$\left|\int_\R \overline{\eta}^Tu^{(n)}\dd x_1\right| \leq \|\varphi\|_{L^2(\R)}\|\eta_3\|_{L^2([n,\infty))}\to 0.$$

	Finally, we discuss the case $k=0$. If $\omega^2\perm_+(\omega)=0$, then the first equation in $T_0\psi-\omega^2\perm(\omega)\psi=0$ reduces to $0=0$, implying  that
	$$u^{(n)}(x_1):=\bspm f(x_1-n)\\ 0 \\0\espm$$ is a Weyl sequence for any $f\in C_c(\R)$ with $\|f\|=1$ (and for $n$ large enough such that $\text{supp}f(\cdot-n)\subset (0,\infty)$). Note that in this case $L_0(u^{(n)})=0$ for all $n$.

	Assume next that $\omega^2\perm_+(\omega)> 0$. The plane waves on $x_1>0$ have the form
	$$
	\xi^{(j)}e^{\ri l_j x_1} \quad \text{with}\quad l_j = (-1)^j\sqrt{\omega^2\perm_+(\omega)}\in \R, \xi^{(j)}\in \text{span}\left\{\bspm 0\\1\\0\espm, \bspm 0\\0\\1\espm\right\}, \quad j =1,2.
	$$
	Let us choose $j=1$ with $\xi^{(1)}:=\bspm 0\\1\\0\espm$ and set $l:=l_1$. A Weyl sequence can be chosen as
	$$u^{(n)}(x_1):=c \frac{e^{\ri l x_1}}{\sqrt{n}} \varphi\left(\frac{x_1-n^2}{n}\right)\xi^{(1)}, \quad \varphi \in C^\infty_c(\R,\R)$$
	with $c:=\|\varphi\|^{-1}$, such that $\|u^{(n)}\|=1$ for all $n$. Clearly, also $u^{(n)}\in D(L_0(\omega))$ holds.

	Next, we check that $L_0(\omega)u^{(n)} \to 0$ in $L^2(\R,\C^3)$ as $n\to \infty$. For $n$ large enough the support of $u^{(n)}$ is disjoint from $(-\infty,0)$ so that $\omega^2\perm u^{(n)}=\omega^2\perm_+u^{(n)}$ and we get
	$$L_0(\omega)u^{(n)}= -2\ri lcn^{-3/2}e^{\ri lx_1}\varphi'\left(\frac{x_1-n^2}{n}\right)\bspm 0\\ 1 \\ 0\espm -cn^{-5/2}e^{\ri lx_1}\varphi''\left(\frac{x_1-n^2}{n}\right)\bspm 0\\ 1 \\ 0\espm.$$
	Hence
	$$\|L_0(\omega)u^{(n)}\|\leq c\left(n^{-1}\|\varphi'\|+n^{-2}\|\varphi''\|\right)\to 0.$$
	Finally, to show $u^{(n)}\rightharpoonup 0$ in $L^2(\R,\C^3)$, let $\eta \in L^2(\R,\C^3)$ be arbitrary. We have
	$$
	\begin{aligned}
		\left|\int_\R \overline{\eta}^Tu^{(n)}\dd x_1\right| &=\frac{c}{\sqrt{n}}\left|\int_\R e^{\ri l x_1}\varphi\left(\frac{x_1-n^2}{n}\right)\overline{\eta}_2(x_1)\dd x_1\right| \\
		&\leq \frac{c}{\sqrt{n}}\int_n^\infty \left|\varphi\left(\frac{x_1-n^2}{n}\right)\right| |\eta_2(x_1)|\dd x_1 \quad \text{for} \ n\ \text{large enough}\\
		&\leq c \|\varphi\|_{L^2(\R)} \|\eta_2\|_{L^2((n,\infty))} \to 0.
	\end{aligned}
	$$
	\epf

	Proposition \ref{L:Weyl_sp_Tk} shows one inclusion in the second equality in \eqref{E:weyl-sp1D}. The other inclusion will be proved in Section \ref{S:1d-pt4}.

	%------- --------------------------------------------
	\subsection{Proof of \eqref{E:disc-sp1D} in Theorem \ref{T:main1D}: Discrete Spectrum of $\cL_k$}\label{S:disc-sp-1D}

	As we show next, outside the exceptional set $\Omega_0$ finite multiplicity eigenvalues of $\cL_{k}$ constitute the discrete spectrum, i.e. the eigenvalue $\lambda=1$ of \eqref{E:gen-sp-prob-1D} is isolated from the rest of the spectrum of the generalized eigenvalue problem
	\beq\label{E:gen-sp-prob-1D}
	T_k(\pa_{x_1})u = \lambda \omega^2\perm(x_1,\omega)u.
	\eeq

	\bprop\label{T:disc_sp_1D}
	\beq\label{E:disc-sp1Dknz-omnz}
	\sigma_d^{\rm{red}}(\cL_{k})=N^{(k)}
	\eeq
	for all $k\in \R$. In summary (together with \eqref{E:pt-sp1D} of Theorem \ref{T:main1D}) we have $\sigma_d^{\rm{red}}(\cL_k)=\sigma_p^{(<\infty)\ \rm{red}}(\cL_k)$ for all $k\in \R$.
	\eprop
	\bpf
	By \eqref{E:pt-sp1D} the set on the right hand side of \eqref{E:disc-sp1Dknz-omnz} is the set $\sigma_p^{(<\infty)\ \rm{red}}(\cL_{k})$ and hence the inclusion $\subset$ in \eqref{E:disc-sp1Dknz-omnz} holds. To show $\supset$, let $\omega \in N^{(k)}$. It remains to check that $\omega$ is an isolated eigenvalue, i.e. that $\lambda=1$ is an isolated eigenvalue of \eqref{E:gen-sp-prob-1D}. Hence, we need the existence of $\delta>0$ such that
	$$\lambda \in B_\delta(1)\setminus\{1\}\subset \C \ \Rightarrow \ \lambda \in \rho(L_k(\omega;\cdot)).$$
	We use the following equivalences
	$$ \lambda \in \rho(L_k(\omega;\cdot)) \ \Leftrightarrow \  1 \in \rho(L_k^{(\lambda)}(\omega;\cdot))  \ \Leftrightarrow \  \omega\in \rho(\cL_k^{(\lambda)}),$$
	where $L_k^{(\lambda)}(\omega;\mu):=T_k(\pa_1)-\mu\lambda \omega^2\perm(x_1,\omega)$ and $\cL_k^{(\lambda)}$ is the corresponding pencil with the pencil parameter $\mu$ and with $\lambda \in \C$ fixed.

	Due to Proposition \ref{T:resolv-1D} we have
	$$\rho(\cL_k)\supset D(\perm)\setminus (\Omega_0\cup M_+^{(k)} \cup M_-^{(k)} \cup N^{(k)}).$$
	Hence also
	$$\rho(\cL_k^{(\lambda)})\supset D(\perm)\setminus (\Omega_0\cup M_+^{(k,\lambda)} \cup M_-^{(k,\lambda)} \cup N^{(k,\lambda)}) \quad \forall \lambda \in\C\setminus\{0\},$$
	where
	$$
	\begin{aligned}
		M_\pm^{(k,\lambda)} &:= \{\omega\in D(\perm): \lambda \omega^2\perm_\pm(\omega)\in [k^2,\infty)\}\\
		N^{(k,\lambda)} &:= \{\omega\in D(\perm): \lambda \omega^2\perm_\pm(\omega)\notin \{0\}\cup[k^2,\infty), k^2(\perm_+(\omega)+\perm_-(\omega))=\lambda \omega^2\perm_+(\omega)\perm_-(\omega)\}.
	\end{aligned}
	$$
	Therefore, it suffices to show that there is a $\delta>0$ such that
	$$\text{if }\lambda\in B_\delta(1)\setminus \{1\}\subset \C, \text{ then } \omega\notin \Omega_0\cup M_+^{(k,\lambda)} \cup M_-^{(k,\lambda)} \cup N^{(k,\lambda)}.$$
	For this note that $\omega \notin \Omega_0$ because $\omega\in N^{(k)}$ and $N^{(k)}\cap \Omega_0=\emptyset$ and that  $\omega \notin N^{(k,\lambda)}$ because
	$$k^2=\frac{\omega^2\perm_+(\omega)\perm_-(\omega)}{\perm_+(\omega)+\perm_-(\omega)} \neq \lambda \frac{\omega^2\perm_+(\omega)\perm_-(\omega)}{\perm_+(\omega)+\perm_-(\omega)}$$
	if $\lambda \neq 1$. Finally, $\omega \notin M_\pm^{(k,\lambda)}$ for $\lambda\in B_\delta(1)\setminus \{1\}$ with $\delta$ small enough because $\omega \notin M_\pm^{(k)}$, i.e., $\omega^2\perm_\pm(\omega)\in \C\setminus [k^2,\infty)$, which is an open set, and hence $\lambda \omega^2\perm_\pm(\omega)\in \C\setminus [k^2,\infty)$ for $\lambda\in \C$ close enough to $1$. This is correct also for $k=0$ since $\omega^2\perm_\pm(\omega)\neq 0$ by assumption.
	\epf

	%-----------------------------------------------------------------------
	\subsection{Proof of \eqref{E:sp1D} and \eqref{E:weyl-sp1D} in Theorem \ref{T:main1D}: Composition of the Spectrum of $\cL_k$}\label{S:1d-pt4}

			In Proposition \ref{L:Weyl_sp_Tk} we showed one inclusion in \eqref{E:weyl-sp1D}. In the following proposition we show the rest of \eqref{E:weyl-sp1D}.
			\bprop\label{P:spec-compos}
			Let $k\in \R$. Then
			$$
			\sigma^{{\rm red}}(\cL_k)=\sweylr(\cL_k)~\dot{\cup}~\sigma^{{\rm red}}_d(\cL_k), \quad \sfiver(\cL_k)=\sweylr(\cL_k)=\soner(\cL_k)=M_+^{(k)} \cup M_-^{(k)},
			$$
			where $\dot{\cup}$ denotes the disjoint union.
			\eprop

			\bpf
			By Propositions \ref{T:resolv-1D} and \ref{T:disc_sp_1D}, we have that
			\be\label{eq:red1} \sigma^{{\rm red}}(\cL_k)\subseteq (M_+^{(k)} \cup M_-^{(k)}) \dot{\cup} N^{(k)}, \quad  \sigma_d^{\rm{red}}(\cL_{k})=N^{(k)}. \ee
			Therefore, by Remark \ref{R:disc-sfive},
			$$\sfiver(\cL_k)=\sigma^{{\rm red}}(\cL_k)\setminus\sigma_d^{\rm{red}}(\cL_{k}) \subset  M_+^{(k)} \cup M_-^{(k)}.$$ Since $\sweylr(\cL_k)=\stwor(\cL_k)\subset \sfiver(\cL_k)$, with Proposition \ref{L:Weyl_sp_Tk} we get  $\sweylr(\cL_k)=\sfiver(\cL_k)=  M_+^{(k)} \cup M_-^{(k)}.$ Using again Remark
			\ref{R:disc-sfive} and the equality $\sigma_d^{\rm{red}}(\cL_{k})=N^{(k)}$ from \eqref{eq:red1},  we
			get
			$$\sigma^{{\rm red}}(\cL_k) =\sfiver(\cL_k) \cup \sigma_d^{\rm{red}}(\cL_{k})=  (M_+^{(k)} \cup M_-^{(k)}) \dot{\cup} N^{(k)},$$
			i.e. the equality in \eqref{eq:red1} holds also.

			Thus, it remains to show that $\stwor(\cL_k)\subseteq\soner(\cL_k)$.  By Proposition \ref{T:pt-spec-1D}, for $\omega\in \sigma^{{\rm red}}(\cL_k)$, we have that $\ker(L_k(\omega))$ is finite-dimensional. Therefore, if $\omega\in\stwor(\cL_k)$, then $L_k(\omega)$ cannot be semi-Fredholm. This implies that $\omega\in\soner(\cL_k)$.
			\epf

			We note that Proposition \ref{P:spec-compos} completes the proof of Theorem \ref{T:main1D}, where the equality $\soner(\cL_k)=\stwor(\cL_k)=\sthreer(\cL_k)=\sfourr(\cL_k)=\sfiver(\cL_k)=\sweylr(\cL_k)$ follows because $\sone\subset\stwo\subset \sthree\subset \sfour\subset\sfive$ always holds. As a result, all forms of the reduced essential spectrum coincide.
			%
			%Finally, we have to show the remaining inclusion in \eqref{E:weyl-sp1D} not proved in Section \ref{S:ess-sp-1D}. \textcolor{blue}{However, this follows immediately from
				%Lemma \ref{L:Weyl_sp_Tk} and Proposition \ref{T:disc_sp_1D}, together with Proposition \ref{T:resolv-1D}, due to the disjointness of the reduced Weyl-spectrum and the reduced discrete spectrum.}

			\subsection{Proof of Theorem \ref{TheoPlum2}}

			In this subsection, we give a full proof of Theorem \ref{TheoPlum2}.

			\begin{proof}  {\bf Part a)} Let $k \in \R \setminus \{ 0 \}$. First we address \eqref{E:1Dspec-Om0}, \eqref{Plum2}, \eqref{Plum3}, and \eqref{E:1Des1-Om0}.\newline\indent
				With
				$$S_{\rho}:=\{ \omega \in \Omega_0 : \omega^2\perm_+ (\omega) = \omega^2\perm_- (\omega) = 0, ~\perm_+ (\omega) \neq 0, ~ \perm_- (\omega) \neq 0 ,~ \perm_+ (\omega) + \perm_- (\omega) \neq 0\},$$
				i.e. the complement (within $\Omega_0$) of the  right-hand side of \eqref{E:1Dspec-Om0},
				and with $S_p^{(<\infty)}, S_p^{(\infty)}, S_{e,1}, S_{p, \overline{e,1}}^{(\infty)}$ denoting the right-hand sides of \eqref{Plum2}, \eqref{Plum3}, \eqref{E:1Des1-Om0}, \eqref{Plum6}, respectively, we will prove further below the set inclusions
				\begin{enumerate}
					\item[A)] $\rho (\cL_k) \cap \Omega_0 \supset S_{\rho}$,
					\item[B)] $(\sigma_p^{(<\infty)} (\cL_k) \cap \Omega_0 ) \setminus (\sigma_{e,1} (\cL_k) \cap \Omega_0 ) \supset S_p^{(<\infty)}$,
					\item[C)] $\sigma_p^{(\infty)} (\cL_k) \cap \Omega_0 \supset S_p^{(\infty)}$,
					\item[D)] $\sigma_{e,1} (\cL_k) \cap \Omega_0 \supset S_{e,1}$,
					\item[E)] $(\sigma_p^{(\infty)} (\cL_k) \cap \Omega_0 ) \setminus
					(\sigma_{e,1} (\cL_k) \cap \Omega_0 )  \supset S_{p, \overline{e,1}}^{(\infty)}$.
				\end{enumerate}

				With A), ..., E) at hand, we can proceed with the proof as follows: The sets on the left-hand sides in A), B), D), E) are pairwise disjoint, and the union of the sets on the right-hand sides in A), B), D), E) is the whole of $\Omega_0$ (to see this, distinguish the cases $\omega \neq 0$ and $\omega =0$). Hence equality holds in A), B), D), E), which already implies \eqref{E:1Dspec-Om0} and \eqref{E:1Des1-Om0}. Moreover, using C), we get
				\be \label{Plum7}
				S_{e,1} \cup S_{p, \overline{e,1}}^{(\infty)} = S_p^{(\infty)} \subset\sigma_p^{(\infty)} (\cL_k) \cap \Omega_0,
				\ee
				and hence
				\be\label{Plum7a} \sigma_{e,1} (\cL_k) \cap \Omega_0 \, \underset{\eqref{E:1Des1-Om0}}{=} \, S_{e,1} \subset \sigma_p^{(\infty)} (\cL_k) \cap \Omega_0\ee
				which implies, together with the equality in D) and E), that
				$$ S_{e,1} \cup S_{p, \overline{e,1}}^{(\infty)} = \sigma_p^{(\infty)} (\cL_k) \cap \Omega_0 , $$
				and hence \eqref{Plum7} gives \eqref{Plum3}.\newline\indent
				Moreover, \eqref{Plum3} and \eqref{E:1Des1-Om0} (or \eqref{Plum7a}) imply $\sigma_{e,1} (\cL_k) \cap \Omega_0 \subset  \sigma_p^{(\infty)} (\cL_k) \cap \Omega_0$, whence
				$$ \sigma_{e,1} (\cL_k) \cap \sigma_p^{(<\infty)} (\cL_k) \cap \Omega_0 = \emptyset, $$
				and thus
				$$ (\sigma_p^{(< \infty)} (\cL_k) \cap \Omega_0 ) \setminus (\sigma_{e,1} (\cL_k) \cap \Omega_0) = \sigma_p^{(<\infty)} (\cL_k) \cap \Omega_0,$$
				which together with the equality in B) implies \eqref{Plum2}.\newline\indent
				The statement about the Fredholm property of of $\omega=0$ will also be shown within our proof of B) below.\newline\indent
				To prove \eqref{E:1Ddiscsp-Om0}, we first note, using \eqref{Plum2}, that  $\sigma_d (\cL_k) \cap \Omega_0 \subset  \sigma_p^{(<\infty)} (\cL_k) \cap \Omega_0 = \{ 0\}$ if $\perm_+ (0) \neq 0, ~ \perm_- (0) \neq 0, ~ \perm_+ (0) + \perm_- (0) = 0$, and is empty otherwise. Also in the first case we obtain $\sigma_d (\cL_k) \cap \Omega_0 = \emptyset$ since $0 \notin \sigma_d (\cL_k) \cap \Omega_0$, as for $\omega = 0$ we have $\omega^2\perm_+ (\omega) = \omega^2\perm_- (\omega) = 0$ and hence {\it every} $\lambda \in \C$ is an eigenvalue of problem \eqref{E:gen-sp-prob-1D}, always with the same associated eigenfunction
				$$ \left( u_1, \frac{\ri}{k} u_1', 0 \right), ~ \text{ where } u_1 (x) = \left\{ \begin{array}{ll} e^{-|k|x} & \text{on } \R_+, \\
					-e^{|k|x} & \text{on } \R_-; \end{array} \right.  $$
				see also the proof of B) below. In particular, the eigenvalue $\lambda = 1$ of problem \eqref{E:gen-sp-prob-1D} is not isolated, whence indeed $\sigma_d (\cL_k) \cap \Omega_0  = \emptyset$. \newline\indent
				Now we address \eqref{E:1DWeyl-Om0}, \eqref{E:1Des1-Om0}, and \eqref{E:1Des5-Om0}. By Remark \ref{R:disc-sfive}, \eqref{E:1Ddiscsp-Om0} implies $\sigma_{e,5} (\cL_k) \cap \Omega_0 = \sigma (\cL_k) \cap \Omega_0 = \Omega_0 \setminus (\rho (\cL_k) \cap \Omega_0 )$, which on one hand gives
				\eqref{E:1Des5-Om0} and
				\beq\label{E:sig-sigp}
				\sigma(\cL_k)  \cap \Omega_0 = (\sigma_p^{(\infty)} (\cL_k) \cap \Omega_0) \; \dot{\cup} \; (\sigma_p^{(< \infty)}(\cL_k) \cap \Omega_0 ),
				\eeq
				where \eqref{E:sig-sigp} follows because $\Omega_0 = S_{\rho} \; \dot{\cup} \; S_p^{(< \infty)} \; \dot{\cup} \; S_p^{(\infty)}$.\newline\indent
				We are left to prove \eqref{E:1DWeyl-Om0}. From the definition of $\sigma_{e,4}$ and because $L_k(0)$ is Fredholm of index $0$, we obtain
				$$ \sigma_{e,4} (\cL_k) \cap \sigma_p^{(< \infty)} (\cL_k) \cap  \Omega_0 = \emptyset$$
				and therefore, using \eqref{E:sig-sigp},
				\be \label{Plum8}
				\sigma_{e,4} (\cL_k) \cap  \Omega_0 \subset \sigma_p^{(\infty)} (\cL_k) \cap  \Omega_0 .
				\ee
				Furthermore,
				\be \label{Plum8a}
				\sigma_p^{(\infty)} (\cL_k) \cap  \Omega_0 \subset \sigma_{\text{Weyl }} (\cL_k) \cap  \Omega_0,
				\ee
				since any $L^2$-orthonormal sequence in the eigenspace of $\omega \in \sigma_p^{(\infty)} (\cL_k)$ forms a Weyl sequence. Since moreover $\sigma_{e,2} (\cL_k) = \sigma_{\text{Weyl }} (\cL_k)$ by Remark \ref{rem:essential}, and  $\sigma_{e,2} (\cL_k) \subset \sigma_{e,3} (\cL_k) \subset \sigma_{e,4} (\cL_k)$ by Remark \ref{rem:essinc}, we obtain the desired equality chain. This completes the proof of part a), after we have shown the inclusion statements A), ..., E), which we will do now.\newline\indent
				Recall from \eqref{E:Dkom} that, for $u  \in L^2 (\R, \C^3)$,
				\be \label{Plum9}
				u \in \cD_{k,\omega} \Leftrightarrow \left\{ \begin{array}{l}
					u_2 \in H^1 (\R), ~ u_2' - \ri k u_1 \in H^1 (\R),\\ u_3 \in H^2 (\R), \\ \perm u_1 \in H^1 (\R), ~ (\perm u_1)' + \ri k \perm u_2 = 0 \end{array} \right\} ,
				\ee
				and
				\be \label{Plum10}
				L_k (\omega) u = \left( \begin{array}{l}
					(k^2 - \omega^2\perm) u_1 + \ri ku_2' \\ \ri k u_1' - u_2'' - \omega^2\perm u_2 \\ - u_3'' + (k^2 - \omega^2\perm) u_3 \end{array} \right) .
				\ee
				{\bf Ad A) and B).} Let $\omega \in S_{\rho} \cup S_p^{(<\infty)}$, whence
				\be \label{Plum11}
				\omega^2\perm_+ (\omega) = \omega^2\perm_- (\omega) = 0, ~ \perm_+ (\omega) \neq 0, ~ \perm_- (\omega) \neq 0
				\ee
				(which can hold only for $\omega = 0$). For $r \in \mathcal{R}_k$ and $u \in \cD_{k,\omega}$, the equations $L_k (\omega) u = r$ now read, by \eqref{Plum10},
				\begin{align}
					k^2 u_1 + \ri k u_2'  &=  r_1,  \label{Plum12}\\
					\ri k u_1' - u_2''  &=  r_2,  \nn \\
					-u_3'' + k^2 u_3  &=  r_3. \label{Plum13}
				\end{align}
				The second equation can be dropped since it follows from \eqref{Plum12}; note that $r_1' + \ri k r_2 = 0$. First we note that \eqref{Plum13} has a unique solution $u_3 \in H^2 (\R)$. By \eqref{Plum11}, the third line of \eqref{Plum9} reads
				\begin{align}
					u_1 \in H^1 (\R_{\pm}), ~ u_1' + \ri ku_2 = 0 \text{ on } \R_{\pm}, \label{Plum14} \\
					\perm_- u_1 (0-) = \perm_+ u_1 (0+). \label{Plum15}
				\end{align}
				\eqref{Plum12} and \eqref{Plum14}, together with the first line of \eqref{Plum9}, are equivalent to
				\be\label{Plum16}
				\ba
				u_1 \in H^3 (\R_{\pm}), ~ u_2 \in H^2 (\R_{\pm}) \cap H^1 (\R), ~ u_2' - \ri ku_1  \in H^1 (\R), \\
				k^2 u_1 + \ri k u_2' = r_1 \text{ on } \R, ~ u_2 = \frac{\ri}{k} u_1' \text{ on } \R_{\pm}.
				\ea
				\ee
				Inserting the last equation into the one before gives
				\be\label{Plum17}
				-u_1'' + k^2 u_1 = r_1 \text{ on } \R_{\pm}.
				\ee
				Its general solution, subject to condition \eqref{Plum15}, reads in the case $k > 0$ (otherwise replace $k$ by $-k$):
				\beq\label{Plum18}
				\begin{split}
					u_1 (x)  &=  \left[ C \perm_- - \frac{1}{2k} \int\limits_{0}^{\infty} e^{-kt} r_1 (t) dt \right] e^{-kx} + \frac{1}{2k}  \left[ e^{kx} \int\limits_{x}^{\infty} e^{-kt} r_1 (t) dt + e^{-kx} \int\limits_{0}^{x} e^{kt} r_1 (t) dt \right] \text{ on } \R_+, \\ u_1 (x) & =  \left[ C \perm_+ -\frac{1}{2k} \int\limits_{- \infty}^{0} e^{kt} r_1 (t) dt \right] e^{kx} + \frac{1}{2k}  \left[ e^{kx} \int\limits_{x}^{0} e^{-kt} r_1 (t) dt + e^{-kx} \int\limits_{-\infty}^{x} e^{kt} r_1 (t) dt \right] \text{ on } \R_-,
				\end{split}
				\eeq
				with $C \in \C$ denoting a free constant. We define (as required by the last condition in \eqref{Plum16})
				\be\label{Plum19}
				u_2 := \frac{\ri}{k} u_1' \text{ on } \R_{\pm}
				\ee
				which gives $u_2 \in H^2 (\R_{\pm})$ and  $u_2' - \ri k u_1 = - \frac{\ri}{k} r_1\in H^1(\R)$, implying also the last-but-one condition in \eqref{Plum16}. In order to get $u_2 \in H^1 (\R)$, \eqref{Plum19} shows that we need $u_1' (0-) = u_1' (0+)$, which by \eqref{Plum18} reads
				$$k C \perm_+ -\int\limits_{- \infty}^{0} e^{kt} r_1 (t) dt = - k C \perm_- + \int\limits_0^{\infty} e^{-kt} r_1 (t) dt,$$
				i.e.
				\be\label{Plum20}
				C (\perm_+ + \perm_-) = \frac{1}{k} \int\limits_{- \infty}^{\infty} e^{-k|t|} r_1 (t) dt .
				\ee

				{\bf Ad A).} Let $\perm_+ + \perm_- \neq 0$.\newline\indent
				Then \eqref{Plum20} provides a unique value for $C$ and hence altogether a unique solution $u \in \cD_{k,\omega}$ to $L_k (\omega) u = r$. Hence, $\omega (=0) \in \rho (\cL_k)$, which proves A).\newline\indent
				{\bf Ad B).} Let $\perm_+ + \perm_- = 0$.\newline\indent
				Then \eqref{Plum20} holds if and only if
				\be \label{Plum21}
				r_1 \in \left\{ e^{-k|x|} \right\}^{\bot_{L^2}};
				\ee
				the constant $C$ remains free in this case, and \eqref{Plum19} defines $u_2 \in H^2 (\R_{\pm}) \cap H^1 (\R)$. Thus, \eqref{Plum16} holds. \newline\indent
				With $P:L^2 (\R, \C^3) \to \mathcal{R}_k$ denoting the orthogonal projection onto $\mathcal{R}_k$, we obtain, for $r \in \mathcal{R}_k$,
				$$\langle r, (e^{-k|x|}, 0 , 0 )\rangle_{L^2} = \langle Pr, (e^{-k|x|}, 0 , 0 )\rangle_{L^2} = \langle r, P(e^{-k|x|}, 0 , 0) \rangle_{L^2},$$
				and $P(e^{-k|x|}, 0 , 0) \neq 0$ since $(e^{-k|x|}, 0 , 0)$ is not orthogonal to $\mathcal{R}_k$ as there are functions $r\in \cR_k$ such that $\int_\R r_1(s)e^{-|k|s}\dd s\neq 0$. Thus, by \eqref{Plum21},
				$$\Ran L_k (\omega) = \{ P (e^{-k|x|}, 0 , 0) \}^{\bot_{\mathcal{R}_k}}$$
				is closed and has co-dimension one. Moreover, by \eqref{Plum18} and \eqref{Plum19}, and since $C$ is free,
				\beq
				\ker L_k (\omega) = \text{span } \left\{ (u_1, u_2, 0) \in \cD_{k,\omega}: u_1 (x) = \left\{ \begin{array} {l} e^{-kx} \text{ on } \R_+ \\ - e^{kx} \text{ on } \R_- \end{array} \right\} , u_2 = \frac{\ri}{k} u_1' \right\}.
				\eeq
				Thus, $\omega (=0)$ is a geometrically simple eigenvalue. By \eqref{Plum11} and the definitions at the beginning of Section \ref{S:main-res-1D} equation \eqref{E:gen-evec-eq}, with $\lambda=1$ and $v\in \ker(A_k(\omega)-B(\omega))\setminus\{0\}$, reads again $L_k(\omega)u=0$. Hence, due to the geometric simplicity, $u$ and $v$ must be linearly dependent, implying by definition that $\omega=0$ is also algebraically simple. Furthermore,
				$L_k (\omega)$ is clearly a Fredholm operator with index $0$, which proves B).
				\newline\indent
				{\bf Ad C).} Let $\omega \in S_p^{(\infty)}$. Here we consider the case $\perm_+ (\omega) =0$. The case $\perm_- (\omega) =0$ is treated analogously.\\
				Since also $\omega^2\perm_+ (\omega) = \omega^2 \perm_+ (\omega) =0$, we find that, for every $v \in H^1 (\R)$ with support in $\R_+$, $(v', \ri kv, 0) \in \ker L_k (\omega)$, which proves the assertion.\newline\indent
				{\bf Ad D).} Let $\omega \in S_{e,1}$. We consider only the case
				\be\label{Plum22}
				\perm_+ (\omega) = 0, ~ \omega^2\perm_- (\omega) \in [k^2,\infty),
				\ee
				since the alternative case can be treated analogously. First we prove that
				\be\label{Plum23}
				\ker L_k (\omega) \subset \{ u \in \cD_{k,\omega} : u = 0 \text{ on } \R_- \}.
				\ee
				Indeed, defining $a: = \sqrt{\omega^2\perm_- - k^2} \in [0,\infty)$, the equation $L_k (\omega) u = 0$ (with $u \in \cD_{k,\omega}$) reads
				\be\label{Plum24}
				\left. \begin{array} {rl} - a^2 u_1 + \ri k u_2' &= 0 \\ \ri k u_1' - u_2'' - (a^2 + k^2) u_2 &= 0 \\ - u_3'' - a^2 u_3 &= 0 \end{array} \right\} \text{ on } \R_- .
				\ee
				Furthermore, since $\omega^2\perm(\omega)_- \neq 0$ by \eqref{Plum22} and thus $\perm_- \neq 0$, the third line of \eqref{Plum9} implies $u_1 \in H^1 (\R_-)$ and $u_1' = - \ri k u_2$ on $\R_-$. Inserting this into the second equation in \eqref{Plum24} gives $-u_2'' - a^2 u_2 = 0$ on $\R_-$. Differentiating $u_1' = -  \ri k u_2$ provides $\ri k u_2' = - u_1''$, and inserting this into the first equation in \eqref{Plum24} implies $-u_1'' - a^2 u_1 = 0$ on $\R_-$. Thus, all three components solve $-v'' - a^2 v =$ on $\R_-$, with solution
				\beq
				v(x) = \left\{ \begin{array}{ll} A \cos (ax) + B \sin (ax) & \text{if } a > 0, \\
					A + Bx & \text{if } a = 0, \end{array} \right\}
				\eeq
				and thus $v \equiv 0$ on $\R_-$ since $v \in L^2 (\R_-)$. This implies \eqref{Plum23}.\newline\indent
				By \eqref{Plum23}, every $f \in L^2 (\R, \C^3)$ with $\supp f \subset \R_-$ is in $(\ker L_k (\omega))^{\bot_{L_2}}$. Now let $(u^{(n)})$ denote the Weyl sequence used in Lemma \ref{L:Weyl_sp_Tk}, but now with $\supp\ u^{(n)}$ moving to $-\infty$ instead of $+\infty$. Since $\supp \ u^{(n)} \subset \R_-$ and hence $u^{(n)} \in (\ker L_k (\omega))^{\bot}$ for $n$ sufficiently large by the above argument, we find that $\Ran L_k (\omega)$ is not closed in $\mathcal{R}_k$ by Lemma \ref{L:ran_not-closed}. Hence, $\omega \in \sigma_{e,1} (\cL_k)$.\newline\indent
				The condition $\perm_+ (\omega) = 0$ in \eqref{Plum22} has not been used in this proof of D). But actually it follows from $\omega \in \Omega_0$, i.e. $\omega^2\perm_+ (\omega) = 0$ or $\omega^2\perm_- (\omega) = 0$, and $\omega^2\perm_- (\omega) \in [k^2, \infty)$, implying $\omega^2\perm_- (\omega) \neq 0$. Hence, $\perm_+ (\omega) \neq 0$ would imply $\omega = 0$, contradicting $\omega^2\perm_- (\omega) \neq 0$.\newline\indent
				{\bf Ad E).} Let $\omega \in S_{p,\overline{e,1}}^{(\infty)}$. Again, we consider only one of the two cases, hence let
				\be\label{Plum25}
				\perm_+ (\omega) = 0, ~ \omega^2\perm_- (\omega) \notin [k^2,\infty).
				\ee
				$\perm_+ (\omega) = 0$ alone implies that $\omega \in \sigma_p^{(\infty)} (\cL_k) \cap \Omega_0$ by C). We are left to show that $\omega \notin \sigma_{e,1} (\cL_k)$ which we do by proving that $L_k (\omega)$ is onto, whence $L_k (\omega)$ is semi-Fredholm. So, let $r = (r_1, r_2, r_3) \in \mathcal{R}_k$ be given.\newline\indent
				If $\perm_- (\omega) = 0$ (implying $\omega^2\perm_- (\omega) = 0$), the problem $L_k (\omega) u = r$ is solved by $(\frac{1}{k^2} r_1, 0 , u_3) \in \cD_{k,\omega}$, with $u_3 \in H^2 (\R)$ denoting the unique solution of $-u_3 + k^2 u_3 = r_3$, whence $L_k (\omega)$ is indeed onto.\newline\indent
				Now let $\perm_- (\omega) \neq 0$. Together with $\perm_+ (\omega) = 0$ (implying $\omega^2\perm_+ (\omega) = 0$), the third line of \eqref{Plum9} reads
				\be\label{Plum26}
				u_1 \in H^1 (\R_-), ~ u_1 (0 -) = 0, ~ u_1' + \ri k u_2 = 0 \text{ on } \R_-,
				\ee
				and for $u \in \cD_{k,\omega}$ the equation $L_k (\omega) u = r$ reads, by \eqref{Plum10},
				\be\label{Plum27}
				\left. \begin{array}{rl}(k^2 - \omega^2\perm_-) u_1 + \ri k u_2' &= r_1 \\
					\ri k u_1' - u_2'' - \omega^2\perm_- u_2 &= r_2 \\
					- u_3'' + (k^2 - \omega^2\perm_-) u_3 &= r_3 \end{array} \right\} \text{ on } \R_-,
				\ee
				\be\label{Plum28}
				\left. \begin{array}{rl}k^2 u_1 + \ri k u_2' &= r_1 \\
					-u_3'' + k^2 u_3 &= r_3 \end{array} \right\} \text{ on } \R_+;
				\ee
				note that the equation for the second component on $\R_+$ can be dropped as it follows from the first. We set $\mu_- := \sqrt{k^2 - \omega^2\perm_-}$; by \eqref{Plum25} we have $\Re  \mu_- > 0$.\newline\indent
				The equations for $u_3$ in \eqref{Plum27}, \eqref{Plum28}, together with the interface conditions $\llbracket u_3\rrbracket  = \llbracket u_3'\rrbracket  = 0$, have a unique solution $u_3 \in H^2 (\R)$, which follows as in the proof of Proposition \ref{T:resolv-1D}, now with $\mu_+ := |k|$; see \eqref{E:u2plus}, \eqref{E:u2minus} (with $u_2, C_2(k), r_2$ replaced by $u_3, C_3 (k), r_3)$, and \eqref{Neq31}. The remaining equations in \eqref{Plum27}, \eqref{Plum28}, together with \eqref{Plum26} and the first line of \eqref{Plum9}, are equivalent to
				\be\label{Plum29}
				u_1 \in H^3 (\R_-), ~ u_1 (0-) = 0, ~ u_2 \in H^2 (\R_-) \cap H^1 (\R), ~ u_2' - \ri ku_1 \in H^1 (\R),
				\ee
				\be\label{Plum30}
				u_2' = - \frac{\ri}{k} [ r_1 - (k^2 - \omega^2\perm_-) u_1] \text{ on } \R_-, ~ u_1 = \frac{1}{k^2} r_1 - \frac{\ri}{k} u_2' \text{ on } \R_+,
				\ee
				\be\label{Plum31}
				\ri ku_1' - u_2'' - \omega^2\perm_- u_2  =  r_2 \quad \text{ on } \R_-,
				\ee
				\be\label{Plum32}
				u_1' + \ri k u_2  =  0 \quad \text{ on } \R_- .
				\ee
				Inserting $u_2 = \frac{\ri}{k} u_1'$ (resulting from \eqref{Plum32}) into the first equation in \eqref{Plum30} gives
				\be\label{Plum33}
				-u_1'' + (k^2 - \omega^2\perm_-) u_1 = r_1 \text{ on } \R_-,
				\ee
				which together with the boundary condition $u_1 (0-) = 0$ (see \eqref{Plum29}) has a unique solution $u_1 \in H^3 (\R_-)$, as $\Re \mu_- >0$ and $r_1\in H^1(\R)$ (because $r_1'=-\ri k r_2$). Now we define $u_2 \in H^2 (\R_-)$ by
				\be \label{Plum34}
				u_2 := \frac{\ri}{k} u_1' \text{ on } \R_-
				\ee
				and extend $u_2$ in an arbitrary way to a function $u_2 \in H^1 (\R)$. Finally, we define $u_1$ on $\R_+$ by the second equation in \eqref{Plum30}. From \eqref{Plum33}, \eqref{Plum34} it becomes clear that \eqref{Plum30}, \eqref{Plum32}, and (since $r_1' + \ri k r_2 = 0$) \eqref{Plum31} hold true, and that moreover
				\beq
				u_2' - \ri k u_1 = \left\{ \begin{array}{ll}
					\frac{\ri}{k} u_1'' - \ri k u_1 = - \frac{\ri}{k} (r_1 + \omega^2\perm_- u_1) & \text{ on } \R_-, \\
					u_2' - \ri k [ \frac{1}{k^2} r_1 - \frac{\ri}{k} u_2'] = - \frac{\ri}{k} r_1 & \text{ on } \R_+ \end{array} \right.
				\eeq
				is in $H^1 (\R)$ since $r_1 \in H^1 (\R)$ and $u_1 (0-) = 0$. Thus also \eqref{Plum29} holds true, and hence $L_k (\omega)$ is onto. We note that also the property $\omega \in \sigma_p^{(\infty)} (\cL_k)$ becomes visible again here by the arbitrariness of the extension of $u_2$ from $\R_-$ to $\R$.\newline\indent
				{\bf Part b)} Now let $k=0$.\newline\indent
				As in the case $k\neq 0$, the set $S_p^{(\infty)}$ given by the right-hand side of \eqref{Plum3} or \eqref{E:1Dptspinf-Om0k0} is contained in $\sigma_p^{(\infty)} (\cL_k) \cap \Omega_0$; see part C) of the proof. Below we show that
				\begin{enumerate}
					\item[F)] $\sigma_p (\cL_k) \cap (\Omega_0 \setminus S_p^{(\infty)}) = \emptyset$,
				\end{enumerate}
				which implies $\sigma_p (\cL_k) \cap \Omega_0 \subset S_p^{(\infty)}$ and hence \eqref{E:1Dptspinf-Om0k0}, as well as $\sigma_p^{(<\infty)} (\cL_k) \cap \Omega_0 = \emptyset$ and therefore also $\sigma_d (\cL_k) \cap \Omega_0 = \emptyset$, i.e. \eqref{E:1Dptsp-Om0k0} holds. Furthermore we prove
				\begin{enumerate}
					\item [G)] $\sigma_{e,1} (\cL_k) \cap \Omega_0 = \Omega_0$,
				\end{enumerate}
				which by Remarks \ref{rem:essinc} and \ref{rem:essential} implies all remaining equalities asserted in b), i.e. \eqref{E:1Dsp-Om0k0}, and  \eqref{E:1Dspec-ess-Om0k0}.

				{\bf Ad F).} Assume that some $\omega \in \sigma_p (\cL_k) \cap (\Omega_0 \setminus S_p^{(\infty)})$ exists. Then we have $L_k (\omega) u = 0$ for some $u \in \cD_{k,\omega} \setminus \{ 0 \}$ and
				\be\label{Plum35}
				\omega^2\perm_+ (\omega) = 0 \text{ or } \omega^2\perm_- (\omega) = 0,~\perm_+ (\omega) \neq 0, ~ \perm_- (\omega) \neq 0.
				\ee
				\eqref{Plum35} can hold for $\omega=0$ only, which in turn implies
				\be\label{Plum36}
				\omega^2\perm_+ (\omega) = \omega^2\perm_- (\omega) = 0.
				\ee
				Using \eqref{Plum35} and \eqref{Plum36}, equations \eqref{Plum9} and $L_k (\omega) u = 0$ (with $k = 0$) imply
				\beq
				u_1 \in H^1 (\R_{\pm}), ~ u_1' = 0 \text{ on } \R_{\pm}, u_2, u_3 \in H^2 (\R), ~ u_2'' = u_3''  = 0 \text{ on } \R,
				\eeq
				which only holds for $u_1 = u_2 = u_3 \equiv 0$, since $u \in L^2 (\R, \C^3)$. This contradicts our assumption.\newline\indent
				{\bf Ad G).} Let $\omega \in \Omega_0$. We assume that $\omega^2\perm_+ (\omega) = 0$; the case $\omega^2\perm_- (\omega) = 0$ is treated analogously. \eqref{Plum9} now reads
				\be\label{Plum37}
				u \in \cD_{k,\omega} \Leftrightarrow u_2, u_3 \in H^2 (\R), ~ \perm u_1 = 0
				\ee
				(note that $(\perm u_1)' = 0$ implies $\perm u_1 = 0$ since $\perm u_1 \in L^2 (\R)$), and \eqref{Plum10} implies
				\be\label{Plum38}
				L_k (\omega) u = \left( \begin{array}{c} 0\\ -u_2'' \\ - u_3'' \end{array} \right) \text{ on } \R_+,
				\ee
				whence in particular
				$$\ker L_k (\omega) \subset \{ u \in \cD_{k,\omega} : u_2 = u_3 = 0 \text{ on } \R_+ \}$$
				(since $u_2'' = u_3'' = 0$ on $\R_+$ implies $u_2 = u_3 = 0$ on $\R_+$). Thus,
				\be\label{Plum39}
				\{ f = (0,f_2, f_3)  \in L^2 (\R, \C^3): \text{ supp}f \subset \R_+\} \subset (\ker L_k (\omega))^{\bot_{L^2}}.
				\ee
				Now choose some $\varphi \in C_c^{\infty} (\R, \R)$ such that $\| \varphi \|_{L^2} = 1$, and define
				$$u_n (x) := \frac{1}{\sqrt{n}} \; \varphi \left( \frac{x-n^2}{n} \right) \cdot \left( \begin{array}{l} 0\\ 0\\ 1 \end{array} \right) \quad (x \in \R,~ n \in \N).$$
				Then $\| u_n\|_{L^2} = 1$ and, by \eqref{Plum37}, $u_n \in \cD_{k,\omega}$. Moreover, for $n$ sufficiently large,
				\be\label{Plum40}
				\text{supp } u_n \subset \R_+,
				\ee
				and therefore
				$$\| L_k (\omega) [u_n] \|^2_{L^2} = \| u_n'' \|_{L^2}^2 = \frac{1}{n^4} \int\limits_{\R} | \varphi'' (y)|^2 dy \to 0 ~ (n \to \infty),$$
				and finally $u_n \rightharpoonup 0$ in $L^2 (\R)$, which follows as in the proof of Lemma \ref{L:Weyl_sp_Tk}. Since \eqref{Plum39}, \eqref{Plum40} (and $u_{n,1} \equiv 0$) imply $u_n \in (\ker L_k (\omega))^{\bot}$ for $n$ sufficiently large, Lemma \ref{L:ran_not-closed} shows that range $L_k (\omega)$ is not closed in $\mathcal{R}_k$, whence $L_k (\omega)$ is not semi-Fredholm and thus $\omega \in \sigma_{e,1} (\cL_k) \cap \Omega_0$.
			\end{proof}

			%-------------------------------------------------------------------------------------------
			%-------------------------------------------------------------------------------------------

			\section{Two-Dimensional Reduction}\label{S:2D}

			In this section we study the problem defined  by \eqref{E:ansatz2D} and \eqref{E:NL-eval-2D}. The functional analytic setting is introduced in Sec.~\ref{S:main-res-2D}.

			\subsection{Explanation of the Functional Analytic Setting}
			%Note that for $\perm E_{\pm,1}, E_{\pm,2}, E_{\pm,3}$ and $(\nabla \times E_\pm)_{2,3}$ continuous on $\overline{\R_\pm^2}$, respectively, the interface conditions \eqref{E:IFC} are equivalent to the continuity of $\perm(\omega) E_1, E_2, E_3, \pa_{x_1}E_2- \pa_{x_2}E_1,$ and $\pa_{x_1}E_3$ across the interface.
			%
			%Once again, the choice of the interface conditions is dictated by the functional setting. More precisely, $u\in \cD_{\omega}$ is equivalent to $u, \nabla\times u, \nabla\times \nabla\times u \in L^2(\R^2,\C^3)$ and $\nabla\cdot (\perm u)=0$ (distributionally), see Appendix \ref{App:traces2D}.

			Similarly to the one-dimensional case in Section \ref{S:1D}, we can rewrite the domain of the operator $\cD_\omega$ in terms of conditions on half spaces and interface conditions. We first note that, due to \eqref{E:ansatz2D}, we have
			\begin{equation}\label{E:curl}
				\nabla\times E=\bspm \pa_{x_2}E_3 \\ -\pa_{x_1}E_3 \\ \pa_{x_1}E_2-\pa_{x_2}E_1\espm \quad \text{and}\quad
				\nabla\times \nabla\times E=\bspm \pa_{x_2}\left(\pa_{x_1}E_2-\pa_{x_2}E_1\right) \\ \pa_{x_1}\left(\pa_{x_2}E_1-\pa_{x_1}E_2\right)\\ -\pa_{x_1}^2E_{3}-\pa_{x_2}^2E_{3}\espm.
			\end{equation}
			Using \eqref{E:curl}, the domain can be equivalently written as
			\begin{align}\label{E:domega}
				\cD_\omega=\{&E\in L^2(\R^2,\C^3): \ \text{the } L^2\text{-conditions} \ \eqref{eq4}, \eqref{eq6},	\text{the divergence condition} \ \eqref{E:div-cond-2D},\\ \nonumber
				&\text{ and the interface conditions } \eqref{E:IFC} \text{ hold}\},
			\end{align}
			\begin{align}
				\pd{1}(E_\pm)_2-\pd{2}(E_\pm)_1, \;  \partial_{x_1}\left(\pd{2}(E_\pm)_1-\partial_{x_1}(E_\pm)_2\right),
				\; \pd{2}\left(\pd{1}(E_\pm)_2-\partial_{x_2}(E_\pm)_1\right) &\in  L^2(\R^2_\pm), \label{eq4}\\
				\pd{1}(E_\pm)_3, \; \pd{2}(E_\pm)_3, \; ( \partial_{x_1}^2 +\partial_{x_2}^2) (E_\pm)_3 &\in L^2(\R^2_\pm),  \label{eq6}
			\end{align}
			\be
			\perm_\pm(\omega) \nabla \cdot  E_\pm = 0\; \text{on} \;\R^2_\pm,
			\label{E:div-cond-2D}
			\ee
			\beq
			\begin{aligned}\label{E:IFC}
				&	\llbracket \perm(\omega) E_1\rrbracket := T^n_+(\perm(\omega) E)-T^n_-(\perm(\omega) E)=0, \ (\llbracket E_2\rrbracket, \llbracket E_3\rrbracket):= T^t_+(E)-T^t_-(E)=0, \\
				&(\llbracket -\pa_{x_1}E_3\rrbracket, \llbracket \pa_{x_1}E_2- \pa_{x_2}E_1\rrbracket)=T^t_+(\nabla\times E)-T^t_-(\nabla \times E)=0,
			\end{aligned}
			\eeq
			where $T^n_\pm, T^t_\pm$ is the normal, resp. tangential trace on $\{0\}\times \R$ taken from $\R^2_\pm$, defined in Appendix \ref{App:traces2D}. A more detailed proof of equality \eqref{E:domega} can also be found there (see Lemma \ref{L:Domega-equiv}).

			Like in the one-dimensional case, choosing $\cR$ (see \eqref{E:R}) rather than all of $L^2(\R^2,\C^3)$ as the range space is necessary in order to obtain a non-empty resolvent set as $\nabla \cdot (L(\omega;\lambda)E)=0$ for any $E\in \cD_\omega$. Again, since there exist $E\in \cD_\omega$ such that $\div{E}\neq 0$ unless $\perm_+=\perm_-$, we have
			$\cD_\omega  \not\subset \mathcal{R}$.

			We can equip $\cD_\omega$ with the inner product
			\be
			\langle E, \tilde E \rangle_{\cD_\omega}:=\langle E, \tilde E \rangle_{L^2}+ \langle \curl{  \curl {E}},\curl{  \curl {\tilde E }}\rangle_{L^2}  \label{eq10}
			\ee
			which makes $\cD_\omega$ a Hilbert space  as  we show in the next lemma.
			\blem\label{L:HS-2D}
			$(\cD_{\omega},\langle\cdot,\cdot\rangle_{\cD_\omega})$ is a Hilbert space for any $\omega \in D(\perm)$.
			\elem
			\bpf
			The proof is analogous to that of Lemma \ref{L:HS-1D}. We simply  replace $\R$ by $\R^2$, $\nabla_k$ by $\nabla$,  and $\cD_{k,\omega}$ by $\cD_\omega$. The equivalence of
			$$\|v\|:=\left(\|v\|_{L^2(\R^2)}^2 + \|\nabla \times v\|_{L^2(\R^2)}^2 + \|\nabla \times \nabla \times v\|_{L^2(\R^2)}^2 \right)^{1/2}$$
			and the norm generated by $\langle \cdot, \cdot\rangle_{\cD_{\omega}}$ uses the fact that
			$$\|\nabla \times v\|_{L^2(\R^2)}^2 = \langle v, \nabla \times \nabla \times v\rangle_{L^2} \leq \|v\|_{L^2(\R^2)}\|\nabla \times \nabla \times  v\|_{L^2(\R^2)}$$
			for each $v\in \cD_{\omega}$. To show the above equality, we use Lemma \ref{L:PI-2D} with $u:=\nabla\times v$. Due to $v\in \cD_\omega$ we have, indeed, $u, v, \nabla\times u, \nabla\times v \in L^2(\R^2,\C^3)$, i.e. the assumptions of Lemma \ref{L:PI-2D} are satisfied.
			\epf

			We shall need the Fourier transformation in what follows and it will be convenient to give a precise definition below.
			\begin{definition} \label{def:ft}
				Let $f$ lie in the Schwartz space of smooth, rapidly decreasing functions over $\R^n$. Then its Fourier transformation is given by
				$$
				(\Fc f)(k) = \hat f(k) =\dfrac1{ (2 \pi)^{n/2}} \int_{\R^n} f(x)e^{-\ri x\cdot k}\dd x\quad \hbox{ for } k \in \R^n.
				$$
			\end{definition}
			The Fourier transform is then extended to all functions $f\in L^2(\R^n)$ by the usual procedure.

			We denote $\hat E(x_1,k):=  [ \Fc(   E(x_1, \cdot ))](k),\hat r (x_1, k):= [\Fc (r(x_1,\cdot))](k).$ By Plancherel's theorem, the $L^2-$conditions for $E$ contained in $\cD_\omega$ transform into $L^2-$conditions for $\hat E$. That is, denoting $\hat E_\pm := \hat  E  |_{\R^2_\pm} = \widehat{E_\pm}$, we have
			$(\hat E_\pm)_j  \in L^2(\R^2_\pm)$, $j=1,2,3$, and conditions (\ref{eq4})-(\ref{E:IFC}) become
			\begin{align}
				\pd{1}( \hat E_\pm)_2  - \ri k (\hat E_\pm)_1, \;  \partial_{x_1}^2( \hat E_\pm)_2- \ri k \pd{1} ( \hat E_\pm)_1 , \;  \ri k\pd1(\hat E_\pm)_2+k^2 (\hat E_\pm )_1 &\in L^2(\R^2_\pm), \label{eq13} \\
				\pd1 (\hat E_\pm)_3, ~\ri k (\hat{E}_\pm)_3, \; \partial_{x_1}^2(\hat E_\pm)_3-k^2 (\hat E_\pm)_3 &\in L^2(\R_\pm^2), \label{eq14}
			\end{align}
			\be
			\perm_\pm(\omega) \left(\pd1(\hat E_\pm)_1+ \ri k (\hat E_\pm)_2\right)=0 \; \text{on} \;\R^2_\pm,
			\label{eq15}
			\ee
			\beq
			\begin{aligned}
				\llbracket \perm(\omega) \hat{E}_1\rrbracket = \llbracket \hat{E}_2\rrbracket =\llbracket \hat{E}_3\rrbracket &=0 \text{ for a.e. } k\in \R, \\
				\llbracket \pa_{x_1}\hat{E}_2- \ri k \hat{E}_1\rrbracket =\llbracket \pa_{x_1}\hat{E}_3\rrbracket &= 0  \text{ for a.e. } k\in \R.
			\end{aligned}
			\label{eq16}
			\eeq
			Note that the functions in \eqref{eq13}-\eqref{eq15} are to be understood as functions of $(x_1,k)$ and the traces in \eqref{eq16} depend on $k$.

			Finally, the conditions for $r\in \cR$ transform into
			\be
			\hat r \in L^2(\R^2,\C^3), ~ \pd{1} \hat r_1 + \ri k \hat r_2=0 \ \text{distributionally for a.e. } k\in \R. \label{eq19}
			\ee

			\brem\label{R:setting-mu-2d}
			Similarly to Remark \ref{R:setting-mu-1d} we note that if $\boldsymbol{\mu}=\boldsymbol{\mu}(x_1)$, then the functional analytic setting changes and the second order formulation on $\R^2$ changes for the same reason as in Remark \ref{R:setting-mu-1d}. Note that the first order formulation with $\boldsymbol{\mu}$ and $\perm$ constant on each $\R^2_+$ and $\R^2_-$ was considered in \cite{CHJ2017} in a specific self-adjoint case.
			\erem

			%-------------------------------------------

			\subsection{The Resolvent Set of $\cL$}

			After taking the Fourier transformation with respect to $x_2$ the inhomogeneous problem $\nabla\times \nabla \times E-\omega^2\perm(x_1,\omega) E =r$, with $E\in \cD_\omega$ and $r\in \cR$, is equivalent to
			\beq
			\left(\begin{array}{ccc}
				k^2  &\ri k  \pd{1}  & 0 \\
				\ri k \pd{1}  & -\partial_{x_1}^2 & 0 \\
				0 & 0& -\partial_{x_1}^2+k^2   \end{array}
			\right)\hat E-\omega^2\perm(x_1 , \omega) \hat E =\hat r, \quad (x_1,k)\in \R^2.
			\label{eq12}
			\eeq
			Since the Fourier transformation $\Fc : L^2(\R)\to L^2(\R)$ is isomorphic and isometric, we obtain:
			\begin{lemma} \label{lem1}
				$\omega \in D(\perm)$ is in the resolvent set $\rho(\cL)$ if and only if, for each $\hat r$ satisfying (\ref{eq19}), problem (\ref{eq12}) has a unique solution
				$\hat E \in L^2(\R^2,\C^3)$  satisfying (\ref{eq13})-(\ref{eq16}) and $\| \hat E \|_{L^2}\leq C \| \hat r \|_{L^2}$ with $C$ independent of $\hat r$.
			\end{lemma}

			We start again by studying first the non-singular set $\omega \in D(\perm)\setminus \Omega_0$. We define
			$$\sigma^{{\rm red}}(\cL):=\sigma(\cL)\setminus \Omega_0, \quad \rho^{{\rm red}}(\cL):=\rho(\cL) \setminus \Omega_0.$$

			\begin{proposition}\label{th2}
				$$\sigma^{{\rm red}}(\cL)\subset M_+\cup M_-\cup N,$$
				where $M_\pm$ and $N$ are defined in \eqref{E:Mpm} and \eqref{E:N}.
			\end{proposition}
			\brem\label{R:res-set-2D}
			In other words, we have $\rho^{{\rm red}}(\cL)\supset D(\perm)\setminus (\Omega_0 \cup  M_+\cup M_-\cup N)$.
			\erem
			\brem\label{R:Nequiv}
			Before starting the proof of Proposition \ref{th2} we note that we can write $N$ as
			$$
			\begin{aligned}
				N= \{&\omega \in D(\perm)\setminus \Omega_0: \omega^2\perm_+(\omega),\omega^2\perm_-(\omega)\notin [a,\infty),\\ &\perm_+(\omega)\sqrt{a-\omega^2\perm_-(\omega)}=-\perm_-(\omega)\sqrt{a-\omega^2\perm_+(\omega)} \ \text{for some} \ a\geq 0\}
			\end{aligned}
			$$
			by the same argument as in the 1D case in Remark \ref{R:Nk-equiv}, where we set $a=k^2$.
			\erem

			Large parts of the proof of Proposition \ref{th2} use calculations performed in the proof of Proposition \ref{T:resolv-1D} already. However, because the $L^2$-conditions are to be shown over $\R^2_\pm$, we need to control also the $k$-dependence of all constants. To that end we use the following Lemma. We define, analogously to the one-dimensional case,
			$$\mu_\pm=\mu_\pm(k):= \sqrt{k^2 -\omega^2\perm_\pm (\omega)}.$$
			\begin{lemma}\label{lem1.3}
				Let $\omega\in D(\perm)\setminus (\Omega_0 \cup  M_+\cup M_-\cup N)$. There exists some $\delta > 0$ such that, for all $k \in \R$,
				\be \label{eq20}
				\Re \mu_+ \ge \delta (|k| + 1),  \; \Re \mu_- \ge \delta (|k| + 1),
				\ee
				\be \label{eq20aPlum}
				\Big| \perm_+(\omega) \sqrt{k^2 -\omega^2\perm_-(\omega)}+  \perm_-(\omega) \sqrt{k^2 -\omega^2\perm_+(\omega)} \Big| \geq \frac{\delta}{|k|+1}  .
				\ee
			\end{lemma}
			\begin{proof}
				First we show (\ref{eq20}) and (\ref{eq20aPlum}) for $|k| \ge k_0$, with $k_0$ sufficiently large. This is obvious for (\ref{eq20}), and for (\ref{eq20aPlum}) we distinguish two cases:\\
				{\it Case 1:} $\perm_+(\omega) + \perm_-(\omega) \neq 0$. Then,
				\begin{eqnarray*}
					&&   \Big| \perm_+(\omega) \sqrt{k^2 -\omega^2\perm_-(\omega)}+  \perm_-(\omega) \sqrt{k^2 -\omega^2\perm_+(\omega)} \Big|\\
					&& \quad = \Big| [\perm_+(\omega)  +  \perm_-(\omega)] |k| - \omega^2 \perm_+(\omega)  \perm_-(\omega)
					\Big[ \dfrac1{\sqrt{k^2 -\omega^2\perm_-(\omega)} + |k|} + \dfrac1{\sqrt{k^2 -\omega^2\perm_+(\omega)} + |k|} \Big] \Big|\\
					&& \quad \geq   |  \perm_+(\omega) + \perm_-(\omega)| |k| - \dfrac{2 \omega^2| \perm_+(\omega)  \perm_-(\omega)|}{|k|} \geq 1 \geq \frac{1}{|k|+1}
				\end{eqnarray*}
				for $|k|$ sufficiently large.\\
				{\it Case 2:} $\perm_+(\omega) + \perm_-(\omega) = 0$. Then the left-hand side of (\ref{eq20aPlum}) equals
				\begin{eqnarray*}
					&& |\perm_+ (\omega) | \; | \sqrt{k^2 + \omega^2\perm_+(\omega)}- \sqrt{k^2 - \omega^2\perm_+(\omega)}|  \\
					&& \hspace*{2cm} =|\perm_+ (\omega) | \; \frac{2|  \omega^2\perm_+(\omega) | }{|\sqrt{k^2 + \omega^2\perm_+(\omega)}+ \sqrt{k^2 - \omega^2\perm_+(\omega)}|}   \\
					&& \hspace*{2cm} = \frac{2|\omega^2\perm^2_+ (\omega) |}{|k| + 1} \cdot \underbrace{\frac{|k| + 1}{| \sqrt{k^2 + \omega^2\perm_+(\omega)} + \sqrt{k^2 - \omega^2\perm_+(\omega)}|} }_{
						\longrightarrow \frac{1}{2} ~ {\rm as} ~ |k| \to \infty }\\
					&& \hspace*{2cm} \ge \frac{\frac{1}{2} |\omega^2\perm12_+ (\omega)|}{|k| + 1}
				\end{eqnarray*}
				for $|k| \ge k_0$, when $k_0$ is chosen sufficiently large. Since $\omega \notin \Omega_0$ and hence $\perm_+ (\omega) \neq 0$, we obtain (\ref{eq20aPlum}) for $|k| \ge k_0$ also in Case 2.

				Moreover, on $[0, k_0]$ all three quantities on the left-hand sides of (\ref{eq20}) and (\ref{eq20aPlum}) are bounded away from $0$, since they are nowhere $0$ (note that $\omega \notin M_+ \cup M_- \cup N$),
				and $k\mapsto \sqrt{k^2 -\omega^2\perm_\pm(\omega)}$ is continuous as $k^2 -\omega^2\perm_\pm(\omega) \not \in (-\infty, 0]$  for all $k\in\R$. This proves the lemma.
			\end{proof}
			We remark that the right-hand side of (\ref{eq20aPlum}) can be replaced by $\delta (|k| +1)$ if $\perm_+ (\omega) + \perm_- (\omega) \neq 0$, as the proof of Lemma \ref{lem1.3} shows. The weaker bound $\delta / (|k| + 1)$ in (\ref{eq20aPlum}) is needed to include also the case $\perm_+ (\omega) + \perm_- (\omega) = 0$ in our analysis.

			\bpf(of Proposition \ref{th2})\\
			Let $\omega\in D(\perm) \backslash  ( \Omega_0\cup M_+\cup M_-\cup N )$.  We have to show that $\omega\in \rho(\cL)$.
			Next, we use Lemma \ref{lem1}. Thus, let $\hat{r}$ satisfying (\ref{eq19}) be given. As mentioned above, we use the calculations in the proof of Proposition \ref{T:resolv-1D}, which are to be understood for almost every $k\in\R$.

			Equation \eqref{eq12} is identical to \eqref{Neq22} after replacing $u$ by $\hat{E}(\cdot,k)$ and $\rho$ by $\hat r (\cdot,k)$. Hence, we use all the calculations performed on \eqref{Neq22}. As a result $\hat{E}_{2,3}(\cdot,k)$ are given by \eqref{E:u2plus} and \eqref{E:u2minus} and $\hat{E}_1(\cdot,k)$ is given by \eqref{E:u1plus} and \eqref{E:u1minus}, where the constants $C_2(k)$ and $C_3(k)$ are given in \eqref{E:C2} and \eqref{Neq31}. This solution satisfies $\hat E(\cdot, k)\in L^2(\R,\C^3)$ (for almost all $k\in \R)$, the divergence condition \eqref{eq15}, and all the interface conditions \eqref{eq16} (for almost all $k\in \R$) as shown in the proof of Proposition \ref{T:resolv-1D}. It remains to show $\hat E \in L^2(\R^2,\C^3)$, conditions \eqref{eq13}, \eqref{eq14}, and
			\beq
			\| \hat E \|_{L^2} \leq C \norm{\hat r }_{L^2}  \mbox{ with } C  \mbox{ independent of }\hat r. \label{eq32}
			\eeq
			These will be proved using Lemma \ref{lem1.3}. Together with \eqref{E:u23pl-est}, \eqref{E:u23min-est}, \eqref{E:u1pl-est}, and \eqref{E:u1min-est} we obtain for almost all $k\in \R$
			\beq\label{E:E23-est-2D}
			\| \hat E_{2,3}(\cdot,k) \|_{L^2(\R_\pm)}  \leq \dfrac1{\sqrt{2\delta (|k| + 1)}} | C_{2,3}(k)| + \dfrac5{4\delta^2}\norm{\hat r_{2,3}(\cdot,k)}_{L^2(\R_\pm)}
			\eeq
			and
			\begin{align}
				\| \hat E_1(\cdot,k) \|_{L^2(\R_\pm)}  &\leq \dfrac1{\delta^2} \norm{  \hat r_1(\cdot,k) }_{L^2(\R_\pm}  + \frac{|k|}{\sqrt{2 \delta^3 (|k| + 1)^3}}  |C_2 (k)| +
				\frac{5|k|}{4 \delta^3 (|k| + 1)^3}  \norm{\hat r_2 (\cdot,k)}_{L^2(\R_\pm)} \nonumber \\
				& \le M \Big[ \frac{|C_2(k) |}{\sqrt{|k| +1}} +   \|  \hat{r}_1 (\cdot, k)\|_{L^2 (\R_\pm)} + \|  \hat{r}_2 (\cdot, k)\|_{L^2 (\R_\pm)}   \Big], \label{E:E1-est-2D}
			\end{align}
			where, here and in the following, $M$ is a (generic) constant independent of $k$.

			Next, we control the constants $C_{2,3}(k)$ in \eqref{E:C2-est} and \eqref{E:C3-est}. Using Lemma \ref{lem1.3} and \eqref{E:rho1_0_est}, we can continue the estimate in  \eqref{E:C2-est1} by
			\begin{align}
				|C_2 (k)| &  \le \frac{1}{\delta^3 (|k| + 1)} \Big[ |k| \;  |\perm_+ - \perm_- | \sqrt{|k| + 1} \Big( \| \hat{r}_1 (\cdot, k) \|_{L^2 (0,\infty)} + \| \hat{r}_2 (\cdot, k) \|_{L^2 (0,\infty)} \Big) \nonumber  \\
				& \hspace*{1cm} + | \perm_+ | \frac{k^2 + |\omega^2\perm_-|}{\sqrt{ 2 \delta (|k| + 1)}} \| \hat{r}_2 (\cdot, k) \|_{L^2 (0,\infty)} + |\perm_-| \frac{k^2 + |\omega^2\perm_+|}{\sqrt{ 2 \delta (|k| + 1)}} \| \hat{r}_2 (\cdot, k) \|_{L^2 (-\infty,0)} \Big] \nonumber \\
				& \le M \sqrt{|k| + 1}(\| \hat{r}_1 (\cdot, k) \|_{L^2 (\R)}+\| \hat{r}_2 (\cdot, k) \|_{L^2 (\R)}). \label{E:C2-est-2D}
			\end{align}
			Similarly, from \eqref{E:C3-est1} we have
			\beq\label{E:C3-est-2D}
			|C_3(k)|\leq\left ( \dfrac1{2\delta}\right )^{3/2}  \norm{ \hat r_3 (\cdot,k)}_{L^2(\R)}.
			\eeq
			Now \eqref{E:E23-est-2D} - \eqref{E:C3-est-2D} and \eqref{eq19} show that $\hat E \in L^2(\R^2,\C^3)$ and \eqref{eq32} holds.

			Next, we show  \eqref{eq13}-\eqref{eq14}. Starting with \eqref{E:pa1_u23-est} and using Lemma \ref{Nl2lem} and (\ref{eq20}), we have
			\be
			\norm{  \pd{1}(\hat E_+)_{2,3}(\cdot,k)}_{L^2(0,\infty)}
			\leq \dfrac{|\mu_+|}{\sqrt{2 \delta(|k| + 1)} }|C_{2,3}(k)|+\dfrac5{4\delta}
			\norm{ \hat r_{2,3}(\cdot,k)}_{L^2(0,\infty)}.\label{eq42}
			\ee
			Since $|\mu_+|=|\sqrt{ k^2-\omega^2\perm_+}|\leq M (|k|+1)$, we obtain from (\ref{eq42}),(\ref{E:C2-est-2D}) that
			$$\dfrac{\pd{1}(\hat E_+)_2}{|k|+1}\in L^2(\R^2_+),$$
			where the left hand side is again understood as a function of $x_1$ and $k$.

			Analogously,  $\dfrac{ \pd{1}(\hat E_-)_2}{|k|+1} \in L^2(\R^2_-)$.  Thus, using (\ref{Neq23}),
			\begin{eqnarray*}
				\pd{1}(\hat E_\pm)_2-\ri k (\hat E_\pm)_1&=& \dfrac1{ k^2 -\omega^2\perm_\pm}  \left[  (k^2-\omega^2\perm_\pm) \pd{1} (\hat E_\pm)_2 -\ri k (\hat r_1 -\ri k\pd{1}(\hat E_\pm)_2) \right]
				\\
				&=&-\dfrac1{k^2 -\omega^2\perm_\pm}[\omega^2\perm\pm \pd{1}(\hat E_\pm)_2+ \ri k\hat r_1],
			\end{eqnarray*}
			implying
			$$|\pd{1}(\hat E_\pm)_2 - \ri k(\hat E_\pm)_1|\leq   \underbrace{
				\dfrac{ |k|+1|}{|k^2 -\omega^2\perm_\pm |}}_{\leq M}
			\left[ |\omega^2\perm_\pm | \dfrac{   \pd{1}(\hat E_\pm)_2 }{ |k|+1} +|\hat r _1| \right],$$
			whence
			$\pd{1}(\hat E_\pm)_2 - \ri k (\hat E_\pm )_1\in L^2 (\R^2_\pm)$.

			Moreover, by (\ref{Neq22}),
			$$\partial_{x_1}^2 (\hat E_\pm)_2- \ri k \pd{1}( \hat E_\pm)_1 =- \omega^2\perm_\pm (\hat E_\pm)_2 -\hat r_2 \in L^2 (\R^2_\pm) ,$$
			$$ \ri k \pd{1}(\hat E_\pm)_2 + k^2 (\hat E_\pm)_1 = \omega^2\perm_\pm(\hat E_\pm)_1 +\hat r _1 \in L^2 (\R^2_\pm),$$
			whence (\ref{eq13}) is proved.

			To show (\ref{eq14}), we make use of (\ref{eq42}) again. Using the fact that
			$\Re (\mu_++\mu_-)\geq 2 \delta (|k| +1)$ by (\ref{eq20}),
			we obtain
			$$ \dfrac{ |\mu_+|}{\Re(\mu_++\mu_-)}\leq M, \;\dfrac{ |k|}{\Re(\mu_++\mu_-)}\leq M,$$
			which by (\ref{E:C3-est1}) implies
			\be
			|\mu_+||C_3(k)|\leq \dfrac{M}{\sqrt { 2 \delta}}   \norm{\hat r_3(\cdot, k)}_{L^2 (\R)}, \label{eq43}
			\ee
			\be
			|k||C_3(k)|\leq \dfrac{M}{\sqrt { 2 \delta}}  \norm{\hat r_3(\cdot, k)}_{L^2 (\R)}. \label{eq44}
			\ee
			Inequality (\ref{eq42}) (and its analogue for $(\hat E_-)_3)$ and (\ref{eq43}) imply
			$$\pd{1}(\hat E_\pm )_3\in L^2 (\R^2_\pm).$$

			Moreover, (\ref{E:u23pl-est}), (\ref{E:u23min-est}) (for $\hat E_3(\cdot,k)$ instead of $u_3$), and (\ref{eq44}), together with $|k|/\Re(\mu_\pm)\leq M$ (which follows from \eqref{eq20}), show that
			$$\ri k (\hat E_\pm)_3 \in L^2 (\R^2_\pm).$$

			Also, by  (\ref{Neq22})
			$$\partial_{x_1}^2 (\hat E_\pm)_3 - k^2 (\hat E_\pm)_3 = - \omega^2\perm_\pm(\hat E_\pm)_3 -\hat r_3 \in L^2 (\R^2_\pm).$$
			This shows (\ref{eq14}) and the proposition is now proved.
			\epf

			%-------------------------------------------
			\subsection{Eigenvalues of $\cL$}\label{S:evals-2D}

			The point spectrum consists only of a subset of the singular set $\Omega_0$, as the following lemmas show.
			\blem\label{L:pt-sp-2D-red}
			\beq\label{E:disc-sp-2D}
			\sigma^{{\rm red}}_p(\cL)=\emptyset.
			\eeq
			\elem

			\bpf
			Let $\omega \in \sigma_p(\cL)\setminus \Omega_0$. Once again, we apply the Fourier transform - this time to \eqref{E:NL-eval-2D} with $E$ denoting the associated eigenfunction. We obtain \eqref{E:NL-eval} for $\hat{E}(\cdot,k)$ instead of $u$. As this is an ODE for each $k$, we get, as in \eqref{E:psi12_form}, that $\hat{E}_3=0$ and
			\beq\label{E:psi-efn-2D}
			\begin{pmatrix}\hat{E}_1\\\hat{E}_2\end{pmatrix}(x_1,k)=\begin{cases}v_+(k)e^{-\mu_+(k) x_1}, & x_1>0, \\ v_-(k)e^{\mu_-(k) x_1}, & x_1<0
			\end{cases}
			\eeq
			with $\Re(\mu_\pm)>0$.

			The calculations of the case $\omega\in D(\perm)\setminus \Omega_0$ in Lemma \ref{L:efn-1D} apply and produce $\mu_\pm(k)=\sqrt{k^2-\omega^2\perm_\pm(\omega)}$. Necessary conditions for $\hat{E}\in L^2(\R^2,\C^3)$ are
			$$k^2-\omega^2\perm_+(\omega)\notin (-\infty,0], \ k^2-\omega^2\perm_-(\omega)\notin (-\infty,0]$$
			for a.e. $k\in \text{supp}(v_+)\cup \text{supp}(v_-)$. From Proposition \ref{T:pt-spec-1D} we get
			$$\perm_+(\omega)\sqrt{k^2-\omega^2\perm_-(\omega)}=-\perm_-(\omega)\sqrt{k^2-\omega^2\perm_+(\omega)}$$
			for a.e. $k\in \text{supp}(v_+)\cup \text{supp}(v_-)$.

			As the latter condition can be satisfied by at most two values of $k\in \R$ (see \eqref{E:om-ev-cond-neces}) and since $E\neq 0$ implies that $\mbox{meas}(\mbox{supp}(v_+) \cup \mbox{supp}(v_-))>0$, we arrive at a contradiction. Hence, there is no eigenvalue and $\sigma_p^{{\rm red}}(\cL)=\emptyset$ is proved.
			\epf

			\blem\label{L:pt-sp-2D-sing}
			$$
			\begin{aligned}
				\sigma_p^{(\infty)}(\cL)\cap \Omega_0  &=  ~\{\omega \in \Omega_0: \perm_+(\omega)=0 \text{ or } \perm_-(\omega)=0 \text{ or } \perm_+(\omega)+\perm_-(\omega)=0 \},\\
				\sigma_p^{(<\infty)}(\cL)\cap \Omega_0  &=  ~\emptyset.
			\end{aligned}
			$$
			\elem
			\bpf
			We consider first the case $\omega \in D(\perm)$ such that
			$$\perm_+(\omega)=0 \ \text{or} \ \perm_-(\omega)=0,$$
			which clearly implies $\omega\in \Omega_0$. Without any loss of generality let $\perm_+(\omega)=0$. Then the divergence condition $\nabla \cdot (\perm E)=0$ makes no statement about $E$ on $\R^2_+$ and we have
			$$\cD_\omega=\{E\in L^2(\R^2,\C^3): \nabla\times E,\nabla\times \nabla\times E \in L^2(\R^2,\C^3), \nabla \cdot (\perm_-(\omega)E) =0 \text{ distrib. in } \R^2_-, \perm_-(\omega)E_1(0-,\cdot)=0\}.$$
			Then an infinite dimensional eigenspace exists again due to the fact that gradient fields lie in the kernel of the curl operator. In detail,
			$$E= \bspm \pa_{x_1}f\\ \pa_{x_2}f\\ 0\espm \quad \text{with} \  f\in C^2_c(\R^2_+) \ \text{arbitrary},$$
			is an eigenfunction. Hence $\omega$ is an eigenvalue of infinite multiplicity.

			Next, we study the case
			$$\omega \in \Omega_0, \ \perm_+(\omega)\neq 0, \perm_-(\omega)\neq 0.$$
			Clearly, $\omega =0$ and hence $\omega^2\perm_+(\omega)=\omega^2\perm_-(\omega)=0$. The third equation in \eqref{E:NL-eval-2D} reads
			$$\Delta E_3=0,$$
			which together with $E_3\in L^2(\R^2)$ implies $E_3=0.$ The Fourier transform (in $x_2$) of the first two equations is
			\beq\label{E:E12-eq-red}
			\begin{aligned}
				k^2\hat{E}_1 + \ri k \pa_{x_1}\hat{E}_2& = 0,\\
				\ri k \pa_{x_1} \hat{E}_1-\pa_{x_1}^2\hat{E}_2&=0
			\end{aligned}
			\eeq
			for $(x,k)\in \R_\pm\times \R$. In addition, the divergence condition $\pa_{x_1} \hat{E}_1 + \ri k \hat{E}_2=0$ on $\R^2_\pm$ has to be satisfied in order for $E\in \cD_\omega$. For $k\neq 0$  we get
			$$\pa_{x_1} \bspm\hat{E}_1\\ \hat{E}_2\espm =\bspm 0 & -\ri k \\ \ri k & 0\espm \bspm\hat{E}_1\\ \hat{E}_2\espm$$
			from the divergence condition and the first equation in \eqref{E:E12-eq-red}. The second equation in \eqref{E:E12-eq-red} then follows automatically.
			Solutions with $\hat{E}(\cdot,k)\in L^2(\R)$ have the form
			$$ \bspm\hat{E}_1\\ \hat{E}_2\espm(x_1,k)=v_\pm(k)e^{\mp |k|x_1}, \ x_1\in \R_\pm,$$
			where $v_+=A(\ri k, |k|)^T$ and $v_-=B(-\ri k, |k|)^T$ with $A=A(k)\in \C, B=B(k)\in \C$ arbitrary. For the rest of the conditions in $E\in \cD_\omega$ we use the formulation \eqref{eq13}-\eqref{eq16}. Equation \eqref{eq15} is satisfied by construction. The conditions $\llbracket E_2\rrbracket =0$ and $\llbracket \perm E_1\rrbracket =0$ for a.e. $k\in \R$ are satisfied with nontrivial $(A,B)$ if and only if
			\beq\label{E:IFC-AB}
			A(k)=B(k)  \text{  for almost every $k\in \R$ and  } \perm_+(\omega)=-\perm_-(\omega)
			\eeq
			as one easily checks (similarly to the proof of Lemma \ref{L:efn-1D}).

			Note that this also shows that $\{\omega\in \Omega_0: \perm_\pm(\omega)\neq 0, \perm_+(\omega)\neq -\perm_-(\omega)\}\subset \Omega_0 \setminus \sigma_p(\cL)$. The condition $\llbracket \pa_{x_1} \hat{E}_3\rrbracket =0$ holds trivially and $\llbracket \pa_{x_1} \hat{E}_2(\cdot,k)-\ri k\hat{E}_1(\cdot,k)\rrbracket =0$ holds also trivially because $\pa_{x_1}\hat{E}_2(\cdot,k)
			-\ri k \hat{E}_1(\cdot,k)=0$ if $k\neq 0$. This covers the conditions in \eqref{eq16}.

			Because the coefficient $A(k)=B(k)$ can be chosen so that $\hat{E}_1, \hat{E}_2\in L^2(\R^2)$ and \eqref{eq13}-\eqref{eq14} hold, we get $\omega\in \sigma_p(\cL)$ if $\perm_+(\omega)=-\perm_-(\omega)$.  Indeed, for instance for $\hat{E}_1$ we get
			$$\|\hat{E}_1\|^2_{L^2(\R_+\times \R)}=\int_\R k^2|A(k)|^2\int_0^\infty e^{-2|k|x_1}\dd x_1 \dd k = \int_\R \frac{|k||A(k)|^2}{2}\dd k$$
			and choosing, e.g. $A(k):=e^{-\alpha |k|}$ with any $\alpha>0$ ensures that $\hat{E}_1$ as well as all $x_1$-derivatives are in $L^2(\R_+\times \R)$. Setting $B=A$, yields also $\hat{E}_1 \in L^2(\R_-\times \R)$.

			Clearly, due to the freedom in the choice of $\alpha$, the eigenvalue $\omega=0$ has infinite multiplicity, too, i.e. it is an element of $\sigma_p^{(\infty)}(\cL)$.

			To summarize, note that $\Omega_0=A~\dot{\cup}~B~\dot{\cup}~C$, where $A:=\{\omega \in D(\perm): \perm_+(\omega)=0 \ \text{or} \ \perm_-(\omega)=0\}$, $B:=\{\omega\in \Omega_0: \perm_\pm(\omega)\neq 0, \perm_+(\omega)+\perm_-(\omega)=0\}$, and
			$C:=\{\omega\in \Omega_0: \perm_\pm(\omega)\neq 0, \perm_+(\omega)+\perm_-(\omega)\neq 0\}$. We have shown that $A\cup B\subset \sigma_p^{(\infty)}(\cL)$ and $C\subset \Omega_0\setminus \sigma_p(\cL)$. This yields all statements of the lemma.
			\epf

			\brem \label{R:evals-mu-2d}
			In the more general setting with $\boldsymbol{\mu}=\boldsymbol{\mu}(x_1,\omega)$ additional eigenvalues of infinite multiplicity are expected to arise. These are zeroes of $\boldsymbol{\mu}(x_1,\cdot)$. These eigenvalues were proved in the first order formulation of a specific self-adjoint case in \cite{CHJ2017}  (Proposition 8) with $\boldsymbol{\mu}$ and $\perm$ constant on each $\R^2_\pm$.
			\erem
			%_------------------------------------------------------
			\subsection{The Weyl Spectrum of $\cL$}

			We proceed by locating the Weyl spectrum.

			\blem\label{L:ess-sp2D-x1}
			$$\sweyl(\cL)\supset M_+\cup M_-\cup\Omega_0 = \{\omega \in D(\perm): \omega^2\perm_+(\omega)\in [0,\infty) \text{  or  } \omega^2\perm_-(\omega)\in [0,\infty)\}.$$
			\elem
			\bpf
			Without loss of generality, we assume that $\omega^2\perm_+(\omega)\in [0,\infty)$.

			Again, similar to the proof of Prop. \ref{L:Weyl_sp_Tk} at least for $\omega^2\perm_+(\omega)>0$ a Weyl sequence with the first two components being non-zero exists. However, we opt for a much simpler sequence with only the third component being non-zero. As the action of the operator $\cL$ acting on a function of the form $(0,0,f)^T$ reduces to the action of $-\pa_{x_1}^2-\pa_{x_2}^2-\omega^2\perm(x_1,\omega)$ on $f$, we are left with studying only this Laplace-type operator.

			We choose a plane-wave solution on $x_1>0$ truncated smoothly to have a compact support which grows and moves to $x_1\to \infty$ as the sequence index grows to infinity. In detail
			$$u^{(n)}(x_1,x_2):=\frac{e^{\ri (\beta_1 x_1 +\beta_2 x_2)}}{n} \varphi\left(\frac{x_1-n^2}{n},\frac{x_2}{n}\right)\bspm 0 \\ 0 \\1\espm,$$
			where $\beta_1^2+\beta_2^2 = \omega^2\perm_+(\omega)$, $\varphi \in C^\infty_c(\R^2)$, and $\|\varphi\|_{L^2(\R^2)}=1$.

			To see that $u^{(n)} \in \cD_{\omega}$, note that the divergence condition as well as the regularity  conditions hold trivially and the interface can be ignored for $n$ large enough as the support of  $u^{(n)}$ lies in $\R^2_+$ for $n$ large enough. The normalization $\|u^{(n)}\|_{L^2(\R^2)}=1$ holds by the choice of $\varphi$.

			Next, we show that $\|L(\omega)u^{(n)}\|_{L^2(\R^2)}\to 0$. As for $n$ large enough
			$$
			\begin{aligned}
				L(\omega)u^{(n)}&= -\Delta u^{(n)}_3- \omega^2\perm_+(\omega) u^{(n)}_3\\
				&=e^{\ri (\beta_1 x_1 +\beta_2 x_2)}\left(n^{-3}\Delta \varphi\left(\frac{x_1-n^2}{n},\frac{x_2}{n}\right) -2\ri n^{-2}(\beta_1\pa_{x_1}\varphi+\beta_2\pa_{x_2}\varphi)\left(\frac{x_1-n^2}{n},\frac{x_2}{n}\right)\right),
			\end{aligned}
			$$
			we get
			$$\|L(\omega)u^{(n)}\|_{L^2(\R^2)} \leq n^{-2}\|\Delta \varphi\|_{L^2(\R^2)}+cn^{-1}(\|\pa_{x_1}\varphi\|_{L^2(\R^2)}+\|\pa_{x_2}\varphi\|_{L^2(\R^2)})\to 0.$$
			Finally, to show $u^{(n)}\rightharpoonup 0$ we only need to observe that for $\eta \in L^2(\R^2)$
			$$
			\begin{aligned}
				n^{-1}&\left|\int_{\R^2} e^{\ri(\beta_1 x_1+\beta_2 x_2)} \varphi\left(\frac{x_1-n^2}{n},\frac{x_2}{n}\right) \overline{\eta}(x_1,x_2)\dd x\right|\\
				&\leq n^{-1}\int_{[n,\infty)\times \R} \left|\varphi\left(\frac{x_1-n^2}{n},\frac{x_2}{n}\right)\right||\eta(x_1,x_2)|\dd x \qquad \text{for $n$ large enough}\\
				&\leq \|\varphi\|_{L^2(\R^2)} \|\eta\|_{L^2([n,\infty)\times \R))} \to 0.
			\end{aligned}
			$$

			Note that in the case $\perm_+(\omega)=0$ we can easily construct another Weyl sequence. Namely, in Lemma \ref{L:pt-sp-2D-sing} we have shown that $\omega$ is an eigenvalue of infinite multiplicity,  and hence an $L^2(\R^2)$-orthonormal sequence $u^{(n)}$ in the eigenspace of $\omega$ is clearly a Weyl sequence as orthonormal sequences converge weakly to zero.

			The case $\omega^2\perm_-(\omega)\in [0,\infty)$ is analogous and one uses a Weyl sequence moving to $x_1\to-\infty$.
			\epf

			\blem\label{L:ess-sp2D-x2}
			\beq\label{E:NsubWeyl}
			\begin{aligned}
				\sweyl(\cL)\supset N= \{&
				\omega \in D(\perm)\setminus \Omega_0: \omega^2\perm_+(\omega),\omega^2\perm_-(\omega)\notin [a,\infty),\\ &\perm_+(\omega)\sqrt{a-\omega^2\perm_-(\omega)}=-\perm_-(\omega)\sqrt{a-\omega^2\perm_+(\omega)} \ \text{for some} \ a\geq 0\}.
			\end{aligned}
			\eeq
			\elem
			\bpf
			Remark \ref{R:Nequiv} explains why the set equality in \eqref{E:NsubWeyl} holds.
			Let $\omega \in D(\perm)\setminus \Omega_0$ and $a\geq 0$ be such that
			\beq\label{E:om-eq-eval}
			\perm_+(\omega)\sqrt{a-\omega^2\perm_-(\omega)}=-\perm_-(\omega)\sqrt{a-\omega^2\perm_+(\omega)},
			\eeq
			and $a-\omega^2\perm_+(\omega)\notin (-\infty,0], a-\omega^2\perm_-(\omega)\notin (-\infty,0]$. Set $k_0:=\sqrt{a}$. Let us first explain that $a=0$ is not possible (allows no solutions of \eqref{E:om-eq-eval}).
			Squaring equation \eqref{E:om-eq-eval} for $a=0$ leads to $\perm_+(\omega)=\perm_-(\omega)$. Inserting this back into \eqref{E:om-eq-eval}, we obtain  $\sqrt{-\omega^2\perm_+(\omega)}=-\sqrt{-\omega^2\perm_+(\omega)}$, which is a contradiction to  $\omega \in \C\setminus \Omega_0$. Hence $k_0=\sqrt{a}> 0$.

			Our choice of a Weyl sequence (for a fixed $\omega$) is given by the plane-wave $e^{\ri k_0 x_2}$ times the $x_1-$dependent (exponentially decaying) eigenfunction $\psi$ of $L_{k_0}(\omega)$. The plane wave is smoothly cut off to have compact support with the support moving to $x_2\to \pm \infty$. First we construct a sequence $w^{(n)}$ which has all the Weyl sequence properties apart from the interface condition for the third component of its curl:
			\beq\label{E:Weylseq-2D-x2}
			w^{(n)}(x):=c_n(v_n(x)+r_n(x)):=\left(n^{-3/2}\varphi_n\left(\frac{x_2-n^2}{n}\right)\psi(x_1)e^{\ri k_0 x_2}+r_n(x)\right),
			\eeq
			where
			$$\psi(x_1):=\begin{cases}v_+e^{-\mu_+ x_1}, & x_1>0,\\
				v_-e^{\mu_- x_1}, & x_1<0,
			\end{cases}\quad \mu_\pm :=\sqrt{k_0^2-\omega^2\perm_\pm(\omega)}, v_+ =\bspm \ri k_0\\
			\mu_+\\ 0\espm, v_- =\frac{\mu_+}{\mu_-}\bspm -\ri k_0\\ \mu_-\\ 0\espm,
			$$
			and where
			$$\hat{\varphi}_n(k):=(k+nk_0)\hat{\varphi}(k), \ k \in \R \quad \text{ with } \quad \varphi \in C^\infty_c(\R), \|\varphi\|_{L^2(\R)}=1,$$
			and

			$$
			\hat{r}_n(x_1,k):= - n^{-1/2} \frac{k-k_0}{k}\hat{\varphi}_n(n(k-k_0))\bspm 0\\ \psi_2(x_1) \\ 0\espm e^{-\ri (k-k_0)n^2}.
			$$
			The constants $c_n>0$ are chosen such that $\|w^{(n)}\|_{L^2(\R^2)}=1$ for all $n\in \N$. Like in the proof of Lemma \ref{L:ess-sp2D-x1}, for a function $f:\R^2\to \C^3$ we use $\hat{f}$ to denote the Fourier transform of $f$ with respect to the second variable.

			Note that unlike at many other instances, here the somewhat complicated Weyl sequence with non-zero first two components has to be chosen since we need the localized modes (eigenfunctions) from the one-dimensional case as building blocks for these Weyl sequences. The eigenfunctions have a vanishing third component.

			The correction term $\hat{r}_n$ ensures that $\nabla \cdot w^{(n)}=0$ or equivalently $\pa_{x_1} \hat{w}_1^{(n)}(x_1,k)+\ri k \hat{w}_2^{(n)}(x_1,k)=0$ for a.e. $x_1,k\in \R$. Indeed, the Fourier transform of $v^{(n)}$ is
			$$\hat{v}^{(n)}(x_1,k)=n^{-1/2}\psi(x_1)e^{-\ri (k-k_0)n^2}\hat{\varphi}_n(n(k-k_0))$$
			and
			$$\pa_{x_1}\hat{v}_1^{(n)} + \ri k\hat{v}_2^{(n)}=n^{-1/2}(\psi_1'(x_1)+\ri k\psi_2(x_1))e^{-\ri (k-k_0)n^2}\hat{\varphi}_n(n(k-k_0)).$$
			Because $\psi_1'+\ri k_0\psi_2=0$ on $\R_+\cup \R_-$ (which follows from the curl-curl structure of the eigenvalue equation   for $\psi$ and from the fact that $\perm$ is constant on $\R_+$ and on $\R_-$), we get
			$$\pa_{x_1}\hat{v}_1^{(n)} + \ri k\hat{v}_2^{(n)}=n^{-1/2}\ri (k-k_0)\psi_2(x_1)e^{-\ri (k-k_0)n^2}\hat{\varphi}_n(n(k-k_0))$$
			on $\R^2_+\cup \R^2_-.$
			Then it follows  that
			$$\pa_{x_1}(\hat{v}_1^{(n)}+ \hat{r}_{n,1}) + \ri k(\hat{v}_2^{(n)}+\hat{r}_{n,2})=0.$$

			We have $\|r_n\|_{L^2(\R^2)}\to 0$ as $n\to \infty$ because
			$$\|\hat{r}_n\|_{L^2(\R^2)}^2= n^{-1}\|\psi_2\|_{L^2(\R)}^2\int_\R n^2(k-k_0)^2|\hat{\varphi}(n(k-k_0))|^2\dd k = n^{-2}\|\psi_2\|_{L^2(\R)}^2\|\widehat{\varphi'}\|_{L^2(\R)}^2\leq cn^{-2}.$$
			For the leading order part $v^{(n)}$ we have
			$$\|\hat{v}^{(n)}\|_{L^2(\R^2)}^2=n^{-1}\|\psi\|_{L^2(\R)}^2\int_{\R}|\hat{\varphi}_n(n(k-k_0))|^2\dd k =  n^{-2}\|\psi\|_{L^2(\R)}^2\int_\R (\kappa+nk_0)^2|\hat{\varphi}(\kappa)|^2\dd \kappa$$
			and hence, because $k_0\neq 0$, there are $\alpha,\beta>0$ such that $\alpha \leq \|v^{(n)}\|_{L^2(\R^2)} \leq \beta$ for all $n$. Together with $\|r_n\|_{L^2(\R^2)}\to 0$ this means that the condition $\|w^{(n)}\|_{L^2(\R^2)}=1$ implies $c_n=O(1)$ and $c_n\nrightarrow 0$ as $n\to \infty$.

			Next we show that $L(\omega)w^{(n)} \to 0 $ in $L^2(\R^2)$. In the following we denote by $\psi'$ and $\psi''$ the derivatives of $\psi$ defined piece-wise on $\R_-$ and $\R_+$. Accordingly, $\|\psi\|_{H^1(\R)}^2:=\|\psi\|_{H^1(\R_-)}^2+\|\psi\|_{H^1(\R_-)}^2$.

			In the Fourier variables we have
			\begin{align}
				&L_k(\omega)\hat{w}^{(n)}(x_1,k)\notag\\
				&=c_n \left[ \frac{1}{\sqrt{n}}e^{-\ri (k-k_0)n^2}\hat{\varphi}_n(n(k-k_0)) \bspm k^2 -\omega^2\perm & \ri k \pa_{x_1} & 0\\ \ri k \pa_{x_1} & - \pa_{x_1}^2-\omega^2\perm(\omega) & 0 \\
				0 & 0 & k^2-\pa_{x_1}^2-\omega^2\perm
				\espm\psi(x_1) + \bspm \ri k \pa_{x_1} \hat{r}_{n,2}(x_1,k)\\
				(-\pa_{x_1}^2-\omega^2\perm(\omega))\hat{r}_{n,2}(x_1,k) \\ 0 \espm\right]\notag\\
				&=\frac{c_n}{\sqrt{n}}e^{-\ri (k-k_0)n^2}\hat{\varphi}_n(n(k-k_0)) \left[\bspm k^2-k_0^2 & \ri (k-k_0) \pa_{x_1} & 0\\ \ri (k-k_0) \pa_{x_1} & 0 & 0 \\
				0 & 0 & 0\espm\psi(x_1) +
				\bspm -\ri(k-k_0)\psi_2'(x_1)\\
				\frac{k-k_0}{k}(\psi_2''+\omega^2\perm(\omega)\psi_2)(x_1)\\ 0 \espm \right],\label{E:Tun-corr2D}
			\end{align}
			where we have used $T_{k_0}(\pa_{x_1})\psi-\omega^2\perm(\omega)\psi=0$ and $\psi_3=0$.

			We rewrite the first term in \eqref{E:Tun-corr2D} as
			$$F(x_1,k):=\frac{c_n}{\sqrt{n}}e^{-\ri (k-k_0)n^2}k\hat{\varphi}(n(k-k_0)) \bspm (k+k_0)n(k-k_0)\psi_1(x_1)+\ri n(k-k_0)\psi_2'(x_1)\\ \ri n(k-k_0)\psi_1'(x_1)\\
			0\espm$$
			and get
			$$\|F\|_{L^2(\R^2)}^2\leq cn^{-2} \int_\R\kappa^2\left(k_0+\frac{\kappa}{n}\right)^2|\hat{\varphi}(\kappa)|^2 \left[\|\psi_2'\|_{L^2(\R)}^2 + \left(2k_0+\frac{\kappa}{n}\right)^2 \|\psi_1\|_{L^2(\R)}^2 +\|\psi_1'\|_{L^2(\R)}^2\right]\dd \kappa\leq cn^{-2}.$$
			The second term in \eqref{E:Tun-corr2D} is
			$$G(x_1,k):=\frac{c_n}{\sqrt{n}}e^{-\ri (k-k_0)n^2}n(k-k_0)\hat{\varphi}(n(k-k_0)) \bspm -\ri k\psi_2'(x_1)\\ (\psi_2''+\omega^2\perm(\omega)\psi_2)(x_1)\\
			0\espm.$$
			We estimate analogously
			$$\|G\|_{L^2(\R^2)}^2\leq cn^{-2}\int_\R|\hat{\varphi}(\kappa)|^2
			\kappa^2 \left[\|\psi_2''+\omega^2\perm(\omega)\psi_2\|_{L^2(\R)}^2
			+ \left(k_0+\frac{\kappa}{n}\right)^2\|\psi_2'\|_{L^2(\R)}^2\right]\dd\kappa\leq cn^{-2}.$$

			Now, we show $w^{(n)}\rightharpoonup 0$. Let $\eta\in L^2(\R^2,\C^3)$ be arbitrary. Because $r_n\to 0$ in $L^2(\R^2)$, we only need to show that
			\beq\label{E:un-weaktozero-2D}
			\begin{aligned}
				I_n:=&n^{-3/2}\int_{\R^2}\varphi_n\left(\frac{x_2-n^2}{n}\right)e^{\ri k_0 x_2}\psi(x_1) \cdot \overline{\eta}(x_1,x_2)\dd x_1 \dd x_2\to 0.
			\end{aligned}
			\eeq
			From the relation $\hat{\varphi}_n(k):=(k+nk_0)\hat{\varphi}(k)$ we get $\varphi_n(x_2)=-\ri \varphi'(x_2)+nk_0\varphi(x_2)$ and
			$$
			\begin{aligned}
				|I_n|=&n^{-1/2}\int_{\R^2}\left|k_0\varphi\left(\frac{x_2-n^2}{n}\right) \right||\psi(x_1)\cdot \overline{\eta}(x_1,x_2)| \dd x_1\dd x_2 \\
				+ &n^{-3/2}\int_{\R^2}\left|\varphi'\left(\frac{x_2-n^2}{n}\right) \right||\psi(x_1)\cdot \overline{\eta}(x_1,x_2)| \dd x_1\dd x_2\\
				=&n^{-1/2}\int_{\R^2\setminus \R\times [-n,n]}\left|k_0\varphi\left(\frac{x_2-n^2}{n}\right) \right||\psi(x_1)\cdot \overline{\eta}(x_1,x_2)| \dd x_1\dd x_2 \\
				+ &n^{-3/2}\int_{\R^2\setminus \R\times [-n,n]}\left|\varphi'\left(\frac{x_2-n^2}{n}\right) \right||\psi(x_1)\cdot \overline{\eta}(x_1,x_2)| \dd x_1\dd x_2
			\end{aligned}
			$$
			for $n$ large enough because of the compact support of $\varphi$. With the Cauchy-Schwarz inequality we arrive at
			$$|I_n|\leq c \|\psi\|_{L^2(\R)}\|\eta\|_{L^2(\R^2\setminus \R\times [-n,n])}\left(\|\varphi\|_{L^2(\R)} + \|\varphi'\|_{L^2(\R)}\right) \to 0.$$

			Because
				$$\llbracket \perm(\omega) \psi_1\rrbracket = \llbracket \psi_2\rrbracket =\llbracket \psi_3\rrbracket =\llbracket \pa_{x_1}\psi_3\rrbracket = 0,$$
				we get
				$$\llbracket \perm(\omega) \hat{w}^{(n)}_1\rrbracket = \llbracket \hat{w}^{(n)}_2\rrbracket =\llbracket \hat{w}^{(n)}_3\rrbracket =\llbracket \pa_{x_1}\hat{w}^{(n)}_3\rrbracket =0
				$$
				for almost all $k\in\R$. Also, all the $L^2$-conditions in \eqref{eq14} are clearly satisfied.

			The last condition for $w^{(n)}\in \cD_\omega$ (and subsequently for $(w^{(n)})$ to be a Weyl sequence) is the jump condition
				\beq\label{E:curl-jump}
				\llbracket \pa_{x_1}\hat{w}^{(n)}_2- \ri k \hat{w}^{(n)}_1\rrbracket =0
				\eeq
				for almost all $k\in\R$. This, however, fails because we have only $\llbracket \pa_{x_1}\psi_2- \ri k_0 \psi_1\rrbracket =0$, which guarantees \eqref{E:curl-jump} for $k=k_0$ but not for other $k\in \R$. To fix this, we modify the sequence $w^{(n)}$. The final Weyl sequence (in the Fourier variables $(x_1,k)$) is
				$$\hat{u}^{(n)}:=b_n(\hat{w}^{(n)}-\alpha_n \hat{s}), \quad s=(s_1, s_2, 0)^T,$$
				where
				\beq\label{E:cond-s}
				\nabla_k\cdot \hat{s}=0, \quad  \llbracket \hat{s}_2\rrbracket=\llbracket \perm\hat{s}_1\rrbracket=0, \quad  \llbracket \pa_{x_1}\hat{s}_2-\ri k  \hat{s}_1\rrbracket=1,
				\eeq
				$$\alpha_n=\alpha_n(k):=\llbracket \pa_{x_1}\hat{w}^{(n)}_2-\ri k\hat{w}^{(n)}_1\rrbracket \in\C,$$
				and $b_n>0$ is selected to ensure $\|u^{(n)}\|_{L^2(\R^2)}=1.$

			In order for $\hat{s}$ to satisfy \eqref{E:cond-s}, we need $$\hat{s}_2=\frac{\ri}{k}\pa_{x_1}\hat{s}_1, \quad \llbracket \pa_{x_1}\hat{s}_1\rrbracket=\llbracket\perm \hat{s}_1\rrbracket=0, \quad \text{and} \quad \llbracket -\pa_{x_1}^2\hat{s}_1+k^2\hat{s}_1\rrbracket=\ri k.$$
				The following choice fulfils the above conditions as well as the $L^2$-conditions in \eqref{eq14}:
				$$ \hat{s}_1(x_1,k):=\begin{cases}
					\perm_-\zeta(x_1), & x_1>0,\\
					\left(\perm_++\frac{1}{2}(\ri k -k^2 (\perm_- - \perm_+))x_1^2\right)\zeta(x_1), & x_1<0,
				\end{cases}$$
				and
				$$\hat{s}_2:=\frac{\ri}{k}\pa_{x_1} \hat{s}_1,$$
				where $\zeta \in C_c^\infty(\R)$ is chosen such that $\zeta(0)=1$ and $\zeta'(0)=\zeta''(0)=0.$ The coefficient $\alpha_n$ is easily calculated to be
				$$\alpha_n(k) = -c_nn^{-1/2}\llbracket\psi_2'+\ri k  \psi_1 \rrbracket n(k-k_0)\hat{\varphi}(n(k-k_0))e^{-\ri (k-k_0)n^2}.$$

			Because
				\beq\label{E:alpha-est}
				\begin{aligned}
					\|\alpha_n \hat{s}\|_{L^2(\R^2)}^2
					&\leq cn^{-1} \int_\R (1+k^2)^2n^2(k-k_0)^2|\hat{\varphi}(n(k-k_0))|^2\dd k\\
					&\leq cn^{-2}\int_\R \kappa^2 \left(1+(k_0+\tfrac{\kappa}{n})^4\right)|\hat{\varphi}(\kappa)|^2\dd \kappa \leq cn^{-2} \to 0,
				\end{aligned}
				\eeq
				and $\|\hat{w}^{(n)}\|=1$, the coefficients $b_n$ are bounded and the weak convergence $u^{(n)}\rightharpoonup 0$ holds. Finally, to check the convergence $L(\omega)(\alpha_n s)\to 0$, note that
				$$L_k(\omega)(\alpha_n \hat{s})=\alpha_n\begin{pmatrix}
					k^2\hat{s}_1+\ri k \pa_{x_1}\hat{s}_2-\omega^2\perm(\omega)\hat{s}_1\\
					\ri k \pa_{x_1}\hat{s}_1-\pa_{x_1}^2\hat{s}_2-\omega^2\perm(\omega)\hat{s}_2\\
					0
				\end{pmatrix}.$$
				Similar estimates to \eqref{E:alpha-est} show that each term in the above expression converges to 0 in the $L^2$-norm.
			\epf

			%----------------------------------------------------------
			\subsection{Proof of Theorem \ref{T:main2D}}

			Combining Lemmas \ref{L:ess-sp2D-x1} and \ref{L:ess-sp2D-x2} yields, since $\Omega_0 \cap (M_+\cup M_-\cup N)=\emptyset$,
			\beq\label{E:Weyl-sp-2D}
			\sweyl(\cL)\setminus \Omega_0\supset M_+\cup M_- \cup N.
			\eeq
			By Proposition \ref{th2} we have
			\beq\label{E:sp-2D-red}
			\sigma(\cL)\setminus \Omega_0\subset  M_+\cup M_-\cup N.
			\eeq
			Equations \eqref{E:Weyl-sp-2D} and \eqref{E:sp-2D-red} imply
			$$\sigma(\cL)\setminus \Omega_0 = \sweyl(\cL)\setminus \Omega_0= M_+\cup M_- \cup N,$$
			which proves \eqref{E:2Dsp-red} and the second equality in \eqref{E:2Dess-red}.

			Equation \eqref{E:2Dptsp-red} is the statement of Lemma \ref{L:pt-sp-2D-red}.

			Using Remark \ref{rem:essential}, we get $\sweyl(\cL)=\stwo(\cL)$ and hence also  $\sweyl(\cL)\setminus \Omega_0=\stwo(\cL)\setminus \Omega_0$. Because $\sigma(\cL)\setminus \Omega_0=\sweyl(\cL)\setminus \Omega_0$ and $\sone \subset \stwo\subset \sthree\subset \sfour\subset \sfive\subset \sigma$, we arrive at  $\sone(\cL)\setminus \Omega_0 \subset \stwo(\cL)\setminus \Omega_0= \sthree(\cL)\setminus \Omega_0= \sfour(\cL)\setminus \Omega_0 = \sfive(\cL)\setminus \Omega_0=\sweyl(\cL)\setminus \Omega_0$.

			Finally, we show that $\stwo(\cL) \setminus \Omega_0\subset \sone(\cL)\setminus \Omega_0$.
			Recall that
			$$
			\begin{aligned}
				\sone(\cL)&=\{\omega\in \C: \Ran(L(\omega)) \text{ is not closed or } \dim\ker(L(\omega))= \dim\coker(L(\omega))=\infty\},\\
				\stwo(\cL)&=\{\omega\in \C: \Ran(L(\omega)) \text{ is not closed or } \dim\ker(L(\omega))=\infty\}.
			\end{aligned}
			$$
			Let $\omega \in \C \setminus \Omega_0$. As we have shown in Lemma \ref{L:pt-sp-2D-red}, the equality $\sigma_p^{{\rm red}}(\cL)=\emptyset$ holds; hence $\ker L(\omega)=\{0\}$. In conclusion, $\sone(\cL) \setminus \Omega_0=\stwo(\cL)\setminus \Omega_0$. This proves the statement about the essential spectra in \eqref{E:2Dess-red} and completes the proof of Theorem \ref{T:main2D}.

			%----------------------------------------------------------
			\subsection{Proof of Theorem \ref{T:main2D_Om0}}

			Equation \eqref{E:2Dptsp-Om0} is the statement of Lemma \ref{L:pt-sp-2D-sing} and the second equality in \eqref{E:2Dess-Om0} (and consequently \eqref{E:2Dsp-Om0}) follows from Lemma \ref{L:ess-sp2D-x1}. Thus also the first equality in \eqref{E:2Dess-Om0} holds for $j=2,\dots,5$.

			Once again, as $\sone \subset \stwo \subset \dots \subset \sfive$ and $\stwo=\sigma_{\text{Weyl }}$, the only remaining part to be shown is $\sone(\cL)\cap\Omega_0=\stwo(\cL)\cap\Omega_0$. For that let $\omega \in \Omega_0$ and w.l.o.g. assume $\omega^2\perm_+(\omega)=0$. We show that in this case $\Ran(L(\omega))\subset \cR$ is not closed, and therefore $\omega$ lies in  $\sone(\cL)$.

			For the non-closedness of  $\Ran(L(\omega))$ we use Lemma \ref{L:ran_not-closed}. Hence, it suffices to find a Weyl sequence in $V(\omega):=\ker(L(\omega))^\perp\cap \cD_\omega$. W.l.o.g. we assume $\omega^2\perm_+(\omega)=0$.
			We construct a suitable Weyl sequence $(u^{(n)})$ by choosing a bounded solution of $L(\omega)\xi=0$ on $\R^2_+$ and smoothly cutting it off to have a compact support with the support moving out to $x_1=\infty$ as $n\to \infty$. The simplest bounded solution is $\xi = \bspm 0\\0\\1\espm$.

			Let $w\in C_c^\infty(\R^2_+,\R)$ such that $\|\Delta w\|_{L^2(\R^2)}=1$. Set $\varphi=\Delta w$ and
			$$u^{(n)}(x_1,x_2):=\frac{1}{n}\varphi\left(\frac{x_1-n^2}{n},\frac{x_2}{n}\right)\bspm 0\\0\\1\espm.$$
			Then $\|u^{(n)}\|_{L^2(\R^2)}=1$ for all $n$ and $u^{(n)}_3=\Delta w^{(n)}$, where $w^{(n)}(x_1,x_2)=nw\left(\frac{x_1-n^2}{n},\frac{x_2}{n}\right)\in C_c^\infty(\R^2_+,\R)$.

			We first show that this sequence lies in the orthogonal complement of the kernel. From  \eqref{E:curl}, we see that if $\psi\in \ker(L(\omega))$, then $\psi_3$ is harmonic. Therefore, for such $\psi$ we have
			$$\int_{\R^2} \overline{\psi}^T u^{(n)}\dd x   = \int_{\R^2} \overline{\psi_3} u^{(n)}_3 \dd x  = \int_{\R^2} \overline{\psi_3} \Delta w^{(n)}\dd x= \int_{\R^2} \overline{\Delta\psi_3} w^{(n)}\dd x=0,$$
			where the boundary terms vanish since $w^{(n)}\in C_c^\infty(\R^2_+,\R)$. Hence, $u^{(n)}\in\ker(L(\omega))^\perp$.

			The divergence condition is trivially satisfied by $u^{(n)}$ and so are the interface conditions for all $n$ large enough.
			Moreover,
			after Fourier transform in the $x_2$-variable, we have
			$$\hat{u}^{(n)}(x_1,k)=\hat\varphi\left(\frac{x_1-n^2}{n},nk\right)\bspm 0\\0\\1\espm,$$
			where $\varphi\in C^\infty_c(\R^2_+,\R)$ and $\|\hat\varphi\|_{L^2(\R^2)}=1$ for all $n$.  Because
			$$
			\begin{aligned}
				\widehat{L(\omega)u^{(n)}}(x_1,k)=L_k(\omega)\hat{u}^{(n)}(x_1,k)&=(0,0,k^2-\pa_{x_1}^2)^T\left(\hat\varphi\left(\frac{x_1-n^2}{n},nk\right)\right)\\
				&=(0,0,(k^2\hat\varphi-n^{-2}\pa_{x_1}^2\hat\varphi)(\tfrac{x_1-n^2}{n},nk))^T,
			\end{aligned}
			$$
			we get
			$$\|L(\omega)u^{(n)}\|_{L^2(\R^2)}^2=n^{-4}\int_{\R^2}\left|\kappa^2\hat\varphi(s,\kappa)-\pa_{x_1}^2\hat\varphi(s,\kappa)\right|^2\dd \kappa\dd s \to 0 \quad (n\to \infty).$$
			Finally, showing $u^{(n)} \rightharpoonup 0$ is analogous to the same step in the proof of Lemma \ref{L:ess-sp2D-x1}. Namely, for any $\eta \in L^2(\R^2,\R)$ we have
			$$
			\begin{aligned}
				\left|\int_{\R^2}\hat\varphi\left(\frac{x_1-n^2}{n},nk\right) \eta(x_1,k)\dd x_1\dd k\right|&=\left|\int_{[n,\infty)\times \R}\hat\varphi\left(\frac{x_1-n^2}{n},nk\right) \eta(x_1,k)\dd x_1\dd k\right| \text{  (for $n$ large enough)}\\
				&\leq \left(\int_{\R^2}\left|\hat\varphi\left(\frac{x_1-n^2}{n},nk\right) \right|^2\dd x_1\dd k\right)^{1/2}\|\eta\|_{L^2([n,\infty)\times \R)}\\
				&= \|\hat\varphi\|_{L^2(\R^2)}^2\|\eta\|_{L^2([n,\infty)\times \R)}=\|\eta\|_{L^2([n,\infty)\times \R)}\to 0 \quad (n \to \infty).
			\end{aligned}
			$$

			The case $\omega^2\perm_-(\omega)=0$ is treated analogously with $u^{(n)}(x_1,x_2):=\varphi\left(\frac{x_1+n^2}{n},\frac{x_2}{n}\right)\rm{e}_3,$ where $\varphi=\Delta w$ and $w\in C^\infty_c(\R^2_-,\R)$.

			This completes the proof of Theorem \ref{T:main2D_Om0}.

			%-----------------------------------------------------------
			\section{Implications for the Time Dependent Maxwell Equations}
			\label{S:t-dep}

			In this section we explain how at some of the spectral points one can generate solutions of the time dependent Maxwell equations.

			Maxwell's equations \eqref{E:Maxw-1stord} (with $\boldsymbol{\mu}_0=1$) derive from the time dependent problem
				\begin{equation}\label{E:Maxw-tdep}
					\begin{aligned}
						\nabla\times \cE & = - \pa_t \cH, \\
						\nabla\times \cH & = \pa_t \cD, \quad \cD =\perm_0 \left(\cE+\int_\R \chi(t-s)\cE(s)\dd s\right),\\
						\nabla \cdot \cD & =0, \quad \nabla \cdot \cH =0
					\end{aligned}
				\end{equation}
				by substituting the ansatz
				\beq\label{E:EH-ansatz}
				(\mathcal E, \mathcal H)(x,t)=(E,H)(x)e^{-\ri\omega t}+(\overline{E},\overline{H})(x)e^{\ri\bar{\omega} t}, \omega \in\C.
				\eeq
				The function $\hat{\chi}$ in \eqref{E:Maxw-1stord} is the temporal Fourier transform of the electric susceptibility $\chi$, i.e. (suppressing the $x$-dependence)
				\beq\label{E:tFT}
				\hat{\chi}(\omega)=\int_\R \chi(t)e^{\ri\omega t}\dd t.
				\eeq
				Equations \eqref{E:Maxw-1stord} collect all terms proportional to $e^{-\ri\omega t}$. The complex conjugated equations \eqref{E:Maxw-1stord} then collect terms proportional to $e^{\ri\bar{\omega}t}$.

			Equations \eqref{E:Maxw-tdep} describe a delayed response of the medium to the applied electromagnetic field, so called material dispersion.
				In order to satisfy causality, $\chi(t)$ must vanish for all $t<0$, i.e., the displacement field $\cD$ must not depend on future values of $\cE$. Hence, in this section, we restrict to causal electric susceptibilities $\chi$.

			If $\hat{\chi}(\omega)$ is well defined, i.e. the integral in \eqref{E:tFT} converges, and if $\omega\neq 0$, then any solution $E$ of \eqref{E:Maxw-2ndord} together with $H:=\frac{1}{\ri\omega}\nabla \times E$ generates a solution of \eqref{E:Maxw-tdep} on the whole time axis via the ansatz in \eqref{E:EH-ansatz}. For $\omega=0$ the resulting fields are constant in time. Then $H$ is a gradient of a harmonic function because $\curl H =0$ and $\div H =0$. Hence also $H$ is harmonic in $\R^n$ ($n=1$ or $n=2$). If we restrict ourselves to bounded solutions, we obtain $H=\text{const}$.

			Clearly, if $\chi \in L^1_\rho(\R_+):=\{f:L^1_\text{loc}(\R_+,\R): \int_0^\infty |f(t)|e^{\rho t}\dd t <\infty\}$, then  $\hat{\chi}(\omega)$ is well defined for all $\omega \in \C$ with $\Im(\omega)\geq -\rho$. Hence, the following is a necessary condition for a solution $E$  \eqref{E:Maxw-2ndord} corresponding to a spectral point $\omega\in \C$ to produce a solution of \eqref{E:Maxw-tdep}:
				\beq\label{E:im-om-cond}
				\Im(\omega)> -\sup\{\rho\in \R: \chi\in L^1_\rho(\R_+)\}.
				\eeq

			Note that for $\omega \in \R$ one can, of course, generalize the definition of the Fourier transform also for tempered distributions $\chi$, like, e.g., the Drude susceptibility $\chi^D$ below. Then equation \eqref{E:Maxw-tdep} would have to be interpreted in an appropriate distributional sense.

			Let us now revisit the three examples of $\hat{\chi}$ considered in the paper. For (non-dispersive) dielectrics we have used $\hat{\chi}\equiv \eta>0$. This corresponds to the instantaneous response, i.e. formally $\chi(t)=\eta\delta(t)$. In a non-dispersive dielectric the displacement field in \eqref{E:Maxw-tdep} is given simply by $\cD=\perm_0(1+\eta)\cE$.

			For metals we have studied two examples: the Drude model \eqref{E:Drude} and the Lorentz model \eqref{E:Lorentz}.

			\paragraph{The Drude Model.}
				The time dependent susceptibility of the Drude model in \eqref{E:Drude} is
				$$\chi^D(t):=\frac{c_D}{\gamma}(1-e^{-\gamma t})\theta(t),$$
				where $\theta$ is the Heaviside function, see \cite{PBK2011}. Due to the term $\tfrac{c_D}{\gamma}\theta(t)$ we have $\chi^D\in L^1_\rho(\R_+)$ for any $\rho<0$ but not for $\rho=0$. For the Drude model we conclude that functions \eqref{E:EH-ansatz} with $\Im(\omega)\leq 0$ cannot be solutions of \eqref{E:Maxw-tdep}!

			The spectrum for the interface of a dielectric with the Drude metal was studied in Example \ref{Ex:1D-Drude} in the one dimensional case and in Example \ref{Ex:2D-Drude} in two dimensions. For the Drude model none of the spectral points generates a solution of the time dependent problem because
				$$\sigma(\cL_k), \sigma(\cL)\subset \{\omega\in \C: \Im(\omega)\leq 0\},$$
				i.e. condition \eqref{E:im-om-cond} is not satisfied.

				\begin{remark}
					This apparent problem of the Drude model is not unknown in the physics community. It originates in the non-unique distinction between the current density $\mathcal{J}$ and the polarization field $\mathcal{P}$ in the macroscopic Maxwell equation
					$\nabla \times \cH = \pa_t \cD + \mathcal{J}$, \cite{Busch}. The constitutive relation is $\mathcal{D}=\perm_0 \cE + \cP$, hence $\mathcal{J}$ and $\pa_t \cP$ play the same role in the equation. In \eqref{E:Maxw-tdep} we have $\cP=\perm_0\int_\R\chi(t-s)\cE(s)\dd s$ and $\mathcal{J}=0$. If, instead, one chooses $\cP=0$ and defines $\mathcal{J}$ on the Fourier side as $-\ri\omega \hat{\chi}(\omega) E$ with the Drude model for $\hat{\chi}$, i.e.,
					$$J=J(E)(x,\omega) := \frac{c_D}{\gamma-\ri\omega}E(x),$$
					then $J(E)(x,\omega)e^{-\ri\omega t}=\int_\R c_D \theta(t-s)e^{-\gamma(t-s)}e^{-\ri\omega s}\dd s$, which converges for all $$\Im(\omega)>-\gamma.$$
					Understanding the first equation in \eqref{E:Maxw-2ndord} as $\nabla\times \nabla \times E -\omega^2 \perm_0 E - \ri\omega \perm_0 J(E)=0$, one obtains via \eqref{E:EH-ansatz} a solution of \eqref{E:Maxw-tdep} with the second line replaced by
					$$\nabla \times \cH = \pa_t \cD + \mathcal{J}, \ \cD =\perm_0 \cE, \ \mathcal{J} = \int_\R c_D \theta(t-s)e^{-\gamma(t-s)} \cE(s)\dd s.$$
					Unfortunately, due to the modified $\cD$ the divergence condition in \eqref{E:Maxw-2ndord} changes and the solutions $E$ computed above do not solve this new formulation of the Maxwell problem.
				\end{remark}

			\paragraph{The Lorentz Model.}
				The time dependent susceptibility of the Lorentz model in  \eqref{E:Lorentz} is
				$$\chi^L(t):=\frac{c_L}{\sqrt{\omega_*^2-\frac{\gamma^2}{4}}} e^{-\frac{\gamma}{2}t}\sin\left(\sqrt{\omega_*^2-\frac{\gamma^2}{4}}t\right)\theta(t),$$
				see \cite{PBK2011}. $\chi^L$ is defined (and real) if $\omega_*>\gamma/2$. Clearly, $\chi^L\in L^1_\rho(\R_+)$ if and only if $\rho<\gamma/2$.

			The spectrum for the interface of a dielectric and a Lorentz metal is plotted in Examples \ref{Ex:1D-Lorentz}  and \ref{Ex:2D-Drude}. One observes that the spectrum lies in the strip $\{\omega\in \C: -\tfrac{\gamma}{2} < \Im(\omega) \leq 0\},$ which is compatible with the condition \eqref{E:im-om-cond}.

			In the one dimensional case, if $\omega\in \sigma_p(\cL_k)$, then the corresponding eigenfunction $u\in D(\cL_k)$ produces an $L^\infty(\R,\cD_{k,\omega})$ solution of \eqref{E:Maxw-2ndord} via the ansatz $E(x) = u(x_1)e^{\ri kx_2}$ and, as described above, if $\omega\neq 0$, then also a solution of the time dependent equations \eqref{E:Maxw-tdep} via the ansatz \eqref{E:EH-ansatz}. If $\omega\in \sigma_{\text{Weyl }}(\cL_k)$, then no solution of \eqref{E:Maxw-2ndord} is available. One merely has a Weyl sequence based on truncated plane waves on either side of the interface. Because of the truncation these are not true solutions. Without the truncation these functions do not satisfy the interface condition.

			It is expected that the Weyl spectrum in $M^{(k)}_+\cup M^{(k)}_-$ describes radiation waves travelling to $x_1\to\pm \infty$ in a time dependent scattering problem \cite{HL2007}.

			It is an open problem to describe the significance of the five different kinds of essential spectrum for the time evolution.

			In two dimensions, besides $\sigma_p$ also a subset of the Weyl spectrum, namely the set $N$, generates time dependent solutions. This is because a Weyl sequence at $\omega\in N$ is a product of an eigenfunction of the one dimensional problem in $x_1$ and a truncated plane wave in $x_2$. Dropping the truncation, we obtain a bounded solution $E(x_1,x_2)=\psi(x_1)e^{\ri k_0 x_2}$ of \eqref{E:Maxw-2ndord}, where $\psi$ is an eigenfunction of $\cL_{k_0}$ in one dimension, see the proof of Lemma \ref{E:NsubWeyl}.

			Also here the Weyl spectrum in $M_+\cup M_-$ is expected to describe radiation travelling to $x_1\to\pm \infty$.

			\section{Conclusion}\label{S:concl}

			This paper is an early contribution to the study of the spectrum generated by non-selfadjoint operator pencils associated with interface problems for Maxwell equations with the temporal frequency $\omega$ being the spectral parameter. We make no assumptions on the way the dielectric function depends on the frequency, which means that our results can be applied to very general situations where two materials meet at a flat interface. On the other hand, our results here only consider very specific reductions of the problem to one- and two-dimensional situations. Therefore, the results could be extended to much more general situations: the full three-dimensional problem, inhomogeneous media in the half-spaces and even a non-flat interface. In upcoming work we intend to address some of these questions.

			One of the contributions of the current paper is the detailed look at the different types of essential spectrum that arise in this model. It is an interesting question as to how the different types of essential spectrum affect the time-dependent problem or whether these distinctions are purely of mathematical interest. To the best of our knowledge, this  has not been looked at in the non-selfadjoint situation.

			%-----------------------------------------------------------
			\appendix
			\section{Integration by Parts for $\nabla_k\times$ and $\nabla\times $}

			This and the two following appendices include some classical result, which we provided for convenience.

			We provide integration by parts formulas for $L^2$ functions $u$ and $v$ with their curl lying also in $L^2$. These are used in Lemma \ref{L:HS-1D} and \ref{L:HS-2D} to justify the formulas
			$$\int_\R (\nabla_k \times \nabla_k \times u)\cdot \overline{u} \dd x = \int_\R |\nabla_k \times u|^2 \dd x \quad \forall u\in \cD_{k,\omega}$$
			and
			$$\int_{\R^2} (\nabla \times \nabla \times u)\cdot \overline{u} \dd x = \int_{\R^2} |\nabla \times u|^2 \dd x \quad \forall u \in \cD_{\omega}$$
			respectively.

			\blem\label{L:PI-1D}
			Let $u,v\in L^2(\R,\C^3)$ satisfy $\nabla_k\times u,\nabla_k\times v \in L^2(\R,\C^3)$. Then
			$$\int_\R (\nabla_k \times u)\cdot \overline{v} \dd x = \int_\R u\cdot \overline{(\nabla_k \times v)} \dd x.$$
			\elem
			\bpf
			First note that due to the form
			$$\nabla_k\times \varphi = (\ri k\varphi_3, -\varphi'_3, \varphi'_2-\ri k\varphi_1)^T$$
			we have that $\varphi, \nabla_k\times \varphi\in L^2(\R,\C^3)$ is equivalent to $\varphi_1\in L^2(\R),\varphi_2, \varphi_3\in H^1(\R)$. For $u$ and $v$ we thus have $u_2,u_3, v_2, v_3\in H^1(\R)$. The calculation
			$$
			\begin{aligned}
				\int_\R (\nabla_k \times u)\cdot \overline{v} \dd x &= \int_\R\ri k u_3\overline{v_1} - u_3'\overline{v_2}+ (u_2'-\ri k u_1)\overline{v_3} \dd x\\
				&=\int_\R u_1\overline{\ri k  v_3} - u_2\overline{v_3'}+ u_3 \overline{(v_2'-\ri k v_1)} \dd x =  \int_\R u\cdot \overline{(\nabla_k \times v)} \dd x
			\end{aligned}
			$$
			clearly holds for $v_2,v_3 \in C^\infty_c(\R)$ and thus by a density argument also in our case.
			\epf
			\blem\label{L:PI-2D}
			Let $u,v\in L^2(\R^2,\C^3)$ satisfy $\nabla\times u,\nabla\times v \in L^2(\R^2,\C^3)$. Then
			$$\int_{\R^2} (\nabla \times u)\cdot \overline{v} \dd x = \int_{\R^2} u\cdot (\nabla \times \overline{v}) \dd x.$$
			\elem
			\bpf
			First, we show that if $\varphi_1,\varphi_2, \pa_{x_1}\varphi_2-\pa_{x_2}\varphi_1 \in L^2(\R^2)$, then
			\beq\label{E:PI-1}
			\int_{\R^2} (\pa_{x_1}\varphi_2-\pa_{x_2}\varphi_1)\psi\dd x  = \int_{\R^2} \varphi_1\pa_{x_2}\psi-\varphi_2\pa_{x_1}\psi\dd x \quad \forall \psi\in H^1(\R^2).
			\eeq
			For $\psi\in C^\infty_c(\R^2)$ this holds since
			$$\int_{\R^2} (\pa_{x_1}\varphi_2-\pa_{x_2}\varphi_1)\psi\dd x  = \pa_{x_1}\varphi_2 [\psi]-\pa_{x_2}\varphi_1[\psi] = \varphi_1[\pa_{x_2}\psi]-\varphi_2[\pa_{x_1}\psi] = \int_{\R^2} \varphi_1\pa_{x_2}\psi-\varphi_2\pa_{x_1}\psi\dd x,$$
			where, e.g., $\pa_{x_1}\varphi_2 [\psi]$ is the distributional action of $\pa_{x_1}\varphi_2$ on $\psi$, which equals  $-\varphi_2[\pa_{x_1}\psi]$ by the definition of the distributional derivative. The final step holds since $\varphi_1,\varphi_2\in L^2(\R^2)$. To obtain \eqref{E:PI-1} recall that $C^\infty_c(\R^2)$ is dense in $H^1(\R^2)$.

			For the statement of the lemma note that
			$$\nabla\times u = (\pa_{x_2}u_3, -\pa_{x_1}u_3, \pa_{x_1}u_2-\pa_{x_2}u_1)^T$$
			and hence $u,v, \nabla\times u, \nabla\times v\in L^2(\R^2,\C^3)$ is equivalent to $u_1,u_2,v_1,v_2, \pa_{x_1}u_2-\pa_{x_2}u_1, \pa_{x_1}v_2-\pa_{x_2}v_1 \in L^2(\R^2)$ and $u_3, v_3\in H^1(\R^2)$.
			Hence, we can use \eqref{E:PI-1} for both $(\varphi,\psi)=(u,\overline{v_3})$ and $(\varphi,\psi)=(\overline{v},u_3)$. As a result
			$$
			\begin{aligned}
				&\int_{\R^2}(\nabla \times u)\cdot \overline{v} \dd x =  \int_{\R^2} \overline{v_1} \pa_{x_2}u_3 - \overline{v_2}\pa_{x_1}u_3 \dd x  + \int_{\R^2} (\pa_{x_1}u_2 - \pa_{x_2}u_1) \overline{v_3} \dd x \\
				&=\int_{\R^2} u_3(\pa_{x_1}\overline{v_2} - \pa_{x_2}\overline{v_1})\dd x + \int_{\R^2} u_1\pa_{x_2}\overline{v_3} - u_2\pa_{x_1} \overline{v_3} \dd x  = \int_{\R^2} u\cdot(\nabla\times \overline{v})\dd x.
			\end{aligned}
			$$
			\epf

			\section{Interface Conditions in 1D} \label{App:interf-1D}
			In Appendix \ref{App:interf-1D} we prove the equivalence of the representation of $\cD_{k,\omega}$ as a space of functions with a weak curl and weak curl-curl over $\R$ and as a space with such weak derivatives over $\R_+$ and $\R_-$ equipped with interface conditions.

			For the one-dimensional case we first provide a natural definition of the distributional and the weak divergence and the curl in $L^2$ for our non-standard gradient $\nabla_k=(\pa_{x_1}, \ri k ,0)^T$.

			Let $\Omega\subset \R$ be open. For $u\in L^2(\Omega,\C^3)$ we define the distributional $\nabla_k\cdot u$ by
			$$(\nabla_k\cdot u) (\varphi) =  -\int_\Omega u\cdot \nabla_{-k} \varphi \dd x \qquad \forall \varphi \in C^\infty_c(\Omega,\C).$$
			The divergence $\nabla_k\cdot u$ is called weak if $\nabla_k\cdot u \in L^2(\Omega,\C)$ (i.e. $(\nabla_k\cdot u) (\varphi)=\int_\Omega w \varphi  \dd x \  \forall \varphi \in C^\infty_c(\Omega,\C)$ for some $w=:\nabla_k\cdot u \in L^2(\Omega,\C)$).\\
			The distributional $\nabla_k\times u$ and $\nabla_k\times \nabla_k\times u$ are defined by
			$$(\nabla_k\times u)(\varphi)=  \int_\Omega u\cdot \nabla_{-k} \times \varphi \dd x \qquad \forall \varphi \in C^\infty_c(\Omega,\C^3)$$
			and
			$$(\nabla_k\times \nabla_k\times u)(\varphi)=  \int_\Omega u\cdot \nabla_{k} \times \nabla_k\times \varphi \dd x \qquad \forall \varphi \in C^\infty_c(\Omega,\C^3)$$
			respectively.
			The curl $\nabla_k\times u$ and the curl-curl $\nabla_k\times \nabla_k\times u$ are called weak if $\nabla_k\times u\in L^2(\Omega,\C^3)$ and $\nabla_k\times \nabla_k\times u\in L^2(\Omega,\C^3)$ resp. (i.e. $(\nabla_k\times u)(\varphi)=\int_\Omega w \cdot \varphi \dd x \ \forall \varphi\in C^\infty_c(\Omega,\C^3)$ for some $w=:\nabla_k\times u \in L^2(\Omega,\C^3)$ and $(\nabla_k\times \nabla_k\times u)(\varphi)=\int_\Omega w \cdot \varphi \dd x \ \forall \varphi\in C^\infty_c(\Omega,\C^3)$ for some $w=:\nabla_k\times \nabla_k\times u \in L^2(\Omega,\C^3)$ resp.).

			\blem\label{lem:1Dinterface}
			We have
			\begin{align}\label{E:dkomega}
				\cD_{k,\omega}=&\{u\in L^2(\R,\C^3): \ \nabla_k\times u_\pm, \nabla_k\times \nabla_k\times u_\pm \in L^2(\R_\pm,\C^3), \\ \nonumber
				&\quad \text{the divergence condition} \ \eqref{E:div-cond-1d}, \ \text{and the interface conditions} \ \eqref{E:IFC-1D} \ \text{hold}\}.
			\end{align}
			\elem

			\bpf
			We first assume $u\in \cD_{k,\omega}$. Clearly, $u_\pm \in L^2(\R_\pm,\C^3)$ and $\nabla_k\times u_\pm, \nabla_k\times \nabla_k\times u_\pm \in L^2(\R_\pm,\C^3).$

			Because $\nabla_k\times u=(\ri k u_3, -u_3', u_2'-\ri k u_1)^T \in L^2(\R,\C^3)$, we automatically have $u_2,u_3\in H^1(\R,\C)$. The distributional identity $\nabla_k\cdot (\perm u) = 0$ is equivalent to
			$$\int_\R\perm u_1\varphi'\dd x=\ri k \int_\R\perm u_2\varphi\dd x \qquad \forall \varphi \in C^\infty_c(\R,\C)$$
			and hence $-\ri k \perm u_2 \in L^2(\R,\C)$ is the weak derivative of $\perm u_1$. As $\perm \in L^\infty(\R,\C)$, we get also $\perm u_1\in L^2(\R,\C)$, concluding that $\perm u_1\in H^1(\R,\C)$. Hence, the continuity of $\perm u_1, u_2$, and $u_3$ follows from Sobolev's embedding theorem.
			Next, since  $\nabla_k\times \nabla_k\times u=(k^2u_1 + \ri k u_2', \ri k u_1'-u_2'', k^2u_3-u_3'')^T \in L^2(\R,\C^3)$, it follows that $u_3', \ri ku_1 -u_2' \in H^1(\R,\C)$. Thus also the continuity of $u_3'$ and $u_2'-\ri k u_1$ follows.

			The equations $\perm_\pm(\omega) \nabla_k\cdot u_\pm =0 \text{ on } \R_\pm$ follow from $\int_\R  \perm u \cdot \nabla_{-k} \varphi\dd x =0$ for all $\varphi\in C^\infty_c(\R,\C)$ by the choice of test functions in $C^\infty_c(\R_\pm,\C)$.

			For the opposite direction we assume $u$ lies in the set on the right hand side of \eqref{E:dkomega}. Analogously to the first part one shows that $\perm u_{\pm,1}, u_{\pm,2},  u_{\pm,3}, u_{\pm,3}',$ $\ri ku_{\pm,1} -u_{\pm,2}'\in H^1(\R_\pm,\C)$.
			By Sobolev's embedding theorem, see Theorem 4.12 in \cite{AF2003}, any $f\in H^1(\R_+)$ is continuous on $\overline{\R_+}=[0,\infty)$ and analogously for $f\in H^1(\R_-)$. Hence, the jumps in \eqref{E:IFC-1D} are well defined in the sense of \eqref{E:jump}.

			%That the one sided limits of these quantities exist at $x_1=0$ was shown already in Remark \ref{R:interf-well-def}.

			Next, we show that $w:=\chi_{\R_+}\nabla_k\times u_+ + \chi_{\R_-}\nabla_k\times u_-$  is the weak $\nabla_k \times$ of $u=\chi_{\R_+}u_+ + \chi_{\R_-}u_-$ on $\R$. This follows by integration by parts using $\llbracket u_3\rrbracket=\llbracket u_2\rrbracket=0 $. Indeed, for any $\Phi \in C^\infty_c(\R,\C^3)$
			$$
			\begin{aligned}
				\int_\R w\cdot \Phi\dd x &= \sum_{\pm}\int_{\R_\pm} \ri k u_{\pm,3}  \Phi_1 - u_{\pm,3}' \Phi_2+(u_{\pm,2}'-\ri k u_{\pm,1}) \Phi_3\dd x\\
				&= \int_\R (-u_1\ri k \Phi_3 - u_2\Phi_3') + u_3 (\Phi_2' +\ri k \Phi_1)\dd x + \llbracket u_3\rrbracket\Phi_2(0)-\llbracket u_2\rrbracket\Phi_3(0)\\
				&= \int_\R u\cdot \nabla_{-k}\times \Phi\dd x.
			\end{aligned}
			$$
			Analogously, using $\llbracket u_2\rrbracket=\llbracket u_3\rrbracket=\llbracket u_3'\rrbracket=\llbracket u_2'-\ri k u_1\rrbracket = 0$, one shows that $w:=\chi_{\R_+}\nabla_k\times\nabla_k\times u_+ + \chi_{\R_-}\nabla_k\times\nabla_k\times u_-$  is the weak $\nabla_k\times\nabla_k \times$ of $u$.

			It remains to prove the distributional equality $\nabla_k\cdot (\perm u)=0.$ Again, via the integration by parts, we have for any $\varphi \in C^\infty_c(\R,\C)$
			$$
			\begin{aligned}
				\int_\R (\perm u)\cdot \nabla_{-k} \varphi\dd x &= \sum_{\pm}\int_{\R_\pm} \perm_\pm u_\pm \cdot \nabla_{-k} \varphi \dd x\\
				&= - \sum_{\pm}\int_{\R_\pm} \perm_\pm \nabla_k \cdot u_\pm \varphi \dd x - \varphi(0)\llbracket \perm u_1\rrbracket\\
				&=0
			\end{aligned}
			$$
			because $\llbracket \perm u_1\rrbracket =0$ and $\perm_\pm \nabla_k\cdot u_\pm =0$.
			\epf

			%-----------------------------------------------------------------

			\section{Interface Conditions in 2D}\label{App:traces2D}
			In this appendix  we include some standard analysis of traces in $\Hdiv$ and $\Hcurl$ for convenience of the readers and explain  that  $\cD_{\omega}$ can be characterised as stated in \eqref{E:domega} using interface conditions, see Lemma \ref{L:Domega-equiv}.
			This follows primarily from smoothness properties of the elements of the spaces $\Hdiv$ and $\Hcurl$ defined for any $\Omega\subset \R^2$ as
			$$\Hdiv(\Omega):=\{u \in L^2(\Omega,\C^3): \nabla \cdot u \in L^2(\Omega,\C)\}, \quad \Hcurl(\Omega):=\{u \in L^2(\Omega,\C^3): \nabla \times u \in L^2(\Omega,\C^3)\},$$
			where $\nabla:=(\pa_{x_1},\pa_{x_2},0)^T$.

			Let $\Gamma := \{x\in \R^2:x_1=0\}$ and $\nu:= \rm{e}_1$ i.e. the unit normal on $\Gamma$ pointing outward from $\R^2_-$.

			For the analysis of the traces in $\Hdiv$ and $\Hcurl$ we use the fact that for any $n\in \N$ the trace operator
			$$T_0:H^1(\R^2,\C^n)\to H^{1/2}(\Gamma,\C^n) \quad (\text{such that  } T_0f=f|_\Gamma \text{ for } f\in C^1(\R^2,\C^n)) $$
			is continuous and surjective, see Thm. 7.39 and Sec. 7.67 in \cite{AF2003}.

			For $u\in \Hdiv(\R^2_-)$ we define a trace of the normal component via
			\beq\label{E:Tr-normal-H}
			T^n_-: \Hdiv(\R^2_-) \to H^{-1/2}(\Gamma),  T^n_- u[\varphi]:= \int_{\R^2_-} u \cdot \nabla \tilde{\varphi} \dd x + \int_{\R^2_-} \tilde{\varphi} \nabla \cdot u  \dd x \qquad \forall \varphi \in H^{1/2}(\Gamma,\C),
			\eeq
			where $\tilde{\varphi} \in H^1(\R^2,\C)$ is such that $T_0 \tilde{\varphi}= \varphi$. We denote here the dual space of $H^{1/2}(\Gamma,\C)$ by $H^{-1/2}(\Gamma)$.

			Due to the surjectivity of $T_0:H^1(\R^2,\C)\to H^{1/2}(\Gamma,\C)$ we know that $\tilde{\varphi}$ exists for any $\varphi$. Next, we show that $T^n_- u[\varphi]$  is independent of the choice of $\tilde{\varphi}$ as long as $T_0 \tilde{\varphi}= \varphi$. Let $T_0 \tilde{\varphi}_1= T_0 \tilde{\varphi}_2=\varphi$ for some $\tilde{\varphi}_{1,2}\in H^1(\R^2,\C)$ and define
			$$\tau^{n,j}_-:= \int_{\R^2_-} u \cdot \nabla \tilde{\varphi}_j \dd x + \int_{\R^2_-} \tilde{\varphi}_j \nabla \cdot u  \dd x, \quad j=1,2.$$
			Indeed, because $\tilde{\varphi}_1-\tilde{\varphi}_2\in H^1_0(\R^2_-,\C)$, we get
			$$\tau^{n,1}_- - \tau^{n,2}_-= \int_{\R^2_-} u \cdot \nabla (\tilde{\varphi}_1-\tilde{\varphi}_2) \dd x + \int_{\R^2_-} (\tilde{\varphi}_1-\tilde{\varphi}_2) \nabla \cdot u  \dd x=0$$
			by the definition of the weak divergence.

			For $u\in C^1(\overline{\R^2_-},\C^3)$ we have $T^n_- u=\nu\cdot u|_\Gamma$ because the divergence theorem implies
			$$ \int_{\R^2_-} u \cdot \nabla \Psi \dd x + \int_{\R^2_-} \Psi\nabla \cdot u  \dd x  = \int_\Gamma \Psi u \cdot \nu \dd S(x) \qquad \forall \Psi \in C^\infty_c(\R^2,\C).$$
			Analogously we define
			\beq\label{E:Tr-normal+H}
			T^n_+: \Hdiv(\R^2_+) \to  H^{-1/2}(\Gamma),  T^n_+ u[\varphi]:= -\int_{\R^2_+} u \cdot \nabla \tilde{\varphi} \dd x - \int_{\R^2_+} \tilde{\varphi} \nabla \cdot u  \dd x \qquad \forall \varphi \in H^{1/2}(\Gamma,\C),
			\eeq
			where $\tilde{\varphi} \in H^1(\R^2,\C)$ is such that $T_0 \tilde{\varphi}= \varphi$. For $u\in C^1(\overline{\R^2_+},\C^3)$ we have again $T^n_+ u=\nu\cdot u|_\Gamma$. Note that the unit outward normal to $\R^2_+$ is $-\nu$.

			For $u\in\Hcurl(\R^2_-)$ there is a trace of the tangential component of $u$:
			\beq\label{E:Tr-tang-H}
			T^t_-: \Hcurl(\R^2_-) \to H^{1/2}(\Gamma,\C^3)',  T^t_- u[\Phi]:= \int_{\R^2_-} u \cdot \nabla \times \tilde{\Phi} \dd x - \int_{\R^2_-} \tilde{\Phi}\cdot \nabla \times u  \dd x \qquad \forall \Phi \in  H^{1/2}(\Gamma,\C^3),
			\eeq
			where $\tilde{\Phi} \in H^1(\R^2,\C^3)$ is such that $T_0 \tilde{\Phi} = \Phi$.  We denote here the dual space of $H^{1/2}(\Gamma,\C^3)$ by $H^{1/2}(\Gamma,\C^3)'$. Note that by similar arguments as for $T^n_-$ one can show that $T^t_- u[\Phi]$ is independent of the choice of $\tilde\Phi$ as long as $T_0 \tilde{\Phi} = \Phi$.

			For $u\in C^1(\overline{\R^2_-},\C^3)$ we have $T^t_- u=\nu\times u|_\Gamma$ because integration by parts implies
			$$ \int_{\R^2_-} \Psi\cdot \nabla \times u  \dd x -\int_{\R^2_-} u \cdot \nabla \times \Psi \dd x  = \int_\Gamma \Psi \cdot  (\nu \times u) \dd S(x) \qquad \forall \Psi \in C^\infty_c(\R^2,\C^3).$$

			Analogously, we set
			\beq\label{E:Tr-tang+H}
			T^t_+: \Hcurl(\R^2_+) \to  H^{1/2}(\Gamma,\C^3)',  T^t_+ u[\Phi]:= -\int_{\R^2_+} u \cdot \nabla \times \tilde{\Phi} \dd x + \int_{\R^n_+} \tilde{\Phi} \cdot \nabla \times u  \dd x \qquad \forall \Phi \in H^{1/2}(\Gamma,\C^3),
			\eeq
			where $\tilde{\Phi} \in  H^1(\R^2,\C^3)$ is such that $T_0 \tilde{\Phi} = \Phi$. For $u\in C^1(\overline{\R^2_+},\C^3)$ we have $T^t_+ u=\nu\times u|_\Gamma$.

			\blem\label{L:Hdiv-interf}
			\begin{itemize}
				\item[(i)] If $u\in \Hdiv(\R^2)$, then  $u_+\in \Hdiv(\R^2_+), u_-\in \Hdiv(\R^2_-)$, $T^n_- u_- = T^n_+ u_+$, where $u_\pm :=u |_{\R^2_\pm}$.
				\item[(ii)] If $u_+\in \Hdiv(\R^2_+), u_-\in \Hdiv(\R^2_-), T^n_- u_- = T^n_+ u_+$, then $$u(x):=\begin{cases}
					u_-(x) , x\in \R^2_-,\\
					u_+(x) , x\in \R^2_+\\
				\end{cases}
				$$
				satisfies $u\in \Hdiv(\R^2)$.
			\end{itemize}
			\elem
			\bpf
			To show (i), note that $T^n_-u_-$ and $T^n_+u_+$ are well defined for $u\in \Hdiv(\R^2)$ (since clearly $u_\pm\in \Hdiv(\R^2_\pm)$). Subtracting the equations in \eqref{E:Tr-normal-H} and \eqref{E:Tr-normal+H}, we get
			$$T^n_- u_-[\varphi] -T^n_+ u_+[\varphi] = \int_{\R^2} u \cdot \nabla \tilde{\varphi} \dd x + \int_{\R^2} \tilde{\varphi} \nabla \cdot u \dd x  \qquad \forall \varphi \in H^{1/2}(\Gamma,\C),$$
			where $\tilde{\varphi} \in H^1(\R^2,\C)$ is such that $T_0 \tilde{\varphi} = \varphi$. The right hand side vanishes by the definition of the weak divergence.

			For (ii) we define
			$$
			w(x):=\begin{cases}\nabla \cdot u_-(x), & x\in \R^2_-,\\
				\nabla \cdot u_+(x), & x\in \R^2_+.
			\end{cases}
			$$
			We have $w\in L^2(\R^2,\C)$ and need to show that $w=\nabla \cdot u$. This follows since for any $\psi \in C^\infty_c(\R^2,\C)$ we have
			$$
			\begin{aligned}
				\int_{\R^2}w \psi \dd x &= \int_{\R^2_-} \psi\nabla \cdot u_- \dd x+ \int_{\R^2_+} \psi\nabla\cdot u_+ \dd x = -\int_{\R^2_-}u_- \cdot \nabla \psi \dd x -\int_{\R^2_+}u_+ \cdot \nabla \psi \dd x + T^n_-u_-[\psi|_\Gamma] - T^n_+u_+[\psi|_\Gamma] \\
				&=-\int_{\R^2} u \cdot \nabla \psi \dd x,
			\end{aligned}
			$$
			where we have used \eqref{E:Tr-normal-H}, \eqref{E:Tr-normal+H}, and the trace assumption in (ii).
			\epf

			\blem\label{L:Hcurl-interf}
			\begin{itemize}
				\item[(i)] If $u\in \Hcurl(\R^2)$, then  $u_+\in \Hcurl(\R^2_+), u_-\in \Hcurl(\R^2_-), T^t_- u_- = T^t_+ u_+$, where $u_\pm :=u |_{\R^2_\pm}$.
				\item[(ii)] If $u_+\in \Hcurl(\R^2_+), u_-\in \Hcurl(\R^2_-), T^t_- u_- = T^t_+ u_+$, then $$u(x):=\begin{cases}
					u_-(x) , x\in \R^2_-,\\
					u_+(x) , x\in \R^2_+\\
				\end{cases}
				$$
				satisfies $u\in \Hcurl(\R^2)$.
			\end{itemize}
			\elem
			\bpf
			To show (i), note that $T^t_-u_-$ and $T^t_+u_+$ are well defined for $u\in \Hcurl(\R^2)$ (since clearly $u_\pm\in \Hcurl(\R^2_\pm)$). Subtracting the equations in \eqref{E:Tr-tang-H} and \eqref{E:Tr-tang+H}, we get
			$$T^t_- u_-[\Phi] -T^t_+ u_+[\Phi] = \int_{\R^2} u \cdot \nabla \times \tilde{\Phi} \dd x - \int_{\R^2} \tilde{\Phi} \cdot \nabla \times u \dd x  \qquad \forall \Phi \in H^{1/2}(\Gamma,\C^3),$$
			where $\tilde{\Phi} \in H^1(\R^2,\C^3)$ is such that $T_0 \tilde{\Phi} = \Phi$. The right hand side vanishes by the definition of the weak curl.

			For (ii) we define
			$$
			w(x):=\begin{cases}\nabla \times u_-(x), & x\in \R^2_-,\\
				\nabla \times u_+(x), & x\in \R^2_+.
			\end{cases}
			$$
			We have $w\in L^2(\R^2,\R^3)$ and need to show that $w=\nabla \times u$. This follows since for any $\Psi \in C^\infty_c(\R^2,\R^3)$ we have
			$$
			\begin{aligned}
				\int_{\R^2}w \cdot \Psi \dd x &= \int_{\R^2_-}\Psi \cdot \nabla \times u_- \dd x+ \int_{\R^2_+}\Psi \cdot \nabla \times u_+ \dd x \\
				&= \int_{\R^2_-}u_- \cdot \nabla\times  \Psi \dd x +\int_{\R^2_+}u_+ \cdot \nabla \times \Psi \dd x - T^t_-u_-[\Psi|_\Gamma] + T^t_+u_+[\Psi|_\Gamma] \\
				&=\int_{\R^2} u \cdot \nabla \times \Psi \dd x,
			\end{aligned}
			$$
			where we have used \eqref{E:Tr-tang-H}, \eqref{E:Tr-tang+H}, and the trace assumption in (ii).
			\epf
			Finally, we use Lemma \ref{L:Hdiv-interf} and \ref{L:Hcurl-interf} to prove the desired equivalence in representing $\cD_\omega$.
			\blem\label{L:Domega-equiv}
			We have \begin{align}\label{E:domega2}
				\cD_\omega=\{&E\in L^2(\R^2,\C^3): \ \text{the } L^2\text{-conditions} \ \eqref{eq4}, \eqref{eq6},	\text{the divergence condition} \ \eqref{E:div-cond-2D},\\ \nonumber
				&\text{ and the interface conditions } \eqref{E:IFC} \text{ hold}\}.
			\end{align}
			\elem

			\bpf
			Let $u\in \cD_\omega$. Because $\perm u \in \Hdiv(\R^2)$ and $u,\nabla\times u \in \Hcurl(\R^2)$, Lemmas \ref{L:Hdiv-interf} and \ref{L:Hcurl-interf} (part (i)) imply the interface conditions \eqref{E:IFC}. The conditions  $u_\pm, \nabla\times u_\pm, \nabla\times \nabla\times u_\pm \in L^2(\R^2,\C^3)$  obviously hold and $\perm_\pm\nabla\cdot u_\pm=0$ can be obtained from the assumption $\int_{\R^2} \perm u\cdot \nabla \overline{\varphi}\dd x =0$ for all $\varphi \in C^\infty_c(\R^2,\C)$ by the choice $\varphi \in C^\infty_c(\R^2_\pm,\C)$.

			Vice versa, let $u$ lie in the set on the right hand side of \eqref{E:domega2}. The distributional divergence condition $\nabla\cdot (\perm u)=0$ follows from
			$$
			\begin{aligned}
				\int_{\R^2} \perm u\cdot \nabla \varphi \dd x &=  \sum_\pm\int_{\R^2_\pm} \perm u_\pm\cdot \nabla \varphi \dd x \\
				&= -\sum_\pm\int_{\R^2_\pm} \nabla\cdot (\perm u_\pm) \nabla \varphi \dd x + T^n_-u[\varphi|_\Gamma]-  T^n_+u[\varphi|_\Gamma]\\
				&=0 \quad \forall \varphi \in C^\infty_c(\R^2,\C).
			\end{aligned}
			$$
			Finally, $u\in L^2(\R^2,\C^3)$ is obvious and $\nabla\times u, \nabla\times \nabla\times u \in L^2(\R^2,\C^3)$ follows from  Lemmas \ref{L:Hdiv-interf} and \ref{L:Hcurl-interf} (part (ii)) applied to $u$ and $\nabla \times u$.
			\epf
			%----------------------------------------------------------

			\section*{Acknowledgements}
			The authors would like to thank Guido Burkard (Konstanz), Maxence Cassier (Fresnel Institute, Marseille), Kirill Cherednichenko and Alexander Kiselev (both Bath University), David Edmunds (Sussex), Christophe Hazard and Patrick Joly (both ENSTA Paris), Carsten Rockstuhl (KIT, Karlsruhe), Marco Marletta, Karl Michael Schmidt (both Cardiff University) and in particular Kurt Busch (HU Berlin) for helpful discussions. Last but not least, we are profoundly grateful to the anonymous referee whose insightful comments and questions helped improve the paper considerably.

			%%%%%%%%%%%%%%%%%%%%%%%%%%%%%%%%%%%%%%%%%%%%%%%%%%%%%%%%%%%%%%%%%%%%%%%%%%%%%%%%%%%%%%%%%%%
			% Bibliography
			%%%%%%%%%%%%%%%%%%%%%%%%%%%%%%%%%%%%%%%%%%%%%%%%%%%%%%%%%%%%%%%%%%%%%%%%%%%%%%%%%%%%%%%%%%%
			\bibliographystyle{plain}
			\bibliography{bibliography}

\begin{thebibliography}{10}

\bibitem{AF2003}
R.~A. Adams and J.~J.~F. Fournier.
\newblock {\em Sobolev Spaces}.
\newblock ISSN. Elsevier Science, 2003.

\bibitem{ABMW19}
G.~S. Alberti, M.~Brown, M.~Marletta, and I.~Wood.
\newblock Essential spectrum for {M}axwell's equations.
\newblock {\em Ann. Henri Poincar\'{e}}, 20(5):1471--1499, 2019.

\bibitem{ARYZ16}
H.~Ammari, M.~Ruiz, S.~Yu, and H.~Zhang.
\newblock {Mathematical analysis of plasmonic resonances for nanoparticles: The
  full Maxwell equations}.
\newblock {\em Journal of Differential Equations}, 261(6):3615--3669, 2016.

\bibitem{ACL2018}
F.~Assous, P.~Ciarlet, and S.~Labrunie.
\newblock {\em Mathematical Foundations of Computational Electromagnetism}.
\newblock Applied Mathematical Sciences. Springer, Cham, 2018.

\bibitem{Busch}
K.~Busch.
\newblock Humboldt University, Berlin, personal communication, 7.7.2024.

\bibitem{CHJ2017}
M.~Cassier, Ch. Hazard, and P.~Joly.
\newblock {Spectral theory for Maxwell's equations at the interface of a
  metamaterial. Part I: Generalized Fourier transform}.
\newblock {\em Communications in Partial Differential Equations},
  42(11):1707--1748, 2017.

\bibitem{CHJ2021}
M.~Cassier, Ch. Hazard, and P.~Joly.
\newblock {Spectral theory for Maxwell's equations at the interface of a
  metamaterial. Part II: Limiting absorption, limiting amplitude principles and
  interface resonance}.
\newblock {\em Communications in Partial Differential Equations},
  47(6):1217--1295, 2022.

\bibitem{CJK2017}
M.~Cassier, P.~Joly, and M.~Kachanovska.
\newblock {Mathematical models for dispersive electromagnetic waves: An
  overview}.
\newblock {\em Computers and Mathematics with Applications}, 74(11):2792--2830,
  2017.
\newblock Proceedings of the International Conference on Computational
  Mathematics and Inverse Problems, On occasion of the 60th birthday of Prof.
  Peter Monk.

\bibitem{CJM2023}
M.~Cassier, P.~Joly, and L.~A.~R. Mart{\'i}nez.
\newblock {Long-time behaviour of the solution of Maxwell's equations in
  dissipative generalized Lorentz materials (I): a frequency-dependent Lyapunov
  function approach}.
\newblock {\em Zeitschrift f{\"u}r angewandte Mathematik und Physik},
  74(3):115, May 2023.

\bibitem{DH2024}
T.~Dohnal and R.~He.
\newblock Bifurcation and asymptotics of cubically nonlinear transverse
  magnetic surface plasmon polaritons.
\newblock {\em Journal of Mathematical Analysis and Applications},
  538(2):128422, 2024.

\bibitem{DR2021}
T.~Dohnal and G.~Romani.
\newblock Eigenvalue bifurcation in doubly nonlinear problems with an
  application to surface plasmon polaritons.
\newblock {\em Nonlinear Differential Equations and Applications NoDEA},
  28(1):1--30, 2021.

\bibitem{DR2022}
T.~Dohnal and G.~Romani.
\newblock Correction to: Eigenvalue bifurcation in doubly nonlinear problems
  with an application to surface plasmon polaritons.
\newblock {\em Nonlinear Differential Equations and Applications NoDEA},
  30(1):9, Nov 2022.

\bibitem{EE18}
D.~E. Edmunds and W.~D. Evans.
\newblock {\em Spectral theory and differential operators}.
\newblock Oxford Mathematical Monographs. Oxford University Press, Oxford,
  2018.
\newblock Second edition of [ MR0929030].

\bibitem{Engstr21}
Ch. Engstr\"{o}m.
\newblock Spectra of {G}urtin-{P}ipkin type of integro-differential equations
  and applications to waves in graded viscoelastic structures.
\newblock {\em J. Math. Anal. Appl.}, 499(2):125063, 14, 2021.

\bibitem{ET18}
Ch. Engstr\"{o}m and A.~Torshage.
\newblock Spectral properties of conservative, dispersive, and absorptive
  photonic crystals.
\newblock {\em GAMM-Mitteilungen}, 41(3):e201800009, 2018.

\bibitem{feynman2}
R.~P. Feynman, R.~B. Leighton, and M.~Sands.
\newblock {\em The Feynman Lectures on Physics, Vol. II: Mainly
  Electromagnetism and Matter}.
\newblock Feynman Lectures on Physics. California Institute of Technology,
  1964.

\bibitem{Grieser14}
D.~Grieser.
\newblock The plasmonic eigenvalue problem.
\newblock {\em Rev. Math. Phys.}, 26(3):1450005, 26, 2014.

\bibitem{HL2007}
Ch. Hazard and F.~Loret.
\newblock Generalized eigenfunction expansions for conservative scattering
  problems with an application to water waves.
\newblock {\em Proceedings of the Royal Society of Edinburgh: Section A
  Mathematics}, 137(5):995--1035, 2007.

\bibitem{HP20}
Ch. Hazard and S.~Paolantoni.
\newblock Spectral analysis of polygonal cavities containing a negative-index
  material.
\newblock {\em Annales Henri Lebesgue}, 3:1161--1193, 2020.

\bibitem{HL07}
D.~Hundertmark and Y.-R. Lee.
\newblock Exponential decay of eigenfunctions and generalized eigenfunctions of
  a non-self-adjoint matrix {S}chr\"{o}dinger operator related to {NLS}.
\newblock {\em Bull. Lond. Math. Soc.}, 39(5):709--720, 2007.

\bibitem{Kato}
T.~Kato.
\newblock {\em Perturbation theory for linear operators}.
\newblock Classics in Mathematics. Springer-Verlag, Berlin, 1995.
\newblock Reprint of the 1980 ed.

\bibitem{Lassas98}
M.~Lassas.
\newblock The essential spectrum of the nonself-adjoint {M}axwell operator in a
  bounded domain.
\newblock {\em J. Math. Anal. Appl.}, 224(2):201--217, 1998.

\bibitem{Markus1988}
A.~S. Markus.
\newblock {\em Introduction to the spectral theory of polynomial operator
  pencils}, volume~71 of {\em Translations of Mathematical Monographs}.
\newblock American Mathematical Society, Providence, RI, 1988.
\newblock Translated from the Russian by H. H. McFaden, Translation edited by
  Ben Silver, With an appendix by M. V. Keldysh.

\bibitem{MT13}
M.~Marletta and Ch. Tretter.
\newblock Spectral bounds and basis results for non-self-adjoint pencils, with
  an application to {H}agen-{P}oiseuille flow.
\newblock {\em J. Funct. Anal.}, 264(9):2136--2176, 2013.

\bibitem{MoellerPivo}
M.~M{\"o}ller and V.~Pivovarchik.
\newblock {\em Spectral Theory of Operator Pencils, Hermite-Biehler Functions,
  and Their Applications}.
\newblock Operator theory : advances and applications. Springer International
  Publishing :Imprint: Birkh{\"a}user, 2015.

\bibitem{Pitarke_2007}
J.~M. Pitarke, V.~M. Silkin, E.~V. Chulkov, and P.~M. Echenique.
\newblock Theory of surface plasmons and surface-plasmon polaritons.
\newblock {\em Reports on Progress in Physics}, 70(1):1--87, 2007.

\bibitem{PBK2011}
L.~J. Prokopeva, J.~D. Borneman, and A.~V. Kildishev.
\newblock Optical dispersion models for time-domain modeling of
  metal-dielectric nanostructures.
\newblock {\em IEEE Transactions on Magnetics}, 47(5):1150--1153, 2011.

\bibitem{Shkaliko96}
A.~A. Shkalikov.
\newblock Operator pencils arising in elasticity and hydrodynamics: The
  instability index formula.
\newblock In I.~Gohberg, P.~Lancaster, and P.~N. Shivakumar, editors, {\em
  Recent Developments in Operator Theory and Its Applications}, pages 358--385,
  Basel, 1996. Birkh{\"a}user Basel.

\bibitem{ST96}
A.~A. Shkalikov and Ch. Tretter.
\newblock Spectral analysis for linear pencils {$N-\lambda P$} of ordinary
  differential operators.
\newblock {\em Math. Nachr.}, 179:275--305, 1996.

\bibitem{Zayats05}
A.~V. Zayats, I.~I. Smolyaninov, and A.~A. Maradudin.
\newblock Nano-optics of surface plasmon polaritons.
\newblock {\em Physics Reports}, 408(3):131--314, 2005.

\end{thebibliography}

			\noindent
			\textbf{Data availability statement:} Data sharing not applicable to this article as no datasets were generated or analysed during the current study.
		\end{document}